\documentclass[preprint]{elsarticle}
\usepackage{amsthm}
\usepackage{thmtools}
\usepackage[colorlinks=true]{hyperref}
\usepackage{amsmath}
\usepackage{amsxtra}
\usepackage{amssymb}
\usepackage{lineno}
\usepackage{turnstile}
\usepackage{verbatim}
\usepackage{enumitem}
\usepackage[retainorgcmds]{IEEEtrantools}
\usepackage{undertilde}
\usepackage[all]{xy}
\usepackage{mathrsfs}
\usepackage{etoolbox}
\newbool{standardmode}
\setbool{standardmode}{true}
\newbool{refereecomments}
\setbool{refereecomments}{false}
  
\declaretheorem[name=Definition,style=definition,qed=$\dashv$,numberwithin=section]{dfn}

\declaretheorem[name=Theorem,style=plain,sibling=dfn]{tm}
\declaretheorem[name=Theorem,style=plain,numbered=no]{tm*}
\declaretheorem[name=Lemma,style=plain,sibling=dfn]{lem}
\declaretheorem[name=Proposition,style=plain,sibling=dfn]{prop}
\declaretheorem[name=Corollary,style=plain,sibling=dfn]{cor}
\declaretheorem[name=Remark,style=definition,sibling=dfn]{rem}
\declaretheorem[name=Remark,style=definition,numbered=no]{rem*}
\declaretheorem[name=Claim,style=plain]{clm}
\declaretheorem[name=Claim,style=plain,numbered=no]{clm*}
\declaretheorem[name=Sublaim,style=plain,numbered=no]{sclm*}
\declaretheorem[name=Case,style=definition]{case}

\declaretheorem[name=Fact,style=plain,sibling=dfn]{fact}

\declaretheorem[name=Notation,style=definition,sibling=dfn]{ntn}
\declaretheorem[name=Convention,style=definition,sibling=dfn]{conv}
\ifbool{standardmode}{
\declaretheorem[name=Definition,style=definition,qed=$\dashv$,sibling=dfn]{newdfn}
\declaretheorem[name=Remark,style=definition,sibling=newdfn]{newrem}
\declaretheorem[name=Lemma,style=definition,sibling=newdfn]{newlem}
\declaretheorem[name=Theorem,style=plain,sibling=newdfn]{newtm}
\declaretheorem[name=Proposition,style=plain,sibling=newdfn]{newprop}
\declaretheorem[name=Corollary,style=definition,sibling=newdfn]{newcor}}
{\declaretheorem[name=NewDefinition,style=definition,qed=$\dashv$,numberwithin=section]{newdfn}
\declaretheorem[name=NewRemark,style=definition,sibling=newdfn]{newrem}
\declaretheorem[name=NewLemma,style=definition,sibling=newdfn]{newlem}
\declaretheorem[name=NewTheorem,style=plain,sibling=newdfn]{newtm}
\declaretheorem[name=NewProposition,style=plain,sibling=newdfn]{newprop}
\declaretheorem[name=OldDefinition,style=definition,qed=$\dashv$,sibling=dfn]{olddfn}

\declaretheorem[name=OldLemma,style=definition,sibling=dfn]{oldlem}
\declaretheorem[name=NewCorollary,style=definition,sibling=newdfn]{newcor}}

\newcommand{\Troot}{\mathrm{root}}

\newcommand{\measdom}{\mathrm{meas}}
\newcommand{\spaceord}{\mathrm{space}}
\newcommand{\trcoll}{\mathrm{trcoll}}
\newcommand{\trclos}{\mathrm{trclos}}
\newcommand{\trcollt}{\mathrm{tc}}
\newcommand{\eqdef}{=_{\mathrm{def}}}

\newcommand{\shortimplies}{\Rightarrow}

\newcommand{\iso}{\cong}
\newcommand{\Ff}{\mathcal{F}}
\newcommand{\CC}{\mathbb C}

\newcommand{\RR}{\mathbb R}

\newcommand{\BB}{\mathbb B}
\newcommand{\sub}{\subseteq}
\newcommand{\cross}{\times}
\newcommand{\all}{\forall}
\newcommand{\ex}{\exists}

\newcommand{\inter}{\cap}

\newcommand{\om}{\omega}
\newcommand{\pow}{\mathcal{P}}
\newcommand{\OR}{\mathrm{OR}}

\newcommand{\Hull}{\mathrm{Hull}}

\newcommand{\cut}{\backslash}
\newcommand{\N}{N}

\newcommand{\Tt}{\mathcal{T}}
\newcommand{\Ss}{\mathcal{S}}
\newcommand{\Uu}{\mathcal{U}}
\newcommand{\Vv}{\mathcal{V}}
\newcommand{\Ww}{\mathcal{W}}

\newcommand{\Wwbar}{{\bar{\Ww}}}
\newcommand{\Xxbar}{{\bar{\Xx}}}
\newcommand{\rg}{\mathrm{rg}}
\newcommand{\dom}{\mathrm{dom}}

\newcommand{\ins}{\trianglelefteq}
\newcommand{\nins}{\ntrianglelefteq}
\newcommand{\pins}{\triangleleft}

\newcommand{\crit}{\mathrm{cr}}

\newcommand{\union}{\cup}
\newcommand{\rest}{\!\upharpoonright\!}
\newcommand{\com}{\circ}

\newcommand{\lh}{\mathrm{lh}}
\newcommand{\Ult}{\mathrm{Ult}}
\newcommand{\Ebar}{{\bar{E}}}

\newcommand{\sats}{\models}
\newcommand{\elem}{\preccurlyeq}
\newcommand{\J}{\mathcal{J}}

\newcommand{\PS}{\mathsf{PS}}

\newcommand{\AC}{\mathsf{AC}}

\newcommand{\HOD}{\mathrm{HOD}}

\newcommand{\ZFC}{\mathsf{ZFC}}
\newcommand{\ZF}{\mathsf{ZF}}

\newcommand{\ZFmin}{\mathsf{ZF^{-}}}

\newcommand{\Coll}{\mathrm{Col}}
\newcommand{\es}{\mathbb{E}}

\newcommand{\qbar}{\bar{q}}

\newcommand{\fbar}{\bar{f}}

\newcommand{\zetabar}{{\bar{\zeta}}}

\newcommand{\kappabar}{{\bar{\kappa}}}
\newcommand{\Pbar}{\bar{P}}

\newcommand{\eps}{\varepsilon}

\newcommand{\Ubar}{{\bar{U}}}
\newcommand{\jbar}{{\bar{j}}}

\newcommand{\core}{\mathfrak{C}}
\newcommand{\her}{\mathcal{H}}
\newcommand{\pred}{\mathrm{pred}}
\newcommand{\tc}{\mathrm{tc}}

\newcommand{\dam}{\mathrm{dam}}

\newcommand{\un}{\union}

\newcommand{\id}{\mathrm{id}}

\newcommand{\sq}{\mathrm{sq}}
\newcommand{\nth}{{\textrm{th}}}
\newcommand{\conc}{\,\,\widehat{\ }}

\newcommand{\bfSigma}{\utilde{\Sigma}}

\newcommand{\rSigma}{\mathrm{r}\Sigma}

\newcommand{\rPi}{\mathrm{r}\Pi}
\newcommand{\rDelta}{\mathrm{r}\Delta}

\newcommand{\spt}{\mathrm{spt}}

\DeclareMathOperator{\Th}{Th}
\DeclareMathOperator{\pTh}{pTh}
\DeclareMathOperator{\card}{card}
\DeclareMathOperator{\cof}{cof}
\DeclareMathOperator{\wfp}{wfp}

\newcommand{\Two}{\mathrm{II}}

\newcommand{\pbar}{\bar{p}}

\newcommand{\xvec}{\vec{x}}
\newcommand{\Dd}{\mathcal{D}}

\newcommand{\gammavec}{\vec{\gamma}}

\newcommand{\bfrSigma}{\utilde{\rSigma}}
\newcommand{\bfrDelta}{\utilde{\rDelta}}

\newcommand{\psub}{\subsetneq}

\newcommand{\Yy}{\mathcal{Y}}

\newcommand{\Xx}{\mathcal{X}}

\newcommand{\cHull}{\mathrm{cHull}}
\newcommand{\minterm}{\mathrm{m}\tau}

\newcommand{\unsq}{\mathrm{unsq}}
\newcommand{\Vvbar}{\bar{\Vv}}
\newcommand{\ttt}{\mathfrak{t}}
\newcommand{\dd}{\mathsf{dd}}
\newcommand{\lhfunc}{\mathfrak{lh}}

\newcommand{\origlast}{\zeta}
\newcommand{\adjlast}{{\origlast'}}

\newcommand{\qhat}{\hat{q}}
\newcommand{\cc}{\mathfrak{c}}
\newcommand{\nt}{\mathrm{nt}}

\newcommand{\M}{\mathsf{M}}
\newcommand{\fin}{M_\infty}

\newcommand{\Mbar}{\bar{\M}}
\newcommand{\pP}{\mathfrak{P}}
\newcommand{\ot}{\mathrm{ot}}

\newcommand{\lpole}{\left\lfloor}
\newcommand{\rpole}{\right\rfloor}
\newcommand{\lrgcrd}{\mathrm{lgcd}}
\newcommand{\iotabar}{\bar{\iota}}
\newcommand{\Hh}{\mathcal{H}}
\newcommand{\univ}[1]{\lpole #1\rpole}
\renewcommand{\dd}{\mathfrak{d}}
\newcommand{\tu}{\textup}

\newcommand{\R}{\mathcal{R}}

\newcommand{\lex}{{\text{lex}}}
\renewcommand{\dam}{{\mathrm{Da}}}

\newcommand{\Ein}{e}
\newcommand{\equality}{{\approx}}

\renewcommand{\qedsymbol}{$\Box$}

\newcommand{\shiftsupport}{\mathsf{Shift}}
\newcommand{\critical}{\mathsf{Crit}}

\newcommand{\starlevel}{\mathsf{Star}}

\newcommand{\definition}{\mathsf{Def}}
\renewcommand{\cut}{\backslash}
\newcommand{\meas}{\mathsf{MeasDef}}
\newcommand{\satisfaction}{\mathsf{Sat}}
\newcommand{\picode}{\varsigma}
\newcommand{\picodebar}{\bar{\picode}}
\newcommand{\adjust}{\mathrm{adj}}
\renewcommand{\dd}{\mathfrak{d}}

\newcommand{\Qfin}{Z}
\newcommand{\Rstar}{Q}
\renewcommand{\R}{R}
\renewcommand{\Ss}{S}
\newcommand{\pr}{\upsilon}
\newcommand{\prbar}{\bar{\pr}}
\newcommand{\prprime}{\epsilon}

\newcommand{\betavec}{\vec{\beta}}
\newcommand{\dropset}{\mathscr{D}}
\newcommand{\lgcd}{\mathrm{lgcd}}
\begin{document}
\title{The definability of $\es$ in self-iterable mice}
\author[fs]{Farmer Schlutzenberg}
\ead{farmer.schlutzenberg@gmail.com}
\fntext[thanks]{Thanks to John Steel for discussions on the topic of this paper.}
\begin{abstract}
Let $M$ be a fine structural mouse and let $F\in M$ be such that $M\sats$``$F$
is a total extender'' and $(M||\lh(F),F)$ is a premouse. We show that it
follows that $F\in\es^M$, where $\es^M$ is the extender sequence of $M$. We also
prove generalizations of this fact.

Let $M$ be a premouse with no largest cardinal
and let $\Sigma$ be a sufficient iteration strategy for $M$.
We prove that if $M$ knows enough of $\Sigma\rest M$
then $\es^M$ is definable over the universe $\univ{M}$ of $M$,
so if also $\univ{M}\sats\ZFC$ then $\univ{M}\sats$``$V=\HOD$''.
We show that this result applies in particular
to $M=M_\nt|\lambda$, where $M_\nt$ is the least non-tame mouse
and $\lambda$ is any limit cardinal of $M_\nt$.

We also show that there is no iterable bicephalus $(N,E,F)$ for which $E$ is
type $2$ and $F$ is type $1$ or $3$. As a corollary, we deduce a uniqueness
property for maximal $L[\es]$ constructions computed in iterable background
universes.
\end{abstract}
\maketitle

\ifbool{refereecomments}{
\pagebreak

\input{comments_to_referee}

\pagebreak
}{}

\section{Introduction}\label{sec:intro}
Kunen \cite{kunen} showed that if $V = L[U]$ where $U$ is a
normal measure, then $U$ is the unique normal measure. Mitchell \cite{mitchell,
mitchellrevisit} constructed inner models with sequences of measurables
and proved analogous results regarding these. Steel \cite[\S8]{cmip}
and
Schimmerling/Steel \cite[\S2]{maxcore} also proved related results,
asserting roughly that if $P$ is an iterate of $K$
and $E$ is an extender which is sufficiently certified
and $(P||\alpha,E)$ is a premouse, then $E$ is on the extender sequence of $P$.

In this paper we will consider variants of these results pertaining to
fine-structural mice $M$ below a superstrong cardinal. Let $F\in M$ be such that
$M\sats$``$F$ is a
total (and possibly wellfounded) extender''. We are
interested in questions such as:
\begin{enumerate}
 \item Is $F$ in the extender sequence $\es_+^M$ of $M$?
 \item Is $F$ the extender of an iteration map on
$M$?\footnote{John Steel suggested this question, under the added
assumption that $V=\univ{M}$.}\footnote{Where the iteration map is via a fine
structural
tree on $M$, as in \cite{fsit}.}
\end{enumerate}

We will show in Theorems
\ref{thm:easy_coh}-\ref{thm:cohering}, the central results of the paper, that
under further reasonable (but significant) assumptions, the answer (to at
least one of the above questions) is in
the affirmative, and thus, $\es^M$ is in some sense maximal with regard to
extenders. The statements of these theorems are analogous to the results
of Steel and Schimmerling/Steel mentioned above. The simplest case is the following theorem. (See 
\S\ref{ssec:ntn} for a review of terminology and notation. Throughout the paper, we use the 
definition of \emph{premouse} from \cite[\S2]{outline}. These models do not have extenders of 
superstrong type on their extender sequence. For the
central results, this is the only anti-large cardinal hypothesis we require.
For premice with superstrongs things are somewhat different,
as discussed in \ref{rem:superstrong}.)

\begin{tm*}[\ref{thm:easy_coh}]
Let $N$ be a $(0,\om_1+1)$-iterable premouse. Let $E\in N$ be such that
$(N||\lh(E),E)$ is a premouse and $E$ is total over $N$.
Then $E\in\es^N$.
\end{tm*}

Steel first proved the following version of \ref{thm:easy_coh};
see \cite{mim} for the proof.

\begin{tm*}[Steel]
Let $N$ be a $(0,\om_1+1)$-iterable premouse with no largest cardinal.
Let $E\in N$ be such that $N\sats$``$E$ is a total wellfounded extender
such that $\nu=\nu(E)$ is regular and $N||\nu=\Ult(N,E)||\nu$''.
Then $E\in\es^N$.\end{tm*}

Note that Steel's result requires slightly less coherence
of $E$ with respect to $\es^N$ (through $\nu(E)$), than does \ref{thm:easy_coh} (through $\lh(E)$).
The ``no largest cardinal'' is more than enough for Steel's result, but
some such closure was used. On the other hand, in \ref{thm:easy_coh},
$\nu(E)$ need not be a cardinal in $N$,
and we do not explicitly demand that $\Ult(N,E)$
be fully wellfounded; it is only required that $\Ult(N,E)$ be wellfounded through
$\lh(E)+1$ (this is part of the requirement that $(N||\lh(E),E)$ be a premouse).

Theorems \ref{thm:finite_coh}--\ref{thm:cohering} are modifications of
\ref{thm:easy_coh} in which $E$ fits on the sequence of an iterate $P$ of $N$ (that is, $(P||\lh(E),E)$ is a premouse).
In \ref{thm:cohering}, we also relax the requirement that $E\in N$,
demanding instead that $E$ be appropriately definable over $N$ (and that the iteration tree leading from $N$ to $P$
is likewise definable).
In order to state and prove this result precisely we must develop
a coding of iteration trees over premice, which takes some work.

Because these results give criteria for an extender $E$ to be in $\es_+^N$, or
to be in $\es_+^P$ for an iterate $P$ of $N$, one can try to use them to show
that a mouse can recognize its own extender sequence, and we achieve this in certain situations.
Given a premouse $M$, we write $\univ{M}$ for the universe of $M$.
Recall that $M_n$ is the canonical proper
class inner model with $n$ Woodin cardinals.
Steel showed that for $n\leq\om$, 
$K^{M_n}=M_n$,\footnote{For $n=1$, Steel argued as follows. In $M_1$, let $\mu$ be 
measurable and $\delta$ be Woodin. Let $K_\mu$ be Steel's core model of height 
$\mu$,
as computed in $V_\delta^{M_1}$ (see \cite{cmip}). Steel showed that $K_\mu=M_1|\mu$.
This yields a definition of $\es^{M_1}$ in $\univ{M_1}$.
For $n>1$, he defines $K|\delta_i$ inductively on $i$,  where $\delta_i$ is Woodin and $\delta_{-1}=0$,
defining
$K|\delta_i=K(K|\delta_{i-1})$, proceeding much as for $M_1$ in the interval $(\delta_{i-1},\delta_i)$. In $M_n$, $K|\delta_i=M_n|\delta_i$
is above-$\delta_{i-1}$, $\delta_{i}$-iterable.} and therefore
$\es^{M_n}$ is definable over
$\univ{M_n}$ and $\univ{M_n}\sats$``$V=\HOD$''.
It seems to be unknown precisely how far Steel's result generalizes.
We will show anyway in \ref{prop:it_error} that
mice with a measurable limit of Woodins typically have a significant
failure of self-iterability, one which seems to present a difficulty
in generalizing Steel's result to this level.
Nonetheless, using results in \S3,
we will prove Theorem \ref{thm:tamesi},
a consequence of which  is the following
(we write $M_\nt$ for the minimal
non-tame mouse; see \S4):

\begin{tm*}
Suppose $M_\nt$ exists and is $(\om,\om_1,\om_1+1)$-iterable.
Let $\lambda$ be a limit cardinal of $M_\nt$ and $N=M_\nt|\lambda$.
Then $\es^N$ is definable over $\univ{N}$.
\end{tm*}

Note that $M_\nt$ is well beyond a measurable limit of Woodins.
Analogous methods work for many tame mice.
The method depends on
Theorems \ref{thm:es_def} and \ref{thm:es_def_2},
a corollary of which is:
\begin{tm*}
Let $Z$ be a premouse satisfying $\ZFC+$``I am $(\om,\OR,\OR)$-iterable''. Then $\es^Z$ is definable 
over $\univ{Z}$ and 
$\univ{Z}\sats$``$V=\HOD$''.
\end{tm*}

The self-iterability hypothesis of this theorem
fails if $Z$ has a Woodin cardinal,
but Theorems \ref{thm:es_def} and \ref{thm:es_def_2} themselves 
are versions suited to premice with Woodin cardinals,
and which give the same conclusion.
The
method breaks down with non-tame mice, because of a more serious lack of self-iterability
(see \ref{prop:nontame}).

Finally, in \S\ref{sec:bicephali}, we adapt Mitchell and Steel's ``Uniqueness of the next extender''
result \cite[9.2]{fsit} to bicephali $B$ of
the form $(\univ{B},E,F)$, with $E$ type 2 and $F$ type 1 or 3.  This has
positive implications regarding the uniqueness of maximal
$L[\es]$-constructions. (Although the question is natural, uniqueness
in this particular case was not required for the arguments in \cite{fsit};
and such bicephali were not considered there.)

The paper proceeds as follows. In \S\ref{ssec:ntn} we review notation and terminology.
In \S\ref{sec:Dodd} we review Dodd fine structure 
for extenders (due to Steel) and prove various related facts. In
\S\ref{sec:cohering} we state and prove the central results
\ref{thm:easy_coh}--\ref{thm:cohering}, making use of most of the work in
\S\ref{sec:Dodd}. However, the proof of \ref{thm:easy_coh}, for example,  only 
depends on the results in \S\ref{sec:Dodd} in the case that $E$ is type 2, so one  can basically  skip 
\S\ref{sec:Dodd} if one is only interested in the case that $E$ is type 1 or 3,
and refer to parts as needed.
In \S\ref{sec:stacking} we consider, for
premice $N$, the
definability of $\es^N$ over $\univ{N}$. This section
uses only Theorems \ref{thm:strategy_coh_proj} and \ref{thm:cohering_simple} from \S\ref{sec:cohering}, and in
particular, not Theorem \ref{thm:cohering}, nor the material on coding of
iteration
trees in \S\ref{sec:cohering}. Finally, \S\ref{sec:bicephali} is on
bicephali. This section can be read independently of
\S\S\ref{sec:cohering},\ref{sec:stacking}, and depends only a little on \S\ref{sec:Dodd}.

The pieces of \S\S\ref{sec:Dodd}--\ref{sec:stacking} due to the author are taken primarily from the author's dissertation \cite{mim}, with various
refinements having been incorporated. There was a significant error in 
a draft of this paper, which is addressed in \ref{rem:it_error} and \ref{prop:it_error}.

\subsection{Conventions and Notation}\label{ssec:ntn}
\textbf{General}: For $\nu\in\OR$, $\her_\nu$ denotes the set of sets
hereditarily of size less than $\nu$. For an extensional $H$, $\trcoll(H)$
denotes the transitive collapse of $H$. Given any set $X$, $\trclos(X)$ denotes the transitive 
closure of $X$.
We use the lexicographic order on $[\OR]^{<\om}$: $a<_\lex b$
iff $a\neq b$ and $\max(a\Delta b)\in b$. We sometimes identify elements of $[\OR]^{<\om}$ with 
strictly descending sequences of ordinals. Let 
$a\in[\OR]^{<\om}$ with $a=\{a_0,\ldots,a_{k-1}\}$
where $a_i>a_{i+1}$ for all $i+1<k$. We write $a\rest j$ for $\{a_0,\ldots,a_{j-1}\}$.
Let either $\sigma,\tau$ be sequences, or $\sigma,\tau\in[\OR]^{<\om}$.
We write $\sigma\ins\tau$ iff there is $k$ such that $\sigma=\tau\rest k$,
and write $\sigma\pins\tau$ iff $\sigma\ins\tau$ and $\sigma\neq\tau$.

\textbf{Premice}: Premice are as in \cite{outline}, except that we officially consider a \emph{premouse}
to be an amenable structure $P=(\J_\alpha^\es,\es,\widetilde{F})$,
where $\widetilde{F}$ is the amenable coding of the active extender $F$
of $P$, as described in \cite[2.9--2.10]{outline}. (Note that this coding is independent from squashing,
and in this paragraph we consider no squashed premice.) We may blur the distinction between $F$ and $\widetilde{F}$;
whenever we write $(\J_\alpha^\es,\es,F)$ we literally mean $P$ as above.
We write $\univ{P}$ for the universe $\J_\alpha^\es$ of $P$, $F^P=F(P)$ for $F$, $\es^P=\es(P)$ for
the extender sequence $\es$ of $P$, excluding $F^P$,
and $\es^P_+=\es_+(P)$ is $\es^P\conc F^P$.
If $P$ has a largest cardinal it is denoted
by $\lrgcrd(P)$. 
We write $Q\ins P$ iff $Q$ is an initial segment of $P$,
that is, either $Q=P$ or $Q=(\J_\alpha^{\es^P},\es^P\rest\alpha,\es^P_\alpha)$
for some limit ordinal $\alpha<\OR^P$.
We write $Q\pins P$ iff $Q\ins P$ and $Q\neq P$.
Let $\alpha\leq\OR^P$ be a limit. Then
$P|\alpha$ denotes the $Q\ins P$ such that $\OR^Q=\alpha$; and
$P||\alpha$ denotes $(\univ{Q},\es^Q,\emptyset)$ where $Q=P|\alpha$.
The notation $\nu(P)=\nu(F^P)$ is discussed below.
\emph{ISC} abbreviates ``initial
segment condition''. Such notation and terminology is used likewise a little more generally, for example with 
respect to segmented-premice (defined in \ref{dfn:segmented}).

\textbf{Extenders}: Our use of the term 
\emph{extender} allows long, non-total measure spaces, and does not require 
full wellfoundedness of corresponding ultrapowers (clarified below).\footnote{The motivation for 
this is as 
follows. Most of the time we will deal with extenders $E$ over some structure $M$, where $E$ is 
likely not over $V$ (it does not measure enough sets).
We want to be able to use the term \emph{extender} to refer to such $E$,
instead of \emph{partial \tu{(}pre-\tu{)}extender}. But for partial extenders, it is difficult to 
give a useful general notion of \emph{completeness} or \emph{wellfoundedness}. Thus, we prefer not 
to make any such demand in general (in our definition of \emph{extender}). Although this 
usage of \emph{extender} diverges from the conventional formal one (being $V$-total and 
 countably complete),
it also seems to be pretty common in informal usage, and is convenient for our present 
purposes. We also need to deal with long extenders in general.} We will only consider 
those extenders which can be considered as being over (the universe of a) premouse. By restricting 
our attention to these we simplify some small considerations. Let 
us clarify. Let $M$ be a premouse. Much as in \cite[Definition 2.1 and 
following 
paragraph]{outline}, we say that $E$ is an
\emph{extender over} $\univ{M}$ iff there is a structure $N$
 ($N$ need not be transitive) and $j:\univ{M}\to N$ which is $\Sigma_0$-elementary 
(in the language of set theory) such that $\kappa=\crit(j)$ exists and $\kappa+1\sub\wfp(N)$, and 
there is $\delta\in(\kappa,\OR^M]$ and $S\sub\wfp(N)$ such that
\[ E=\{(x,a)\mid x\in M\ \&\ \ex\xi<\delta[x\sub[\xi]^{|a|}]\ \&\ a\in[S]^{<\om}\inter j(x)\}. \]
We write $\kappa_E=\crit(E)$ for the critical point of $E$.
The \emph{space} of $E$, denoted $\spaceord(E)$, is $\delta$.
The \emph{measure domain} of $E$ is
\[ \measdom(E)=\bigcup_{(\xi,n)\in\delta\cross\om}\pow([\xi]^n)\inter M.\]
We say that $E$ is \emph{total} iff
$\measdom(E)=\bigcup_{(\xi,n)\in\delta\cross\om}\pow([\xi]^n)$.
The \emph{support} of $E$ is $\spt(E)=S$. 
We say that $E$ is \emph{short} iff $\spaceord(E)=\kappa_E+1$.
(One could require that $S\sub\bigcup_{\xi<\delta}j(\xi)$ and
there be at most one $\xi<\delta$ such that $S\sub j(\xi)$,
as only this part of $E$ is relevant.) If $E$ is short and $S=\lambda\in\OR$,
we write $E_{j,\lambda}=E$, and write $E_j=E_{j,j(\kappa_E)}$.

An \emph{amenable transitive structure} is a structure 
$R=(\univ{R},P_0,\ldots,P_{k-1})$ with transitive universe $\univ{R}$ and finitely many 
predicates $P_i$, such that each $P_i$ is amenable to $\univ{R}$. An amenable transitive structure 
$R=(\univ{R},\es,\ldots)$ is \emph{pm-based}\footnote{\emph{pm} abbreviates 
\emph{premouse}.} iff $(\univ{R},\es,\emptyset)$ is a (passive)
premouse.
Given $F$ and a pm-based $R$, we say that $F$ is an \emph{extender over} $R$ iff
$F$ is an extender over $\univ{R}$.
A \emph{pm-based extender} is an extender over (the universe of) some premouse (equivalently, 
over some pm-based structure).
So pm-based extenders can be be non-short,
and when short, we use the phrase \emph{extender over} (a premouse) where 
\cite{outline} uses \emph{pre-extender over}.
Finally, \emph{in this paper all extenders we consider are pm-based extenders}, so at this point we 
adopt the:
\begin{conv}\label{conv:extender}
 \emph{Extender} abbreviates \emph{pm-based extender}.
\end{conv}
Thus, we always explicitly state assumptions
on the totality of extenders, and on the wellfoundedness of ultrapowers by extenders.

Let $R$ be a premouse and $E$ an extender over $R$. Let $M$ be a pm-based 
structure. Suppose that 
$\her_{\spaceord(E)}^M=\her_{\spaceord(E)}^R$. Then $E$ is an extender over $M$. In fact, 
$N=\Ult_0(M,E)$ 
is formed as usual (including the predicates of $N$)
and likewise the associated ultrapower embedding $j=i^{M}_E$, and $j:M\to N$ is 
cofinal and $\Sigma_0$-elementary with respect to the language of $M$; this is routine by Lo\'{s}' 
Theorem. (But $N$ might be illfounded.)
We abbreviate $i^M_E$ by $i_E$ if $M$ is understood.
Ultrapowers are generally by default at degree $0$ (as above), unless context dictates otherwise.
If $M$ is a premouse, then $i^{M,k}_E$
denotes the degree $k$ ultrapower embedding $M\to\Ult_k(M,E)$, if defined.
Sometimes we might abbreviate this with $i^M_E$ or $i_E$.
Given $a\in[\spt(E)]^{<\om}$ and an $\bfrSigma_k^M$ function $f$,
$[a,f]^{M,k}_E$ denotes the object represented by the pair $(a,f)$ in $\Ult_k(M,E)$;
we may also write $[a,f]^M_E$ if we wish to suppress $k$.
The notation $\Ult_k(E,F)$, for $F$ an extender, and the notation $\com_k$, are
introduced in \ref{dfn:extcomp}.

Let $E$ be an extender over $M$. Note that if $\spt(E)\in\OR$ then for each $\alpha\in\spt(E)$ we have $\alpha=[\{\alpha\},f]^M_E$
where $f=\id$ has sufficiently large domain.
For $\beta\leq\alpha\in\spt(E)$, recall that $\beta$ is a \emph{generator} of $E$
iff $\beta\neq[a,f]^M_E$ for all $f$ and $a\in[\beta]^{<\om}$.
The \emph{natural length} $\nu(E)=\nu_E$ of $E$ is the maximum of $(\kappa_E^+)^M$ and the strict sup of generators of $E$.
For $X\sub\spt(E)$, $E\rest X$ is the sub-extender of $E$ with support $X$ and domain the 
least $\delta\leq\spaceord(E)$ such that either $\delta=\spaceord(E)$ or $i_E(\delta)\geq\sup(X)$.
(Thus, $E\rest\spt(E)$ is equivalent to $E$ in terms of the ultrapowers it produces, but might 
have smaller domain.)
Given an active premouse or related structure $P$, $\lh(F^P)$ denotes
the length of $F^P$, that is, $\OR^P$.
We sometimes write $\lh(F)$ to denote $\lh(F^P)$,
when $P$ is as above and
$F=F^P\rest(\tau\cup t)$ and $F^P$
is generated by $\tau\cup t$ where $\tau\in\OR$ and $t\in[\OR]^{<\om}$
(see \ref{dfn:gen}).

\textbf{Fine structure}: For definability over premice, we basically follow \cite{outline},
using the $\rSigma_n$ hierarchy and $n^\nth$ core $\core_n(P)$ basically as there.
As in \cite{outline}, $\core_0(P)$ denotes the squash $P^\sq$ of $P$ if $P$ is type 3,
and otherwise just denotes $P$ (as mentioned earlier, $P$ is by definition amenable).
If $P$ is not type 3, we also define $P^\sq=P$, so in all cases, $\core_0(P)=P^\sq$.
Also in general, $\core_0(P)^\unsq$ denotes $P$.

Let $n<\om$, let $P$ be an $n$-sound premouse with $\om<\rho_n^P$, and let $X\sub
\core_0(P)$. Let $H$ be (i) the set of points 
in $\core_0(P)$ definable
over $\core_0(P)$ with a generalized $r\Sigma_{n+1}$ term from parameters in
$X\un\{u_n^P\}$ (cf.~\cite[2.8.1]{fsit}). By 
\cite[Proof of Lemma 2.10]{fsit}, $H$ coincides with (ii) the set of points $y\in\core_0(P)$ such 
that for some $\rSigma_{n+1}$ formula $\varphi$ and $\vec{x}\in(X\un\{u_n^P\})^{<\om}$, $y$ is the 
unique $y'\in\core_0(P)$ such that $\core_0(P)\sats\varphi(\vec{x},y')$ (this uses the 
$n$-soundness of $P$ and that $\vec{x}$ can include $u_n^P$). We write 
$\Hull^P_{n+1}(X)$ for the
structure
\[ (H,H\inter\es^{\core_0(P)},H\inter F^{\core_0(P)}),\]
and write $\cHull_{n+1}^P(X)$ for its transitive collapse. We 
may occasionally
identify a type 3 premouse with its squash, so in particular, if $P$ is type
3, then we might write $\cHull_{n+1}^P(X)$ where we really mean
$\cHull_{n+1}^P(X)^{\unsq}$. We define $\Hull_\om^P(X)$ and $\cHull_\om^P(X)$
similarly. For $\delta\in[\rho_{n+1}^P,\OR^P]$, the $\delta$-core of $P$ is
$\cHull_{n+1}^P(\delta\un\{p_{n+1}^P\})$.
Also, $\Th_{n+1}^P(X)$ and $\pTh_{n+1}^P(X)$ respectively denote the generalized
and pure $\rSigma_{k+1}$ theories of $\core_0(P)$ in parameters in
$X\un\{u_n^P\}$. (These theories include only generalized or pure $\rSigma_{k+1}$ formulas,
not negations thereof.) We will make use of the stratification of the pure theories given in 
\cite[Proof of Lemma 2.10]{fsit}. Because $H$ above is determined by the corresponding pure theory, 
we also have a corresponding stratification of $H$; see
\ref{dfn:minterm}.
Given $p\in[\rho_0^P]^{<\om}$, an \emph{$(n+1)$-solidity witness
for $(P,p)$} (or just for $p$) is a theory
\[ \cHull_{n+1}^P(\alpha\cup(p\cut(\alpha+1))),\]
where $\alpha\in p$. A \emph{generalized $(n+1)$-solidity witness
for $(P,p)$} is analogous, but defined as in \cite{zeman}.
A \emph{\tu{(}generalized\tu{)} $(n+1)$-solidity witness for $P$} is 
that for $(P,p_{n+1}^P)$. And $w_{n+1}^P$ denotes the set of
all $(n+1)$-solidity witnesses for $P$.

We take \emph{weak $k$-embedding} $\pi:M\to N$ to be defined as in \cite{copy_con} (this definition is due to Steve Jackson).
That is, the definition is as usual, except that we add the demand that there
be a cofinal set $X\sub\rho_k^M$ such that $\pi$ is $\rSigma_{k+1}$-elementary on $\Hull_{k+1}^M(X)$.
(This ensures that the proof of the Shift Lemma goes through as expected. We do not know
whether one can prove the Shift Lemma for weak $k$-embeddings as defined in \cite{fsit},
when $1\leq k<\om$.)

\textbf{Iteration trees}: Structures such as premice, phalanxes of premice, bicephali, pseudo-premice, etc, we call 
\emph{premouse-related}.
All iteration trees (see \cite{outline}) we consider in this paper will be fine-structural, in that they are much as in 
\cite[\S5]{fsit}, and based on premouse-related structures.
We will not specify 
exactly what we mean by the general term \emph{iteration tree}, but it 
suffices to consider 
$k$-maximal trees (see below) and stacks thereof. Let $\Tt$ be an iteration tree. We write 
$M^{*\Tt}_{\alpha+1}$ for 
the model to
which $E^\Tt_\alpha$ applies after any drop in model, and
$i^{*\Tt}_{\alpha+1,\beta}:M^{*\Tt}_{\alpha+1}\to M^\Tt_\beta$
for the canonical embedding, if it exists. We write
$\kappa^\Tt_\alpha=\crit(E^\Tt_\alpha)$,
$\nu^\Tt_\alpha=\nu(E^\Tt_\alpha)$ and $\lh^\Tt_\alpha=\lh(E^\Tt_\alpha)$. If
$\lh(\Tt)=\theta+1$, then $\fin^\Tt=M^\Tt_\theta$, $b^\Tt=[0,\theta]_\Tt$,
and if there is no drop along
$b^\Tt$ then $i^\Tt=i^\Tt_{0,\theta}$. We say
$\Tt$ is \emph{above} $\rho$ iff $\rho\leq\crit(E^\Tt_\alpha)$ for each
$\alpha+1<\lh(\Tt)$.

For $k\leq\om$, the notion \emph{$k$-maximal iteration tree $\Tt$} (on a $k$-sound premouse) is 
defined in 
\cite[Definition 3.4]{outline}, and
equivalently in \cite[Definition 6.1.2]{fsit}.
The main points are that for all $\alpha+1<\lh(\Tt)$, with $\kappa=\kappa^\Tt_\alpha$, we have: 
(i) $\lh(E^\Tt_\beta)<\lh(E^\Tt_\alpha)$ for all 
$\beta<\alpha$;\footnote{In \cite{fsit} this requirement is included in the definition of 
\emph{iteration tree}, whereas in 
\cite{outline} it is a requirement in the game $\mathcal{G}_k(M,\theta)$. The book \cite{cmip} 
seems to weaken this requirement in its use of \emph{$k$-maximal}.} (ii) 
$\pred^\Tt(\alpha+1)$ is the least $\beta$ such that $\kappa<\nu^\Tt_\beta$; (iii) 
$M^{*\Tt}_{\alpha+1}$ is the largest $N\ins M^\Tt_\beta$ such that $E^\Tt_\alpha$ measures 
$\pow(\kappa)\inter N$; and (iv) $\deg^\Tt(\alpha+1)$ is the largest $n\leq\om$ such that 
$\kappa^\Tt_\alpha<\rho_n(M^{*\Tt}_{\alpha+1})$ and either $[0,\alpha+1]_\Tt$ drops or 
$n\leq k$. We will also extend, in an obvious 
manner, the term \emph{$k$-maximal iteration tree} to trees on premouse-related structures. Any 
non-obvious details in relation to such will hopefully be clear in 
context. We often abbreviate \emph{$k$-maximal iteration tree} by \emph{\tu{(}$k$-\tu{)}maximal 
\tu{(}tree\tu{)}}. Whenever we use the term \emph{maximal} with regard to
an iteration tree, it is in the sense of \emph{$k$-maximal} (for some, or for the relevant, $k$).
The iteration trees we consider will all be either maximal, or stacks thereof (but most
will be maximal).

Let $k\leq\om$, let $M$ be a $k$-sound premouse, and let $\theta\leq\OR$. The notions 
\emph{$(k,\theta)$-iteration strategy} for $M$ and \emph{$(k,\theta)$-iterability} of $M$ are 
defined in \cite[Definition 3.9]{outline}. (In particular, such a strategy $\Sigma$ yields 
$\Tt$-cofinal wellfounded branches exactly for $k$-maximal trees $\Tt$ which are according to 
$\Sigma$ and have limit length $<\theta$.) For \emph{$(k,\alpha,\theta)$-iteration strategy} and 
\emph{$(k,\alpha,\theta)$-iterable} see \cite[Definition 4.4]{outline}.\footnote{
The author claims that all assumptions of $(k,\alpha,\theta)$-iterability
in the paper can actually be weakened and replaced with $(k,\alpha,\theta)^*$-iterability
(see \cite[p.~1202]{cmwmwc}).}
Given $\rho\in\OR$ and $\Sigma$, we say that $\Sigma$ is an ($(k,\theta)$-, etc) \emph{iteration 
strategy for $M$ above $\rho$} iff $\Sigma$
works as an ($(k,\theta)$-, etc) iteration strategy with respect to trees above $\rho$.
We extend this language in the obvious manner to \emph{iteration strategies} and \emph{iterability}
for premouse-related structures.

If $\Tt$ is $k$-maximal then $\Phi(\Tt)$ denotes the phalanx associated to
$\Tt$ (see \cite{cmip}). 
Let $M,N$ be premice and $m,n\leq\om$ such that $M$ is $m$-sound and $N$ is $n$-sound.
Let $\delta\leq\OR^M$ and $\lambda\leq\OR^N$ with
$\delta\leq\lambda$, $M||\delta=N|\delta$ and $\delta$ is a cardinal of $N$.
We write $\pP=((M,m,\delta),(N,n),\lambda)$ for the phalanx on which maximal trees $\Tt$ are formed
with the usual conditions augmented by the following: (i) $M^\Tt_{-1}=M$ and $\deg^\Tt(-1)=m$ and $M^\Tt_0=N$ and $\deg^\Tt(0)=n$, (ii) $E^\Tt_0$ is the first extender of $\Tt$ and $E^\Tt_0\in\es_+^N$
with $\lambda\leq\lh(E^\Tt_0)$, (iii) if $\crit(E^\Tt_\alpha)<\delta$ then $\pred^\Tt(\alpha+1)=-1$.
For the phalanx notation $\Phi(\Tt,\iota,E)$ see \ref{dfn:potentialtree}.

\section{Dodd structure}\label{sec:Dodd}
The proof of Theorems   \ref{thm:easy_coh} and \ref{thm:finite_coh}--\ref{thm:cohering} will 
involve the
analysis of a comparison. In some cases we will need to deal with the possibility that some 
generators of extenders are moved by an iteration map resulting from the comparison. We will need 
to analyze how such generators are moved. This section
gives us the tools for this analysis. Parts of this section are also needed in 
\S\ref{sec:bicephali}.
Somewhat restricted versions\footnote{In particular,
when the extender $E$ in \ref{thm:easy_coh} and \ref{thm:finite_coh}--\ref{thm:cohering} is type 1 or 3.} of the above-mentioned results
basically do not rely on the current section,
except that they refer to Definition \ref{dfn:potentialtree}
and use a simple case of Lemma \ref{lem:extendermaximality}.

We begin by surveying some definitions and notation. These definitions
are taken from, or are slight variants of definitions from, \cite{fsit} and
\cite{combin}. See 
\S\ref{ssec:ntn} for background notation, etc.

\begin{dfn}[Generators; Dodd condensation]\label{dfn:gen} 
Let $E$ be an extender\footnote{Our use of the term \emph{extender} is specified in 
Convention~\ref{conv:extender}.}
 over a premouse $N$.
We say that $E$ is \emph{standard} iff $\spt(E)\in\OR$ and
$\alpha=[\id,\{\alpha\}]^N_E$ for each $\alpha<\spt(E)$.\footnote{Note that this is 
independent of $N$, and if we consider a degree $k$ ultrapower ($k\leq\om$) instead of degree $0$, we also get the same result;
likewise for other notions defined in \ref{dfn:gen}.}
Suppose $E$ is standard and let $\gamma=\spt(E)$ and $\kappa=\kappa_E$.

Let $X\sub\gamma$ and $\alpha<\gamma$. Then
$\alpha$ is \emph{generated by $X$} (with respect to $E$) iff
there are $f\in N$ and $b\in[X]^{<\om}$ such that $\alpha=[b,f]^N_E$.
We say that $X$ \emph{generates} $E$ iff every element of $\gamma$ is 
generated
by $X$. We write $E\approx E\rest X$ iff $X$ generates $E$.
We say that $E$ is \emph{finitely generated} iff there is some finite $X$ generating $E$.
For $t\in[\gamma]^{<\om}$ and $\alpha<\gamma$ we say that $\alpha$ is a
\emph{$t$-generator} iff
$\alpha$ is not generated by $\alpha\un t$.
We say $E$ is \emph{weakly amenable \tu{(}to $N$\tu{)}}\footnote{This is equivalent to the 
requirement 
that for every $a\in[\gamma]^{<\om}$ and $\xi<\spaceord(E)$ and $f\in N$
such that $f:\kappa\to\pow([\xi]^{|a|})$, we have $\{\gamma\mid f(\gamma)\in E_a\}\in N$. Thus, it 
agrees with the terminology of 
\cite[\S2.3]{zeman} regarding extenders on structures. However, it differs from the condition 
described in \cite[Definition 1.0.8]{fsit} -- we might, for example, have 
$\OR^N<\spt(E)\in\OR$.}
iff
\[\pow(\kappa)\inter N=\pow(\kappa)\inter\Ult(N,E).\]
If $E$ is weakly amenable then $(\kappa^+)^E$ denotes $(\kappa^+)^N=(\kappa^+)^{\Ult(N,E)}$.

Suppose $E$  is 
weakly amenable.
We define the \emph{Dodd parameter} $t_E=t(E)$ of $E$,
and \emph{Dodd projectum} $\tau_E=\tau(E)$ of $E$. Let $(k,t_0,\ldots,t_{k-1})$
be lex-largest in $[\OR]^{<\om}$ (ordered as in \S\ref{ssec:ntn})
such that for each $i<k$,
\[ t_i\text{ is the largest }\{t_0,\ldots,t_{i-1}\}\text{-generator }\alpha\text{ of }E\text{ such 
that }\alpha\geq(\kappa^+)^E.\]
Here $k$ exists as $t_{i+1}<t_i$ for $i+1<k$ (maybe $k=0$).
Then $t_E=\{t_0,\ldots,t_{k-1}\}$
and $\tau_E$ is the sup
of $(\kappa^+)^E$ and all $t_E$-generators of $E$.\footnote{Note that $\tau_E>(\kappa^+)^N$ 
iff $\tau_E$ is a limit of $t_E$-generators of $E$ (by choice of $k$). 
Possibly $k=0$ and  $\tau_E=\nu(E)$; 
this happens for example if $E$ is the active extender of a type 1 or 3 premouse.}
We say $E$ is \emph{Dodd-sound} iff both
\begin{enumerate}[label=--]
 \item $E\rest (t_i\un t_E\rest i)\in\Ult(N,E)\text{ for all }i<k\text{, and}$
 \item $\text{ if }\tau_E>(\kappa^+)^N\text{ then }E\rest(\alpha\un t_E)\in\Ult(N,E)\text{
for all }\alpha<\tau_E$.
\end{enumerate}

Now let $F$ be any extender (possibly non-standard) over $N$. Let 
$E$ be the
(equivalent) standard extender, derived from $i_F$, such that
\[ \spt(E)=\sup_{\alpha\in\spt(F)}([\{\alpha\},\id]^N_F+1)\]
and $\spaceord(E)$ is as small as possible (i.e. $E=E\rest\spt(E)$).
We might apply the preceding terminology to $F$ in two 
ways,
depending on context: (a) applying it literally to $E$ (for example, ``$X$
generates $\alpha$'' would
apply to a set $X$ iff $X\sub\spt(E)$, even if
$\spt(E)\not\sub\spt(F)$); (b) dealing directly with $X\sub\spt(F)$, as with
the notation ``$F\rest X$'' introduced in \S\ref{ssec:ntn}.\footnote{That is,
let $\pi:N\to T$ be $\Sigma_0$-elementary, such that $F$ is derived from $\pi$.
Let $U=\Ult_0(N,F)$ and $k:U\to T$ the factor map. So $\rg(k)$ is isomorphic to $U$,
but maybe $\alpha<k(\alpha)\in\spt(F)$.
In case (b) we deal literally with $X\sub\spt(F)$, so might have $k(\alpha)\in X$,
whereas in case (a) we deal with the collapsed version of $F$, and sets $X\sub\spt(E)$, so might have $\alpha\in X$ (and might have $\alpha\sub X$).}
The intended meaning will hopefully be clear in context.

Let $\sigma\in\OR$ and 
$s,t\in[\OR\cut\sigma]^{<\om}$ with $\lh(s)=\lh(t)$. Let $E,F$ be extenders with $\measdom(E)=\measdom(F)$ and 
$\sigma,s,t\sub\spt(E)\inter\spt(F)$. We write
$E\equiv_{\sigma,s,t}F$ iff  
$E\rest(\sigma\un s)$ and $F\rest(\sigma\un t)$ are isomorphic, via the mapping of support which is the identity below $\sigma$ and sends $s$ to $t$. That is, 
\[ (x,a\un s)\in E\iff(x,a\un t)\in F \]
for all $x\in\measdom(E)$ and $a\in[\sigma]^{<\om}$. \end{dfn}

The first fact below is straightforward to prove; see \cite[3.1]{combin}.

\begin{fact}\label{fact:Dodd} Let $E$ be a weakly amenable
extender with $\spt(E)\in\OR$ and $\kappa=\kappa_E$.
Then $\tau_E$ is the least $\tau\geq(\kappa^+)^E$ such that
there is $t\in[\OR]^{<\om}$ with $\tau\cup t$ generating $E$. And $t_E$ is the least
$t\in[\OR]^{<\om}$ witnessing this property of $\tau_E$.
In particular, $E\approx E\rest\tau_E\un t_E$.

If $\tau_E>(\kappa^+)^E$ then $E$ is not
finitely generated. If $\tau_E>(\kappa^+)^E$ and $\tau_E\in\wfp(\Ult(N,E))$ then
$\tau_E$ is a cardinal of $\Ult(N,E)$. If
$\tau_E=(\kappa^+)^E$ then $E$ is finitely generated, indeed, generated by $t_E\un\{\kappa\}$.
\end{fact}

\begin{fact}[Steel,
\protect{\cite[3.2]{combin},\cite[4.1]{deconstruct}}]\label{fact:Doddsound}
Let $N$ be an active, $1$-sound, $(0,\om_1,\om_1+1)$-iterable premouse. Then
$F^N$ is Dodd-sound.\footnote{We remark that although the proof superficially uses $\AC$,
the fact nonetheless follows from 
$\ZF$ alone. This is because one can first pass to an inner model of choice containing the premouse 
and closed under the iteration strategy. Likewise for solidity, condensation etc.}\end{fact}

\begin{newrem}\label{rem:Dodd-sound_all_frags}
 Let $M$ be a Dodd-sound active premouse, $F=F^M$, $\mu=\kappa_F$, $t=t^M$, $\tau=\tau^M$, 
and suppose that $(\mu^+)^M<\tau$.
Let $q\in[\OR^M]^{<\om}$ and $\gamma<\tau$. Then $F\rest(q\un\gamma)$ is in $M$.
For $F$ is generated by $t\un\tau$, so we may pick $\xi\in[\gamma,\tau)$
 such that $q$ is generated by $t\un\xi$. By the second clause of Dodd-soundness, $F\rest(t\un\xi)\in M$, but $F\rest(q\un\gamma)$ is easily computed 
from $F\rest(t\un\xi)$.
\end{newrem}

\begin{rem}\label{rem:crit_generated} For active premice $N$ such that
$\Ult(N|(\kappa^+)^N,F)$ is wellfounded, where $F=F^N$ and $\kappa=\kappa_F$,
the definition of Dodd-soundness is optimal, in that no larger fragments
of $F$ can possibly appear in $\Ult(N,F)$ (but this seems to use the assumption
that $N$ is below superstrong). If $\tau_F>(\kappa^+)^F$ the optimality is clear. Suppose
$\tau_F=(\kappa^+)^F$. It is well-known that $\{\kappa\}$ generates $(\kappa^+)^F$. So if 
$t_F=\emptyset$ then $F\approx F\rest\{\kappa\}$, so $F$ is just a normal measure.
So suppose $t_F\neq\emptyset$. Then we claim that
\[ F\approx F\rest t_F\notin\Ult(N,F).\]
For 
$F\approx F\rest t_F\un\{\kappa\}$,
so it 
suffices to see
that $\kappa$ is generated by $\{\gamma\}$ for every
$\gamma\in[\kappa,\lambda)$ where $\lambda=i_F(\kappa)$.
So fix $\gamma\in[\kappa,\lambda)$. By \cite[Claim 2 of proof of
8.27]{outline} (the proof of which uses the non-superstrong assumption), we may fix $f\in 
N\inter{^\kappa\kappa}$
such that $\gamma\leq i_F(f)(\kappa)<\lambda$.
Define $g:\kappa\to\kappa$ by letting $g(\alpha)$
be the least $\beta$ such
that $\beta=\alpha$ or $f(\beta)\geq\alpha$. Then $g\in N$ and
$i_F(g)(\gamma)=\kappa$.\end{rem}

We now introduce a variant of the Dodd parameter and projectum, more analogous
to the standard parameter and projectum.

\begin{dfn}\label{dfn:lexorder} Let $\Dd$ be the class of pairs
$(s,\sigma)\in[\OR]^{<\om}\cross\OR$ such that if $s\neq\emptyset$ then
$\sigma\leq\min(s)$.
We extend $<_\lex$ to wellorder $\Dd$. Fix $(s_j,\sigma_j)\in\Dd$ for $j=1,2$.
Let $s_j^*$ be the strictly decreasing sequence with range $s_j$. Let 
$s_j'=s_j^*\conc\left<\sigma_j\right>$,
a monotone decreasing sequence. Let
$(s_1,\sigma_1)<_\lex(s_2,\sigma_2)$ iff either
$s_1'\ins s_2^*$ or there is $i\leq\min(\lh(s_1^*),\lh(s_2^*))$ such that
$s_1^*\rest 
i=s_2^*\rest i$ and 
$s_1'(i)<s_2'(i)$.
\end{dfn}

\begin{dfn}\label{dfn:Dsp} Let $M$ be a pm-based structure and $\pi:M\to N$ be
$\Sigma_0$-elementary ($N$ need not be wellfounded). We say $\pi$ is \emph{cardinal preserving} 
iff 
\[ M\sats\text{``}\alpha\text{ is an cardinal''}\iff N\sats\text{``}\pi(\alpha)\text{ is a 
cardinal''} \]
for all $\alpha\in\OR^M$. We say $\pi$ is \emph{Dodd-appropriate} iff $\pi$ is cardinal preserving,
$\mu=\crit(\pi)$ exists and is inaccessible in $M$, $M||(\mu^+)^M=N|(\mu^+)^N$ and
there is $\lambda\leq\pi(\mu)$ such that $\lambda\in\wfp(N)$ and $E_{\pi,\lambda}\notin N$.

Assume $\pi$ is Dodd-appropriate. Let $\mu=\crit(\pi)$. The \emph{Dodd-fragment parameter,
projectum} $(s_\pi,\sigma_\pi)$ of $\pi$
is the $<_\lex$-least $(s,\sigma)\in\Dd$ with
\[ \sigma\geq(\mu^+)^M\text{ and } E_{\pi}\rest(\sigma\un s)\notin N.\]

If $F$ is a weakly amenable extender over $M$, $N=\Ult(M,F)$ and $\pi=i_F$ is Dodd-appropriate, then $(s_F,\sigma_F)$ denotes
$(s_\pi,\sigma_\pi)$. For an active premouse $R$,
$(s_R,\sigma_R)$ denotes $(s_{F^R},\sigma_{F^R})$ (note that $i_{F^R}$ is Dodd-appropriate).
\end{dfn}

\begin{rem}\label{rem:Dfrag} Let $\pi:M\to N$ be Dodd-appropriate with $N$ wellfounded.
Let $\mu=\crit(\pi)$ and $\pi(\mu)\sub S\sub\OR^N$ and $F$ be the (possibly 
long) extender derived from $\pi$ with support $S$; so $E_\pi\sub F$. Let $U=\Ult(M,F)$ and $k:U\to 
N$ be the factor map. Then $\crit(k)>i_F(\mu)=\pi(\mu)$ and
\begin{equation}\label{eqn:mu^+_agmt} U|i_F(\mu)=N|\pi(\mu)\ \text{and}\ 
U|(\mu^+)^U=N|(\mu^+)^N=M||(\mu^+)^M,\end{equation}
and it easily follows that $i_F$ is Dodd-appropriate and
\[
(s_F,\sigma_F)=(s_\pi,\sigma_\pi)\leq_\lex(\emptyset,\pi(\mu)).\]

Moreover, $(s_F,\sigma_F)\leq_\lex(t_F,\tau_F)$:
Suppose not. Let $G=F\rest(t_F\un\tau_F)$. Then
\begin{equation}\label{eqn:t_tau_<} 
(t_F,\tau_F)<_\lex(s_F,\sigma_F)\leq_\lex(\emptyset,\pi(\mu)).\end{equation}
Therefore $G$ is short, and because $t_F\un\tau_F$ generates $F$, therefore $G\notin U$.
But $G\in U$ by line (\ref{eqn:t_tau_<}) and the definition of $(s_F,\sigma_F)$, contradiction.
It easily follows that $F$ is Dodd sound iff $(s_F,\sigma_F)=(t_F,\tau_F)$.

There is a characterization of $s_\pi,\sigma_\pi$ analogous to the definition of
the Dodd parameter and projectum.
Let $s=\left<s_0,\ldots,s_{l-1}\right>\in\OR^{<\om}$ be the longest possible
sequence such that for each $i<l$,
\[ s_i\text{ is the largest }\alpha\geq(\mu^+)^M\text{
 such that } E_\pi\rest ((s\rest i)\un\alpha)\in N.\]
Note $s_{i+1}<s_i$. Then $s_\pi=s$ and
$\sigma_\pi$ is the sup of $(\mu^+)^M$ and
all $\alpha$ such that $E_\pi\rest (s_\pi\un\alpha)\in S$. We omit the proof.
\end{rem}

In \S\ref{sec:bicephali} we apply some of the results of \S\ref{sec:Dodd}
to pm-based structures which are not premice. For example, we
will deal with structures of the form $\Ult(P,H)$, where $P$ is type 3 and
$\crit(H)=\nu_P$ (computing the ultrapower without squashing $P$). See
\cite[\S9]{fsit}. We now set up terminology in relation to this.

\begin{dfn}\label{dfn:segmented}
We say that $M$ is a
\emph{segmented-premouse} (\emph{seg-pm}) iff either (a) $M$ is a passive premouse, or (b)
there are $N,F,\widetilde{F}$ such that:
\begin{enumerate}
 \item $M=(N,\widetilde{F})$,
 \item $N$ is a passive premouse,
 \item $N$ has a largest cardinal $\delta$,
 \item $F$ is a short extender over $N$ of length $\OR^N$,
 \item $\nu(F)\leq\delta$,
 \item $\OR^N=(\delta^+)^{\Ult(N,F)}$,
 \item $N=\Ult(N,F)|\OR^N$, and
 \item $\widetilde{F}$ is the amenable coding of $F$ as in
\cite[2.9--2.10]{outline}, except that $\delta$ is used in place of $\nu(F)$; that
is, $\widetilde{F}$ is the set of all tuples $(\gamma,\xi,a,x)$ such that
\begin{enumerate}[label=--]
\item $\kappa_F<\xi<(\kappa_F^+)^N$ and $\delta<\gamma<\OR^N$, and
\item $F'\in N|\gamma$ and $(a,x)\in F'$ where $F'=F\inter(N|\xi\cross[\delta]^{<\om})$.
\end{enumerate}
\end{enumerate}

A segmented-premouse which is not a premouse is \emph{proper}.
\end{dfn}

\begin{rem}\label{rem:segmented}
Premice are segmented-premice. However, active seg-premice $P$ can
fail the ISC, or can have $\nu(F^P)<\lrgcrd(P)$, or both. Proper
seg-premice can arise, for
example, from taking ultrapowers of type 3 premice at the unsquashed level.
There are two main kinds of examples of this.
Let $Q$ be a type 3 premouse and $\nu=\nu(F^Q)=\lgcd(Q)$.
Let $E$ be a short extender weakly amenable to $Q$ and $P=\Ult(Q,E)$ (formed at the unsquashed level);
suppose $P$ is wellfounded.
Let $\mu=\crit(E)$. The first example arises in the case that
\[ \mu=\cof^Q(\nu)<\nu. \]
For then as discussed in \cite[\S9]{fsit} (or see \ref{lem:Dsp3}),
$i^Q_E$ is discontinuous at $\nu$ at
\[ \nu(F^P)=\sup i^Q_E``\nu<i^Q_E(\nu)=\lgcd(P).\]
But in this case $P$ satisfies the ISC.
The second kind of example results if
$\nu=\mu$ (so $\nu$ is inaccessible in $Q$).
In this case $P$ fails the ISC, as (by \ref{lem:Dsp3})
\[ \nu<\nu(F^P)=\nu(E),\]
and $F^Q\rest\nu=F^P\rest\nu$ but $F^Q\rest\nu\notin P$.

Regarding fine structure for proper seg-premice $P$, we only deal
with degree $0$. That is, we only define $\rho_0=\OR^P$ and $0$-soundness (declared trivially true), and ultrapowers of $P$ are degree $0$, formed
without squashing. Moreover, whenever we form $U=\Ult_0(P,E)$ for an
active seg-premouse $P$, without squashing $P$, we define
$\widetilde{F^U}$ as usual, that is,
\[ \widetilde{F^U}=\bigcup_{x\in 
P}i^P_E(\widetilde{F^P}\inter x).\]\end{rem}

\begin{lem}\label{lem:Dsp3} Let $P$ be an active seg-pm,
$F=F^P$, and $H$
a \textup{(}possibly long\textup{)} extender over $P$, with $W=\Ult_0(P,H)$ a seg-pm.\footnote{Seg-premousehood requires just that $W$ is wellfounded and $i^P_H$ is continuous at $(\kappa_F^+)^P$.}
Let $A_F\sub\lh(F)$ generate $F$ and $A_H\sub\spt(H)$ generate
$H$. Let $j=i^P_H$. For an extender $E$, given $t\in[\spt(E)]^{<\om}$,
write $G_{E,t}$ for the set of $t$-generators of $E$, and $G_{E}=G_{E,\emptyset}$.

Then:
\begin{enumerate}[label=\tu{(}\roman*\tu{)}]
 \item\label{item:generating_set} $A_H\un j``A_F$ generates $F^W$. Therefore $G_H\un j``G_F$ 
generates $F^W$
and if $\spaceord(H)\leq\crit(F)+1$ then $j``A_F$ generates 
$F^W$.
 \item\label{item:i_H_pres_gens} $j``G_{F,t}=G_{F^W,j(t)}\inter\rg(j)$ for each $t\in[\OR^P]^{<\om}$.
 \item\label{item:G_H_gens} If $\crit(F)<\crit(H)$ then $G_H\sub G_{F^W}$.
 \end{enumerate}
So if $\crit(F)<\crit(H)$ then $G_H\un j``G_F$ is a generating set of generators 
of $F^W$; and if $\spaceord(H)\leq\crit(F)+1$ then $j``G_F$ is a generating set of generators.\footnote{However, $F^W$ can also have generators outside of 
$G_H\un j``G_F$. For example, suppose that $P$ is a type 2 premouse,
$\crit(F)<\xi=\crit(H)<\nu(F)$, $W$ is wellfounded, and $H$ is short,
with $G_H$ bounded in $j(\xi)$. So $W$ is a premouse and $j(\xi)$ 
is a cardinal of $W$, so (by the ISC) is a limit of generators of $F^W$, whereas $\xi$ is not a 
limit point of $G_H\un j``G_F$.}
\end{lem}

\begin{proof} Part \ref{item:generating_set}: The first sentence is proven like \cite[9.1]{fsit}: 
Let
$\alpha<\nu(F^W)$.
Let
\[ \psi:\Ult_0(P,F)\to\Ult_0(W,F^W)\] be
the map induced by $j$.
Let $f\in P$ and $b\in [A_H]^{<\om}$ with
$\alpha=j(f)(b)$. Let
$g\in P$ and $c\in [A_F]^{<\om}$ with $f=i^P_F(g)(c)$. Then $\alpha$ is generated by
$b\un j(c)$ with respect to $F^W$. For $\psi\rest P=j$ and $\psi\com
i^P_F=i^W_{F^W}\com j$, so
\[ \alpha = j(f)(b)=\psi(f)(b)=\psi(i^P_F(g)(c))(b) =
i^W_{F^W}(j(g))(j(c))(b).\]

The second sentence follows easily; note that if $\spaceord(H)\leq\crit(F)+1$
then $G_H\sub j(\crit(F))=\crit(F^W)$.

Part \ref{item:i_H_pres_gens}: Because being a $t$-generator is an $\rPi_1$ property of $t$.

Part \ref{item:G_H_gens}:
Suppose $\kappa=\crit(F)<\crit(H)$ and $\alpha\in G_H$ but $\alpha\notin G_{F^W}$. 
Fix
$b\in[\alpha]^{<\om}$ and $f:[\kappa]^{|b|}\to\kappa$ with $f\in W$, such that
\[ i^W_{F^W}(f)(b)=\alpha.\]
Then $f\in P$; in fact $(\kappa^+)^P=(\kappa^+)^W<\crit(H)$
because $W$ is a premouse.\footnote{Not just a protomouse; see \cite{coveringuptowoodin}.}
Let $\mu\in P$ be such that $j(\mu)>\alpha$. Let
\[ g=i^P_{F^P}(f)\inter([\mu]^{|b|}\cross\mu).\]
Then $g\in P$ and
\[ j(g)=i^W_{F^W}(f)\inter([j(\mu)]^{|b|}\cross j(\mu)),\]
so $j(g)(b)=\alpha$,
contradicting the fact that $\alpha\in G_H$.\end{proof}

Lemma \ref{lem:Dsp4} below is essentially \cite[2.1.4]{coveringuptowoodin}.
It can be deduced
from \ref{lem:Dsp3}, using that $\tau_F\un t_F$ generates $F$, and, by
$\rPi_1$-elementarity, that
$j$ sends total fragments of $F$ to total
fragments of $F^W$.
\begin{lem}\label{lem:Dsp4} Let $P,F,H,W,j$ be as in \ref{lem:Dsp3}, with
$\spaceord(H)\leq\tau_F$.
Assume that $F$ is Dodd sound. Then $F^W$ is Dodd sound. Moreover,
$t_{F^W}=j(t_F)$ and $\tau_{F^W}=\sup(j``\tau_F)$.\end{lem}

\begin{lem}\label{lem:Dsp5} Let $P,F,H,W,j$ be as in \ref{lem:Dsp3}.
Assume $\sigma_F\leq\crit(H)$ and
\[ \pow(\sigma_F)\inter P=\pow(\sigma_F)\inter W. \] Then $F^W$ is \emph{not} Dodd
sound.
Moreover, $s_{F^W}=j(s_F)$ and $\sigma_{F^W}=\sigma_F$.\end{lem}

\begin{proof} Let $\sigma=\sigma_F$ and $s=s_F$ and 
$s'=j(s)$. We have $(s_{F^W},\sigma_{F^W})\geq_\lex(s',\sigma)$
as $j$ maps fragments of $F$ to fragments of $F^W$. 
As $\sigma\leq\crit(H)$,
\[ F\equiv_{\sigma,s,s'}
F^W\]
(cf.~\ref{dfn:gen}). So $F\rest(\sigma\un s)$ and $F^W\rest(\sigma\un s')$ are coded by identical sets in $\pow(\sigma)$,
and $\pow(\sigma)\inter P=\pow(\sigma)\inter W$, so
\[ F^W\rest(\sigma\un s')\notin W.\]
Therefore $(s_{F^W},\sigma_{F^W})=(s',\sigma)$. And as in the proof of 
\ref{lem:Dsp3}\ref{item:G_H_gens}, $\crit(H)$ is an $s'$-generator
of $F^W$, so $F^W$ is not Dodd sound.
\end{proof}

The proof of \ref{lem:gen1hull} to follow is similar to that of
\cite[4.4]{combin}. Moreover, \ref{cor:taurho_1} improves the conclusion of
\cite[4.4]{combin} to $0=1$. Recall from \cite{fsit} that the premouse language uses the constant symbols $\dot{\mu}$,
$\dot{\nu},\dot{\gamma}$, interpreted by $\mu^P=\crit(F^P)$, $\nu^P=\nu(F^P)$ and $\gamma^P=$ the largest witness to the ISC for $P$, when $P$ is type 2.

\begin{newdfn}
Let $P$ be a premouse, $F=F^P$, $\mu=\mu^P$, $X\sub\OR^P$. Then
$\Lambda^P(X)$ denotes the set of all $x\in P$ such that  $x=[a,f]^P_{F^P}$
for some $f\in P$ and $a\in[X]^{<\om}$ (note we may take $f\in P|(\mu^+)^P$).
\end{newdfn}

\begin{lem}\label{lem:gen1hull} Let $P$ 
be a type 2 premouse, $F=F^P$,
$\mu=\mu^P$, $X\sub\OR^P$ and $\Lambda=\Lambda^P(X)$. Suppose $\nu^P,\gamma^P\in\Lambda$.
Then
\[ \Lambda=\{x\in P\mid x\in\Hull_1^P(X\un(\mu^+)^P)\}.\]
\end{lem}
\begin{proof}
Clearly
\[ \mu\un X\sub\Lambda\sub H\eqdef\Hull_1^P(X\un(\mu^+)^P).\]
Let us show $H\sub\Lambda$. We have $\mu\in\Lambda$ because $\nu^P\in\Lambda$ and by the 
argument in
\ref{rem:crit_generated}. So
$(\mu^+)^P\un X\sub\Lambda$, so it suffices to see that 
$\Lambda\elem_{\rSigma_1} P$, and for this we define a premouse $\bar{P}$ and an 
$\rSigma_1$-elementary 
$\pi:\bar{P}\to P$ such that $\rg(\pi)=\Lambda$. Recall here that 
$\rSigma_1$-elementarity is with respect to the language of premice, incorporating
$\dot{\mu},\dot{\nu},\dot{\gamma}$. 

Let $F^*=F\rest(\Lambda\inter\OR)$, $Q=P|(\mu^+)^P$ and $\bar{R}=\Ult(Q,F^*)$ and $R=\Ult(Q,F)$. Let
$\varrho:\bar{R}\to R$ be the natural factor map. Let $\varrho(\bar{\nu})=\nu^P$ and 
$\xi=(\bar{\nu}^+)^{\bar{R}}$ and 
$\pi=\varrho\rest(\bar{R}|\xi)$. So $(\mu^+)^P<\crit(\pi)$ and
$\rg(\pi)=\Lambda$ and $\varrho(\xi)=\OR^P$. Let $\bar{F}$ be the $(\mu,\xi)$-extender derived from $i^Q_{F^*}$; so $(X,a)\in\bar{F}$ iff $(X,\pi(a))\in F^*$.
Let $\bar{P} = (\bar{R}|\xi,\widetilde{\bar{F}})$, where $\widetilde{\bar{F}}$ codes $\bar{F}$ as in
\cite[2.9--2.10]{outline},
but replacing the triple $(\alpha,E_\alpha,\nu(E_\alpha))$ of \cite{outline}
with $(\xi,\bar{F},\bar{\nu})$. So $\pi:\bar{P}\to P$. Let $\gamma^{\bar{P}}=\pi^{-1}(\gamma^P)$, 
$\nu^{\bar{P}}=\bar{\nu}$,
$\mu^{\bar{P}}=\mu$. It suffices to see that $\pi$ is $\rSigma_1$-elementary (so $\bar{P}$ is a
premouse).

By construction, 
$\pi(\dot{\gamma}^{\bar{P}},\dot{\nu}^{\bar{P}},\dot{\mu}^{\bar{P}})=(\gamma^P,\nu^P,
\mu^P)$.\footnote{This is the main reason for incorporating the 
assumption that $\gamma^P,\nu^P\in\Lambda$. There are standard counterexamples when these 
assumptions fail; see the proof 
that iterable $1$-sound mice are Dodd-sound.} Now $\varrho\com i_{\bar{F}}^Q=i_F^Q$. So for
$A\in\pow([\mu]^{<\om})\inter P$ and $\gamma<\xi$, we have
\[ \pi(i^Q_{\bar{F}}(A)\inter[\gamma]^{<\om}) = i^Q_F(A)\inter[\pi(\gamma)]^{<\om}.\]
So $\pi``\widetilde{\bar{F}}\sub\widetilde{F^P}$. Since
$\pi\rest(\mu^+)^P=\id$, \cite[2.9]{outline} shows that
$\pi$ is cofinal in $P$. It follows that $\pi$ is
$\rSigma_1$-elementary.\end{proof}
Part (ii) of the following lemma is proved in almost the same way; part (i) is easy:

\begin{newlem}\label{lem:ext_hull_equiv}
 Let $P,F,\mu$ be as in \ref{lem:gen1hull}, $\xi\in\OR^P$ and $q\in[\OR^P]^{<\om}$,
 with $(\mu^+)^P\leq\xi$. Then (i) if $\cHull_1^P(q\un\xi)\in P$ then $F\rest(q\un\xi)\in 
P$, and (ii) if $F\rest(q\un\xi)\in P$ and $\gamma^P,\nu^P\in\Lambda^P(q\un\xi)$ then 
$\cHull_1^P(q\un\xi)\in P$.
\end{newlem}

\begin{cor}\label{cor:taurho_1}
Let $P$ be a type 2,
Dodd sound premouse, and $\mu=\mu^P$.

Then
$\tau_P=\rho_1^P\cup(\mu^+)^P$.\end{cor}
\begin{proof}
Let $\gamma=\rho_1^P\un(\mu^+)^P$. Easily $\tau_P\geq\gamma$, so suppose
$\tau_P>\gamma$. Let $X=q\un\gamma$,
where
$q\in[\OR^P]^{<\om}$ is such that $p_1^P,\nu^P,\gamma^P\in\Lambda^P(q)$. By 
\ref{rem:Dodd-sound_all_frags},
$F^P\rest X\in P$, so
by \ref{lem:ext_hull_equiv}, $\cHull_1^P(X)\in P$, so
$\Th_1^P(\rho_1^P\un\{p_1^P\})\in P$, contradiction.
\end{proof}

\begin{lem}\label{lem:Dsp6} Let $k\geq 1$ and let $P$ be a $k$-sound, type 2,
Dodd sound premouse. Let $H$ be a short extender,
weakly amenable to $P$, with $\crit(H)<\rho_k^P$ and $R=\Ult_k(P,H)$ wellfounded.
Let $j=i^{P,k}_H$.

Then $R$ is Dodd sound. Moreover, $t_R=j(t_P)$.

If $\tau_P=((\mu^P)^+)^P$ or $k>1$ then $\tau_R=j(\tau_P)$.

If $k=1$ then
$\tau_R=\sup j``\tau_P$.\end{lem}

\begin{proof}\footnote{If $R$ is $(0,\om_1,\om_1+1)$-iterable, then the Dodd
soundness
of $R$ follows from \ref{fact:Doddsound}. But we don't want to assume this much iterability here.}
If $k>1$, elementarity considerations give the lemma.
So suppose $k=1$. Let $F=F^P$, $\mu=\mu^P$, $t=t_F$ and $\tau=\tau_F$. Now $j$ maps fragments of $F$ to fragments of $F^R$ 
by elementarity.

Suppose $\tau=(\mu^+)^P$. Then $F$ is generated by $t\un\{\mu\}$, an $\rPi_2$
condition, preserved by $j$, so
$R$ is Dodd sound with $t_R=j(t_P)$ and $\tau_R=(j(\mu)^+)^R=j(\tau_P)$. But $j$ is continuous at $(\mu^+)^P$, as $(\mu^+)^P$ is $\bfrSigma_1^P$-regular, as
$\mu<\rho_1^P$.

Now suppose $\tau>(\mu^+)^P$. It suffices to see that $j(t)\un\sup
j``\tau$ generates $F^R$. Suppose not. Let $\kappa=\crit(H)$. Let $a,f$ be such that 
$\gamma=[a,f]^{P,1}_H$ is a
$j(t)$-generator 
for $F^R$ with $\gamma\geq\sup j``\tau$, where $f:[\kappa]^n\to P$
is given by a $\rSigma^P_1(\{q\})$ term for some $q\in P$. Since the statement
``$\alpha$ is a $j(t)$-generator'' is $\rPi_1$, $\rg(f)$
includes a $\tau$-cofinal set of $t$-generators. We have
\[\rg(f)\sub 
J=\Hull_1^P(\kappa\cup\{q\}).\] Let $\lambda<\tau$ be
large enough that
\[ \mu,q,\nu^P,\gamma^P\in\Lambda=\Lambda^P_F(t\un\lambda).\]
Then $\Lambda$ has no $t$-generators above $\lambda$.
But by \ref{lem:gen1hull}, $J\sub\Lambda$, contradiction.
\end{proof}

For the iteration trees $\Uu$ we will encounter, $E^\Uu_\alpha$ will always
be weakly amenable (over $M^{*\Uu}_{\alpha+1}$), but ostensibly may not be close to $M^{*\Uu}_{\alpha+1}$.
We now show that weak amenability gives a little more fine structural
preservation than
was established in \cite{fsit}. The
argument is related to that for \cite[6.2(Claim 5)]{fsit} (though there, the
extenders were always close to the relevant models), and also related
to some of the preceding lemmas.

\begin{dfn}\label{dfn:solidp} Let $N$ be a $k$-sound premouse with $\om<\rho_k^N$. The
\emph{$(k+1)$-solidity
parameter, projectum} of $N$, denoted
$(z^N_{k+1},\zeta^N_{k+1})$, is the $<_\lex$-least pair
$(z,\zeta)\in\Dd$ such that $\zeta\geq\om$ and
$\Th_{k+1}^N(\zeta\un\{z\})\notin N$. (See \ref{dfn:lexorder}.)\end{dfn}

\begin{rem}\label{rem:solidp}
The (rather trivial) requirement that $\zeta\geq\om$ is just made to
simplify the comparison between $\zeta_{k+1}$ and $\rho_{k+1}$.
By \cite[2.10]{fsit}, one can equivalently replace ``$\Th$'' in
\ref{dfn:solidp} with ``$\pTh$''. The author does not know an example, but
it seems that $\zeta_{k+1}^N$ might fail to be a cardinal of $N$. However, the
following facts follow easily from the definition; let $(z,\zeta)=(z_{k+1}^N,\zeta_{k+1}^N)$ and 
$(p,\rho)=(p_{k+1}^N,\rho_{k+1}^N)$:
\begin{enumerate}[label=--]
\item There is a characterization of
$(z,\zeta)$ like that of the Dodd-solidity ordinals in
\ref{rem:Dfrag}.
\item $\rho\leq\zeta\leq\rho_k^N$ and $(z,\zeta)\leq_\lex(p,\rho)$.
\item $N$ is $(k+1)$-solid iff
$p=z$ iff $\rho=\zeta$.
\item $N$ is $(k+1)$-sound iff $N=\Hull_{k+1}^N(\zeta\un
z)$.
\end{enumerate}
(The last item follows from the others: Suppose $N=\Hull_{k+1}^N(\zeta\un
z)$ but $\rho<\zeta$. Note $p=z\un q$ for some 
$q\sub[\rho,\zeta)$ with $q\neq\emptyset$. But then $(p,\rho)<_\lex(z,\zeta)$, a contradiction.)
\end{rem}

\begin{lem}\label{lem:soliditypreservation}Let $N$ be a $k$-sound premouse. Let
$E$ be a short extender
weakly amenable to $N$ with $\crit(E)<\rho_k^N$. Let $U=\Ult_k(N,E)$ and $j=i^{N,k}_E$. Suppose $U$ is wellfounded. Then
$z^U_{k+1}=j(z^N_{k+1})$ and
$\zeta^U_{k+1}=\sup j``\zeta^N_{k+1}$.\end{lem}
\begin{proof}
Let $(z,\zeta)=(z_{k+1}^N,\zeta_{k+1}^N)$ and $z'=j(z)$ and $\zeta'=\sup
j``\zeta$.

First observe that $(z',\zeta')\leq_\lex(z^U_{k+1},\zeta^U_{k+1})$: By 
the proof of \cite[4.6]{fsit}, if $(b,\alpha)\in\Dd$ and
$\Th_{k+1}^N(\alpha\un\{b\})\in N$ then
\[ \Th_{k+1}^U(j(\alpha)\un\{j(b)\})\in U.\]

It remains to see that
\begin{equation}\label{eqn:theorymissing}
t'\eqdef\pTh_{k+1}^U(\zeta'\un\{z'\})\notin U. \end{equation}

So suppose $t'\in U$. We show
$t\eqdef\pTh_{k+1}^N(\zeta\un\{z\})\in N$, a contradiction.

Let $\kappa=\crit(E)$.
If $\kappa\geq\zeta$ then $t\approx t'$ 
by $\rSigma_{k+1}$ elementarity (exchange $z$ with $z'$). So by weak amenability, $t\in N$.
(The details in \ref{lem:Dsp5} are analogous.)

So $\kappa<\zeta$. Also, $\zeta<\rho_k^N$.
For otherwise $\zeta=\rho_k^N$. But
$\pTh_{k+1}^N(\rho_k^N)\notin N$
(since $\pTh_{k+1}$ has access to parameters $p_k^N,u_{k-1}^N$), so
$z=\emptyset$, and $\rho_k^U=\zeta'$, so
\[ t'=\pTh_{k+1}^U(\rho_k^U)\notin U,\]
contradiction.

Since $\zeta<\rho_k^N$ there is $f\in N$ and $b\in[\nu(E)]^{<\om}$ such that
$f:[\kappa]^{<\om}\to N$ and
$t'=[b,f]^{N,k}_E$. (We get $f\in N$ even if $\rho_k^N=(\zeta^+)^N$,
because in this case there is no $\bfrSigma_k^N$ singularization of $\rho_k^N$.)
We will compute
$t$ from $f$, giving $t\in N$. Now $\pTh_{k+1}$ is a set of
$\bfrSigma_{k+1}$ formulas. We may
assume that for each $x\in[\kappa]^{<\om}$, $f(x)$ is a set of $\rSigma_{k+1}$ formulas in parameters in
$\zeta\un\{z,u_k^N\}$.

For $\alpha<\zeta$, let
\[ t_\alpha^N=\pTh_{k+1}^N(\alpha\un\{z\})\]
and let $f_\alpha:[\kappa]^{|b|}\to N$ be defined
\[ f_\alpha(x)=f(x)\rest(\alpha\un\{z,u_k^N\}), \]
where $T\rest X$ is the restriction of a theory $T$ to parameters in $X$.
For $\alpha<\zeta'$, let
\[ t_\alpha^U=\pTh_{k+1}^U(\alpha\un\{z'\}). \]

\begin{case} For each $\gamma<\zeta$,
$j(t_\gamma^N)=t_{j(\gamma)}^U$.\end{case}

We leave this case to the reader; it is simpler than the next one.

\begin{case}\label{case:a_case}For some $\gamma<\zeta$, $j(t_\gamma^N)\neq
t_{j(\gamma)}^U$.\end{case}

Fix such a $\gamma$. Let $<^*$ be the prewellorder on $\pTh_{k+1}^N(N)$ defined
in \cite[Proof of 2.10]{fsit}. It follows that $t_\gamma$ is cofinal in
$<^*$, and that
$\cof^N(<^*\rest t_\gamma)=\kappa$. Let
$h\in N$ be such that $h:\kappa\to t_\gamma$ and $h$ is cofinal,
increasing and continuous with respect to $<^*$.
For $\alpha\in[\gamma,\zeta)$ let
$g_\alpha:\kappa\to N$ be defined
\[ g_\alpha(\beta)=t^N_\alpha\rest^{<^*}h(\beta) \]
where $T \rest^{<^*}\varphi$ is the restriction of a theory $T$ (of the appropriate kind) to 
formulas $\psi$ such that $\psi<^*\varphi$. Then $g_\alpha\in N$ (maybe $\left<g_\alpha\right>_{\alpha}\notin N$) and
\[ [g_\alpha,\{\kappa\}]^{N,k}_E=t_{j(\alpha)}^U.\]
We
may assume that $\kappa=\min(b)$. Let $g_\alpha':[\kappa]^{|b|}\to N$ be defined
\[ g_\alpha'(x)=g_\alpha(\min(x)).\]
So there is $X_\alpha\in E_b$ such that
$g_\alpha'\rest X_\alpha=f_\alpha\rest X_\alpha$.

Define the relation
$R\sub\zeta\cross([\kappa]^{|b|})^2$ by
\[ R(\alpha,x_1,x_2)\iff [\max(x_1)<\min(x_2)\text{ and }f_\alpha(x_1)\sub f_\alpha(x_2)].\]
So
$R\in N$. Define the relation $\psi\sub\zeta\cross[\kappa]^{|b|}$ by
\[ \psi(\alpha,x_1)\iff R(\alpha,x_1,x_2)\text{ holds for }E_b\text{-almost all }x_2.\]
Fix $\alpha\in[\gamma,\zeta)$. Note that if $x_1,x_2\in X_\alpha$ and
$\max(x_1)<\min(x_2)$, then $R(\alpha,x_1,x_2)$. So if $x_1\in X_\alpha$
then $\psi(\alpha,x_1)$.
On the other hand, if $\psi(\alpha,x_1)$ then there is
$x_2\in X_\alpha$ such that
$R(\alpha,x_1,x_2)$, and so
$f_\alpha(x_1)\sub t_\alpha$. So let 
$S_\alpha=\{x_1\mid\psi(\alpha,x_1)\}$. Then $S_\alpha\in N$ as $R\in N$ and by weak amenability, and
\[ t_\alpha=\bigcup_{x_1\in S_\alpha}f_\alpha(x_1). \]

Suppose $\cof^N(\zeta)\leq\kappa$.
Fix a cofinal function $g:\kappa\to\zeta$ with $g\in N$ and
$\rg(g)\sub[\gamma,\zeta)$. 
Again by weak amenability and as $R\in N$,
\[ S'\eqdef\{(\beta,x)\mid\psi(g(\beta),x)\}\in N.\]
But
then
\[ t=\bigcup_{(\beta,x)\in S'}f_{g(\beta)}(x)\in N,\]
as required.

Now suppose instead that $\cof^N(\zeta)>\kappa$. Let
$R_\alpha$ be the $\alpha$-section of $R$. For
$\alpha_1,\alpha_2\in[\gamma,\zeta)$ with $\alpha_1<\alpha_2$, we have
$R_{\alpha_2}\sub R_{\alpha_1}$.
So $R_{\alpha_1}=R_{\alpha_2}$ for all sufficiently large
$\alpha_1,\alpha_2<\zeta$. Fix such an $\alpha_1$. Then
$S\eqdef S_\alpha$ is independent of
$\alpha\in[\alpha_1,\zeta)$, and $S\in N$. So let
$S'=[\alpha_1,\zeta)\cross S$. Then
\[ t=\bigcup_{(\alpha,x)\in S'}f_\alpha(x)\in N, \]
completing the
proof.\renewcommand{\qedsymbol}{$\Box$(\ref{lem:soliditypreservation})}
\end{proof}

\begin{ntn} Let $N$ be a $k$-sound seg-pm. Let
$\Ff=\left<F_\alpha\right>_{\alpha<\lambda}$ be a sequence of weakly amenable short extenders. For
$\xi\leq\lambda$, let $\left<N_\alpha\right>_{\alpha\leq\xi}$
be the degree $k$ iterated ultrapower based
on $N$, using the extenders in $\Ff$, if it is defined. That is, $N_0=N$,
\[ N_{\alpha+1}=\Ult_k(N_\alpha,F_\alpha),\]
and we take direct limits at limit $\alpha$.
If $N_\alpha$ is illfounded, or $F_\alpha$ is not an extender over $N_\alpha$,
or $\crit(F_\alpha)\geq\rho_k^{P_\alpha}$, then for all $\beta>\alpha$,
$N_\beta$ is undefined. Say that $\Ff$ is \emph{$k$-pre-good for $N$} iff $N_\lambda$ is defined (but 
is maybe illfounded), and \emph{$k$-good} iff $k$-pre-good and wellfounded.
Suppose $\Ff$ is $k$-pre-good for $N$. We write
\[ \Ult_k(N,\Ff)=N_\lambda \]
and
\[ i^{N,k}_\Ff:N\to\Ult_k(N,\Ff) \]
for the ultrapower embedding. If $N$ is active,
$\Ult_k(F^N,\Ff)$ and $\Ult_k(\id,\left<F^N\right>\conc\Ff)$ both denote
$F(\Ult_k(N,\Ff))$.\end{ntn}

Following the (kind of) calculations used in \cite[Proof of Lemma 
2.10]{fsit}:

\begin{newdfn}[Minimal Skolem terms]\label{dfn:minterm}Let $\varphi$ be an
$\rSigma_{k+1}$ formula of
$n+1$ free variables. The \emph{minimal Skolem term associated
to $\varphi$} is denoted $\minterm_\varphi$, and
has $n$ variables.
Let $R$ be a $k$-sound premouse with $\om<\rho_k^R$. We define the partial
function
\[ \minterm_\varphi^R:\core_0(R)^n\to\core_0(R), \]
by induction on $k$, maintaining the following: Let
$\Gamma=\rSigma_{k+1}^R(\{p_k^R,u_{k-1}^R\})$. The graph of $\minterm_\varphi^R$
is $\Gamma$, and
$\Gamma$ is closed under substitution of
minimal Skolem terms; this is uniform jointly in $\varphi$ and in $k$-sound premice
$R$ of the same type. Moreover, for $X\sub\core_0(R)$,
\[ \Hull_{k+1}^R(X)=\{\minterm^R_\varphi(\xvec,u_k^R)\mid\varphi\text{ is
}\rSigma_{k+1}\ \&\ \xvec\in X^{<\om}\}.\footnote{Recall
here that $\Hull_{k+1}^R(X)$ includes $u_k^R$ by definition;
likewise $\Th_{k+1}^R(X)$ is a theory in parameters in $X\cup\{u_k^R\}$.}\]

If $k=0$ then $\minterm_\varphi^R$ is just the usual $\rSigma_1$
Skolem function associated to $\varphi$.

Suppose $k\geq 1$. Let $u=u_{k-1}^R$ and $p=p_k^R$.
Let $\psi$ be an $\rSigma_1$ formula and $\varphi(v,w)$ be the $\rSigma_{k+1}$ formula
\[ \ex t\in T_k[\psi(t,v,w)].\]
Let $T$ be a set
of $\rSigma_k$ formulas in parameters in $\alpha\cup\{u,p\}$, where $\alpha<\rho_k^N$.
Let $\gamma<\alpha$, let $\vec{\beta}=(\beta_0,\beta_1,\beta_2,\beta_3)\in\alpha^4$,
let $\vec{r}=(r_0,r_1,r_2,r_3)$ be minimal $\rSigma_k$ Skolem terms.
Write $s_i$ for the term-in-parameters $r_i(\beta_i,u,p)$.
Let $x\in\core_0(R)$.
Say $T$ \emph{codes a witness to $\ex w\varphi(x,w)$ at
$(\gamma,\vec{\beta},\vec{r})$} iff
\begin{enumerate}[label=--]
\item $x=s_2^R$ (that is, $r_2^R(\beta_2,u,p)$ is defined and $=x$),
\item $T$ contains the formula ``$\ex t,q,x',w[t=s_0,q=s_1,x'=s_2,w=s_3$,
$t$ is a set of $\rSigma_k$ formulas in parameters in $\gamma\cup\{q\}$,
and $\psi(t,x',w)]$'',
\item for each $\rSigma_k$ formula $\varrho$ and $\gammavec\in\gamma^{<\om}$,
$T$ contains the formula
\[ \text{``}\ex t,q[t=s_0,q=s_1\text{ and }t\text{ contains the formula ``}\varrho(\gammavec,q)\text{''}]\text{''}\]
iff $T$ contains the formula ``$\ex q[q=s_1\text{ and }\varrho(\gammavec,q)]$''.
 \end{enumerate}

Let $\varphi(v,w)$ be $\rSigma_{k+1}$.
 Then $\varphi'(v,y)$ is the formula ``$\varphi(v,y)$ and letting $(\alpha,\gamma,\betavec,\vec{r})$ be lex-least
 such that $\Th_k(\alpha\cup\{p_k\})$ codes a witness to $\ex w\varphi(v,w)$, then $y=r_3(\beta_3,u_{k-1},p_k)$''
 (the formula is to be interpreted over $k$-sound premice $S$, with $u_{k-1},p_{k-1}$ interpreted, of course, as $u_{k-1}^S,p_k^S$).
 Note that:
 \begin{itemize}[label=--]
  \item 
$\varphi'$ is $\rSigma_{k+1}(\{u_{k-1},p_k\})$ (uniformly so),
\item $\varphi\mapsto\varphi'$ is recursive (when we choose $\varphi'$ naturally),
\item $\core_0(R)\sats\ex w\varphi(x,w)$ iff $\core_0(R)\sats\ex w\varphi'(x,w)$ iff $\core_0(R)\sats\ex !w\varphi'(x,w)$.
\end{itemize}

We now define $\minterm_\varphi^R$, for $\varphi$ $\rSigma_{k+1}$.
Let $\xvec\in\core_0(R)^n$.
If $\core_0(R)\sats\neg\ex y\varphi(\xvec,y)$, then
$\minterm_\varphi^R(\xvec)$ is undefined.
Otherwise
\[ \minterm_\varphi^R(\xvec)=\text{ the unique }w\in\core_0(R)\text{ such that }\core_0(R)\sats\varphi'(\xvec,w).\qedhere\]
\end{newdfn}

\begin{cor}\label{cor:k+1solid}Let $N$ be a $k$-sound premouse.
Let $\Ff=\left<F_\alpha\right>_{\alpha<\lambda}$ be a sequence of weakly amenable short
extenders, $k$-good for $N$. Let $N_\alpha=\Ult_k(N,\Ff\rest\alpha)$.
\tu{(}So $N_\lambda$ is well-defined and wellfounded.\tu{)}
Let $j=i^{N,k}_\Ff$.
Then:
\begin{enumerate}
 \item\label{item:z_zeta_pres} $z_{k+1}^{N_\lambda}=j(z_{k+1}^N)$ and $\zeta_{k+1}^{N_\lambda}=\sup
j``\zeta_{k+1}^N$.
 \item\label{item:p_rho_pres} If $N,N_\lambda$ are both $(k+1)$-solid then
$p_{k+1}^{N_\lambda}=j(p_{k+1}^N)$ and $\rho_{k+1}^{N_\lambda}=\sup j``\rho_{k+1}^N$.
\item\label{item:non-solid} If $N$ is not $(k+1)$-solid then neither is $N_\lambda$.
\item\label{item:soundness_charac} $N_\lambda$ is $(k+1)$-sound iff $N$ is $(k+1)$-sound and
$\crit(F_\alpha)<\rho_{k+1}^{N_\alpha}$ for
all $\alpha<\lambda$.
\end{enumerate}
\end{cor}

\begin{proof}
Let $\rho=\rho_{k+1}^N$, $p=p_{k+1}^N$,
$\zeta=\zeta_{k+1}^N$, $z=z_{k+1}^N$,
$\zeta'=\sup j``\zeta$ and $z'=j(z)$. 

Part \ref{item:z_zeta_pres} is by induction on $\lambda$, using
\ref{lem:soliditypreservation}. We just verify that the induction does not fail at a
limit stage. We do have the necessary
theories in $N_\lambda$ as usual. So suppose
\[ t'=\Th_{k+1}^{N_\lambda}(\zeta'\un z')\in N_\lambda. \]
Let $\alpha<\lambda$ and $i_{0\alpha}=i^{N,k}_{\Ff\rest[0,\alpha)}$ and $i_{\alpha\lambda}=i^{N_\alpha,k}_{\Ff\rest[\alpha,\lambda)}$
 and $t^*$ be such that
$i_{\alpha\lambda}(t^*)=t'$. Let
$z^*=i_{0\alpha}(z)$ and $\zeta^*=\sup i_{0\alpha}``\zeta$.
Then by $\rSigma_{k+1}$ elementarity,
\[ \Th_{k+1}^{N_\alpha}(\zeta^*\un z^*)=t^*\in N_\alpha, \]
contradicting the inductive hypothesis.

Part \ref{item:p_rho_pres} follows from part \ref{item:z_zeta_pres} and \ref{rem:solidp}.

Part \ref{item:non-solid}: Suppose that $N$ is not $(k+1)$-solid. Then by
\ref{rem:solidp}, $\rho<\zeta$.
For a premouse $M$ and $q\in M$, let $\zeta_{k+1}^M(q)$ and
$z_{k+1}^M(q)$ be the relativization of the $(\zeta,z)$ notions for $M$ at degree
$k+1$, to theories in the expanded language with a constant symbol $\dot{q}$ interpreted 
by $q$. Clearly the preceding results
relativize (for example,
$i(z_{k+1}^M(q))=z^R_{k+1}(i(q))$ for appropriate ultrapower maps $i:M\to R$, etc). Now 
$\zeta^N_{k+1}(p)=\rho$ and
$z^N_{k+1}(p)=\emptyset$. So
\[
\rho_{k+1}^{N_\lambda}\leq\zeta^{N_\lambda}_{k+1}
(j(p))=\sup
j``\zeta_{k+1}^N(p)<\zeta'. \]
So again by \ref{rem:solidp}, $N_\lambda$ is not $(k+1)$-solid.

Part \ref{item:soundness_charac}: Suppose that $N$ is not $(k+1)$-sound. Then by \ref{rem:solidp},
\[ N\neq
H\eqdef\Hull_{k+1}^N(\zeta\un z),\]
and since $\zeta'\leq j(\zeta)$ it suffices to see
\[ N_\lambda\neq\Hull_{k+1}^{N_\lambda}(j(\zeta)\un z').\]
So let $x\in N\cut H$. We claim that
\begin{equation}\label{eqn:image_x_not_in_hull}
j(x)\notin\Hull_{k+1}^{N_\lambda}(j(\zeta)\un z'),\end{equation}
which suffices. Now clearly $\zeta<\rho_k^N$. Stratify $H$ via the methods of  \cite[2.10]{fsit} (cf.~\ref{dfn:minterm}).
So
\[ H=\bigcup_{\zeta\leq\alpha<\rho_k^N}H_\alpha,\]
where $H_\alpha$ is the
set of all $y\in H$ such that for some $\beta<\zeta$ and minimal $\rSigma_{k+1}$ Skolem term $r$,
\[ t_\alpha=\Th_k^N(\alpha\un\{p_k^N\})\]
codes a witness to the fact that $y=r(z,\beta,u_k^N)$. 
(Here if $k=0$, $t_\alpha$ should be replaced with the $\alpha^\nth$ ``fragment'' of
$\core_0(N)$ in the usual manner.)

Let $\alpha\in(\zeta,\rho_k^N)$ with
\[ x,z,u_k^N\in\Hull_k^N(\alpha\un\{p_k^N\}),\]
or $x,z,u_k^N\in t_\alpha$ if $k=0$, and  $\alpha$ a  limit if $k>0$. The
fact that $x\notin H_\alpha$ is then an $\rSigma_{k}$ assertion about
$(x,z,u_k^N,t_\alpha)$.\footnote{The $\rSigma_k$ complexity arises in identifying some terms which output $x,z,u_k^N$.} 
But
\[j(t_\alpha)=\Th_k^{N_\lambda}(j(\alpha)\un\{p_k^{N_\lambda}\}),\] and
$j$ is $\rSigma_{k}$-elementary, so  $j(x)\notin
H^{N_\lambda}_{j(\alpha)}$, where $H^{N_\lambda}_\gamma$ is defined
over $N_\lambda$ analogously. Since $\rho_k^{N_\lambda}=\sup
j``\rho_k^N$, line (\ref{eqn:image_x_not_in_hull}) follows. So by
\ref{rem:solidp}, $N_\lambda$ is not $(k+1)$-sound.

Now suppose instead that $N$ is $(k+1)$-sound. Let $\xi$ be least such that either
\[ \xi=\lambda\text{ or }\crit(F_\xi)\geq\rho'=\sup i_{0\xi}``\rho.\] Let
$p'=i_{0\xi}(p)$. Then by standard arguments, 
$\rho_{k+1}^{N_\xi}=\rho'$, $p_{k+1}^{N_\xi}=p'$ and $N_\xi$ is $(k+1)$-sound. (Alternatively,
some of those arguments can be avoided by noting that by part \ref{item:z_zeta_pres} and some 
other arguments,\[ N_\xi=\Hull_{k+1}^{N_\xi}(\zeta^{N_\xi}_{k+1}\un\{z_{k+1}^{N_\xi}\}),\]
so by \ref{rem:solidp}, $N_\xi$ is $(k+1)$-sound.) 
So we are done if 
$\lambda=\xi$, and if $\lambda>\xi$ then note that
$\cHull_{k+1}^{N_\lambda}(\zeta'\un z')=N_\xi\neq N_\lambda$,
so $N_\lambda$ is not sound.\renewcommand{\qedsymbol}{$\Box$(\ref{cor:k+1solid})}
\end{proof}

\begin{newrem}\label{rem:fs_pres_optimal}
 The previous result is optimal in the sense that a degree $0$ ultrapower of a 
$1$-sound but non-$2$-sound structure, can be fully sound. For let $N$ be a $1$-sound premouse such 
that for some $\kappa<\lambda<\OR^N$, $N=\J_\kappa(N|\lambda)$ and $\rho_1^N=\lambda$ and $\lambda$ 
is $\bfrSigma_1^N$-regular, and suppose there is a short extender $E$ weakly amenable to $N$ with 
$\crit(E)=\kappa$. Let $U_1=\Ult_1(N,E)$ and $U_0=\Ult_0(N,E)$, and suppose that each $U_i$ is 
wellfounded. Then it is easy to see that $U_0\pins U_1$, and so $U_0$ is fully sound. We can arrange 
this situation with $N$ being non-$2$-sound (and so $\rho_2^N<\lambda$, but we might have either 
$\kappa<\rho_2^N$ or $\rho_2^N\leq\kappa$).
\end{newrem}

We now analyze extenders used in
iteration trees. The analysis decomposes such extenders into linear
compositions of Dodd sound extenders, identified via Dodd-fragment parameters and projecta.
The associativity of extenders underlies the analysis.

\begin{dfn}\label{dfn:extcomp} Let $P$ be
a $k$-sound segmented-premouse, $F=F^P$, and $E$ a short extender over $P$ with
$\crit(E)<\rho_k^P$. Then $E\com_k P$ denotes $\Ult_k(P,E)$; $E\com_k F$ or
$\Ult_k(F,E)$ denotes $F^{\Ult_k(P,E)}$. Write $\com$ for $\com_0$. In the
absence of parentheses, we take association of $\com_k$ to the right:
$E\com_k F\com_l Q = E\com_k (F\com_l Q)$.
\end{dfn}

\begin{lem}[Associativity of Extenders]\label{lem:extass}
Let $P,Q$ be seg-premice.
Suppose  $F=F^P\neq\emptyset$, $F$ is over $Q$, $Q$ is $k$-sound and $\kappa_F<\rho_k^Q$.
Let $E$ be a  short extender over $P$ such that $(\kappa_F^+)^P<\kappa_E$.\footnote{We initially
had the stronger assumption that $\kappa_F<\kappa_E$ and $E$ is weakly amenably;
the referee noticed that it might be enough to assume that
$(\kappa_F^+)^P<\kappa_E$, which it is.}
Let $U=\Ult_k(Q,F)$ and $U^P_E=\Ult_0(P,E)$.
Suppose  $U$ and $U^P_E$ are wellfounded. Then
\[ (E\com F)\com_k Q= E\com_k(F\com_k Q), \]
\[ i^{Q,k}_{E\com F}=i^{U,k}_E\com i^{Q,k}_F.\qedhere\]\end{lem}
\begin{proof}
Let
$\lambda=\OR^P$. Then $\lambda\text{ is a }U\text{-cardinal, }\lambda<\rho_k^U\text{ and }
\her_\lambda^U=\univ{P}$. So $E$ is an extender over $\her_\lambda^U$, hence over $U$.
Let $j=i^{U,k}_E$. Note $j$ is continuous at $\lambda$, $j(\lambda)=\OR^{U^P_E}$ and
$\her_{j(\lambda)}^{\Ult_k(U,E)}=\univ{U^P_E}$ and $j\rest\her_\lambda^U=i^P_E$. By
\ref{lem:Dsp3},
\[ \nu_{E\com F}=\nu_E\un\sup i^P_E``\nu_F.\]
And letting $\xi=\max(\nu_F,\lgcd(P))$, for
$A\in P\inter\pow([\kappa_F]^{<\om})$,
\begin{equation}\label{eqn:assoc} i_{E\com
F}(A)\inter[\nu_{E\com
F}]^{<\om}=i^P_E(i_F(A)\inter[\xi]^{<\om})\inter[\nu_{E\com F}]^{<\om}.
\end{equation}

Define an isomorphism $\psi:\Ult_k(Q,E\com F)\to\Ult_k(U,E)$ by
\[ \psi([\tau_q,i^P_E(a)\un b]^{Q,k}_{E\com F})=
[\tau'_{(q',a)},b]^{U,k}_E, \]
for $k$-terms $\tau_q$ defined from parameter $q\in Q$, and $a\in[\nu_F]^{<\om}$,
$b\in[\nu_E]^{<\om}$, and where $q'=i^{Q,k}_F(q)$ and
$\tau'$ is defined from $\tau$ by converting the appropriate arguments to
parameters. Los' Theorem, (\ref{eqn:assoc}) and the fact that degree $k$
embeddings respect the $T_k$ predicate, show $\psi$ is
well-defined and $\rSigma_k$-elementary; surjectivity is clear. Moreover, $\psi$
commutes with the ultrapower embeddings, so $i^{Q,k}_{E\com
F}=i^{U,k}_E\com i^{Q,k}_F$.\end{proof}

\begin{cor}\label{cor:nass}
Let $P_n,\ldots,P_0,Q$ be segmented-premice. Suppose that for each $i$,
$E_i=F^{P_i}\neq\emptyset$,
$\crit(E_{i+1})>\crit(E_i)$,
$E_{i+1}$ is over $P_i$, $E_0$ is over
$Q$, and $\crit(E_0)<\rho^Q_k$. Then writing $Q_0=Q$ and $Q_{i+1}=\Ult_k(Q_i,E_i)$,
\begin{eqnarray*} ((\ldots(E_n\com E_{n-1})\com\ldots)\com E_0)\com_k
Q&=&E_n\com_k(\ldots\com_k(E_0\com_k Q));\\
 i^{Q,k}_{((\ldots(E_n\com E_{n-1})\com\ldots)\com E_0)}&=&i^{Q_n,k}_{E_n}\com\ldots\com 
i^{Q_0,k}_{E_0}.\end{eqnarray*}
\end{cor}

\begin{dfn}[Dodd core]\label{dfn:Dcore}
Let $G$ be an extender such that $(\sigma_G,s_G)$ are defined. The \emph{Dodd
core} of $G$, denoted
$\core_D(G)$, is the transitive collapse of $G\rest\sigma_G\un s_G$.\footnote{
That is, let $E_0=G\rest\sigma_G\cup s_G$,
let $s'=[s,\id]_{E_0}^M$ where $G$ is over $M$,
and let $j=i_{E_0}^M$. Then $\core_D(G)$ is the extender derived from $j$ with support $\sigma_G\cup s'$.}
We often identify $\core_D(G)$ with its trivial completion.
\end{dfn}

\begin{rem}\label{rem:Dcore}
Let $S$ be a $k$-sound premouse
such that every  $E\in\es_+^S$  is Dodd sound.
Let $\Ww$ be a $k$-maximal tree on $S$. Let $\alpha<\lh(\Ww)$ and
$G\in\es_+(M^\Ww_\alpha)$,  with $G$ not Dodd sound.

By elementarity, $G=F(M^\Ww_{\alpha})$, so $\alpha$ is the
unique $\alpha'$ such that $G\in\es_+(M^\Ww_{\alpha'})$. Write
$\alpha_G=\alpha$.
Lemmas \ref{lem:Dsp4}--\ref{lem:Dsp6} show $\deg^\Tt({\alpha_G})=0$ and
$\core_D(G)=F(M^{*\Ww}_{\beta+1})$, where
$\beta$ is
the least $\beta'$ such that (i) $\beta'+1\leq_\Ww{\alpha_G}$ and
$\Ww$ does not drop in model in $(\beta'+1,\alpha_G]_\Ww$, and (ii)
$F(M^\Ww_{\beta'+1})$
is not Dodd sound; condition (ii) can equivalently be replaced with (ii')
$\crit(i^{*\Ww}_{\beta'+1,{\alpha_G}})\geq\tau_F$ where
$F=F(M^{*\Ww}_{\beta'+1})$.

Also, $\gamma=\pred^\Ww(\beta+1)$ is the unique ordinal $\gamma'$ such that
$\core_D(G)\in\es_+(M^\Ww_{\gamma'})$. For $\Ww$ is normal and
$\lh(E^\Ww_\gamma)\leq\OR(M^{*\Ww}_{\beta+1})=\lh(\core_D(G))$.\footnote{Literally
here, we mean $\OR(M^{*\Ww}_{\beta+1})=\lh(G')$ where $G'$ is the trivial completion of $\core_D(G)$.}

For Dodd sound $G\in\es_+(M^\Ww_\alpha)$, $\alpha_G$ denotes the least
$\beta$ with $G\in\es_+(M^\Ww_\beta)$.
\end{rem}

\begin{dfn}[Dodd ancestry]\label{dfn:Dancestry}
Let $\Ww,\alpha$ be as in  \ref{rem:Dcore} and $G\in\es_+(M^\Ww_\alpha)$.
We define the
\emph{\textup{(}Dodd\textup{)}
ancestry of $G$ in $\Ww$}, denoted $\dam^\Ww(G)$, and the ordering
$<^\Ww_\dam$, recursively on $\alpha_G$. If $G$ is Dodd-sound let
$\dam^\Ww(G)=\emptyset$. Suppose $G$ is not Dodd-sound. Let $\gamma$ be such that
$\core_D(G)\in\es_+(M^\Ww_\gamma)$ (by \ref{rem:Dcore}, $\gamma$ is unique).
Define
$\dam^\Ww(G)$ to be the sequence $d=\left<d_\beta\right>_{\beta\in D}$ with domain
\[ D=\{\beta\mid\beta+1\in(\gamma,\alpha_G]_W\},\]
such that $d_\beta= \dam^\Ww(E^\Ww_\beta)$.

Now recursively define the relation $<^\Ww_\dam$ by:
\[ \beta<^\Ww_\dam\gamma\ \iff\ 
\ex\beta'\in\dom(\dam^\Ww(E^\Ww_\gamma))[\beta=\beta'\ \text{or}\
\beta<^\Ww_\dam\beta'].\]
We also write $E^\Ww_\beta<^\Ww_\dam E^\Ww_\gamma$ to mean
$\beta<^\Ww_\dam\gamma$.
\end{dfn}

Note that in \ref{dfn:Dancestry}, $G$ is Dodd-sound iff
$\dam^\Ww(G)=\emptyset$.

Figure \ref{fgr:Dancestry} presents a typical Dodd ancestry. An
extender $E$ is represented
by the symbol $\rfloor$, with $\crit(E)$ and $\lh(E)$ corresponding to the
lower and upper bounds of the symbol respectively. $E<_\dam F$ iff $E$ is
pictured to
the left of $F$, within its vertical bounds. So Dodd sound extenders have
no extenders to their left. In the figure, $G<_\dam H<_\dam J$ and $G<_\dam J$,
but $G,H\not<_\dam I$.
\begin{figure}
\[
\setlength{\unitlength}{1mm}
\begin{picture}(100,51)(24,0)

\put(100,0){\line(-1,0){2}}
\put(100,0){\line(0,1){50}}
\put(96,47){$J$}



\put(90,1){\line(-1,0){2}}
\put(90,1){\line(0,1){22}}
\put(86,20){$H$}

\put(90,24){\line(-1,0){2}}
\put(90,24){\line(0,1){8}}
\put(86,29){$I$}

\put(90,35){\circle*{0}}
\put(90,36){\circle*{0}}
\put(90,37){\circle*{0}}

\put(90,40){\line(-1,0){2}}
\put(90,40){\line(0,1){5}}

\put(90,46){\circle*{0}}
\put(90,47){\circle*{0}}
\put(90,48){\circle*{0}}



\put(80,2){\line(-1,0){2}}
\put(80,2){\line(0,1){5}}

\put(80,8){\line(-1,0){2}}
\put(80,8){\line(0,1){9}}
\put(76,14){$G$}


\put(80,18){\circle*{0}}
\put(80,19){\circle*{0}}
\put(80,20){\circle*{0}}


\put(80,25){\line(-1,0){2}}
\put(80,25){\line(0,1){5.5}}

\put(73,28){\circle*{0}}
\put(70,28){\circle*{0}}
\put(67,28){\circle*{0}}
\put(64,28){\circle*{0}}



\put(73,5){\circle*{0}}
\put(70,5){\circle*{0}}
\put(67,5){\circle*{0}}
\put(64,5){\circle*{0}}
\put(61,5){\circle*{0}}
\put(58,5){\circle*{0}}
\put(55,5){\circle*{0}}

\put(50,4){\line(-1,0){2}}
\put(50,4){\line(0,1){2}}

\put(73,11){\circle*{0}}
\put(70,11){\circle*{0}}
\put(67,11){\circle*{0}}
\put(64,11){\circle*{0}}

\put(60,10){\line(-1,0){2}}
\put(60,10){\line(0,1){2}}

\put(60,13){\circle*{0}}
\put(60,14){\circle*{0}}
\put(60,15){\circle*{0}}


\put(60,27){\line(-1,0){2}}
\put(60,27){\line(0,1){2}}


\end{picture}
\]
\caption{\label{fgr:Dancestry}\textit{Dodd ancestry of extender $J$}}
\end{figure}

\begin{lem}
Let $\Ww$, $G$ be as in \ref{dfn:Dancestry}.
Then:
\begin{itemize}
 \item[\textup{(}a\textup{)}] For all $\beta<^\Ww_\dam\gamma$ we have
$\crit(E^\Ww_\gamma)<\crit(E^\Ww_\beta)<\lh(E^\Ww_\beta)<\lh(E^\Ww_\gamma)$.
\item[\textup{(}b\textup{)}] Let $\gamma_1,\gamma_2\in\dom(\dam^\Ww(G))$, with
$\gamma_1<\gamma_2$. For
all $\gamma_1'\leq^\Ww_\dam\gamma_1$ and $\gamma_2'\leq^\Ww_\dam\gamma_2$, we
have $\gamma_1'<\gamma_2'$.
\item[\textup{(}c\textup{)}] Let $\gamma\in\dom(\dam^\Ww(G))$ and let
$\lambda=\pred^\Ww(\gamma+1)$.
Suppose that $E^\Ww_{\gamma}=F(M^\Ww_{\gamma})$. Then $\Ww$ drops in model at
some $\beta\in(0,\gamma]_\Ww$ such that $\beta>\lambda$.
\end{itemize}
\end{lem}
\begin{proof}[Proof Sketch] We omit the proof.
Parts (a),(b) use \ref{rem:Dcore} and the normality
of $\Ww$. Part (c) is mostly similar to part of the proof of closeness,
\cite[6.1.5]{fsit}.\end{proof}

\begin{dfn}[Dodd decomposition]\label{dfn:Ddecomp}
Let $\Ww$, $G$ be as in \ref{dfn:Dancestry}.
Let $\mathcal{E}=\{\core_D(E^\Ww_\beta)\mid\beta\leq_\dam\alpha_G\}$.
The \emph{Dodd decomposition of $G$} is the sequence of extenders
enumerating $\mathcal{E}$
in order of increasing critical point. 
\end{dfn}

\begin{dfn}[Core sequence]\label{dfn:coreseq}
Let $P,Q$ be premice, $j:P\to Q$ a $k$-embedding.
We define the \emph{degree $k$ core
sequence} $\left<Q_\alpha,j_\alpha\right>_{\alpha\leq\lambda}$ of $j$.
Set $Q_0=P$ and $j_0=j$. Let $j_\alpha:Q_\alpha\to Q$ be given.
If $j_\alpha=\id$ or is Dodd-inappropriate
set $\lambda=\alpha$. Otherwise let $(s,\sigma)=(s_{j_\alpha},\sigma_{j_\alpha})$ and $H_\alpha=\rg(j_\alpha)$. Set
\[ Q_{\alpha+1}=\cHull_{k+1}^{Q}(H_\alpha\un
s\un\sigma)\] 
and $j_{\alpha+1}:Q_{\alpha+1}\to Q$
the uncollapse. Take
direct limits at limits $\alpha$. Since $H_\alpha\psub H_\beta$ for
$\alpha<\beta$, the process terminates.
\end{dfn}

\begin{lem}\label{lem:Ddecomp}
Let $\Ww,G$ be as in \ref{dfn:Dancestry} with $\Ww$ being $k$-maximal on $S$. Let
$\Ff=\left<F_\alpha\right>_{\alpha<\lambda}$ be the Dodd decomposition of $G$,
and $G_\alpha = \Ult_0(\id,\Ff\rest\alpha)$.

Then $G_1=\core_D(G)$ and $G_\lambda=G$.

Let $\alpha\in[1,\lambda]$.
Then $G_\alpha=F^{N_\alpha}$ for some premouse $N_\alpha$,
and if $\alpha<\lambda$ then $F_\alpha$ is close to $N_\alpha$ and
$\rho_1^{N_\alpha}\leq\crit(F_\alpha)$.
Moreover, for all $\alpha$, there is a $k$-maximal tree $\Ww_\alpha$ on $S$ and
$\zeta\leq\alpha_G$ and $m\in\om$ such that\textup{:}
\begin{enumerate}[label=\tu{(}\roman*\tu{)}]
\item $\Ww_\alpha\rest\zeta+1=\Ww\rest\zeta+1$,
\item $\lh(\Ww_\alpha)=\zeta+m+1$,
\item $G_\alpha\in\es_+(\fin^{\Ww_\alpha})$,
\item $\lh(E^{\Ww_\alpha}_\beta)<\lh(G_\alpha)$ for
all $\beta+1<\lh(\Ww_\alpha)$,
\item\label{item:crit(G)<nu_zeta} if $\zeta<\alpha_G$ then $\crit(G)<\nu^{\Ww}_\zeta$, and 
\item if $m>0$ then $\crit(G)<\nu^{\Ww_\alpha}_{\zeta}$.
\end{enumerate}
\end{lem}

\begin{proof}
We prove the lemma by induction on $\alpha_G$ for $\alpha_G+1<\lh(\Ww)$, with a
subinduction on $\alpha\in[1,\lambda]$. So assume it holds for
$H=E^\Ww_{\alpha_H}$ for each $\alpha_H<\alpha_G$.

If $G$ is Dodd sound then $\lambda=1$, and we use $\zeta=\alpha_G$
and $m=0$.

Suppose $G$ is not Dodd sound. So $\lambda>1$. We use \ref{rem:Dcore} in the
following.

Suppose $\alpha=1$. We have $G_1=\core_D(G)$. 
Set $m=0$ and $\zeta$ with $G_1\in\es_+(M^\Ww_\zeta)$ and $\beta+1$ such that $\zeta=\pred^\Ww(\beta+1)\leq_\Ww\alpha_G$.
Then \ref{item:crit(G)<nu_zeta} holds because
\[ \crit(G)=\crit(G_1)<\tau_{G_1}\leq\crit(i^{*\Ww}_{\beta+1,\alpha_G}
)<\nu^\Ww_\zeta.\]

Suppose $\alpha=\beta+1>1$. So $G_{\beta+1}=\Ult_0(G_\beta,F_\beta)$. Let
$\beta'$ be such that
$F_\beta=\core_D(E^\Ww_{\beta'})$.
Let $\gamma'\leq_\dam^\Ww\alpha_G$ be such
that $\beta'\in\dom(\dam^\Ww(E^\Ww_{\gamma'}))$. Let
$\eps'=\pred^\Ww(\beta'+1)$.
Let $\zeta$ be such that $F_\beta\in\es_+(M^\Ww_\zeta)$. Then define
$m\in[1,\om)$ and $\Xx=\Ww_{\beta+1}$ by setting
$E^\Xx_\zeta=F_\beta$, and
$E^\Xx_{\zeta+i}=F(M^\Xx_{\zeta+i})$ for $1\leq i<m$,
until we reach $M^\Xx_{\zeta+m}$ with active extender $F$ with $\crit(F)=\crit(G)$. Each
$E^\Xx_{\zeta+i}$ is a
sub-extender of some $G'\leq_\dam^\Ww G$, and applies to the same premouse
in $\Xx$ as does $G'$ in $\Ww$. So
$\crit(E^\Xx_{\zeta+i+1})<\crit(E^\Xx_{\zeta+i})$.

For
instance, $\pred^\Xx(\zeta+1)=\eps'$. For by \ref{rem:Dcore} and normality of
$\Ww$,
\[ \crit(F_\beta)=\crit(E^\Ww_{\beta'})<\min(\nu^\Ww_\zeta,\nu^\Ww_{\eps'}),\]
and for each
$\delta<\zeta$, we have $\lh(E^\Ww_\delta)\leq\nu^\Ww_\zeta<\lh(F_\beta)$.
In particular, $\zeta\geq\eps'$. But
$\Ww\rest\zeta+1=\Xx\rest\zeta+1$, so
$\pred^\Xx(\zeta+1)=\eps'$. If $\gamma'=\alpha_G$ then $m=1$. Otherwise
$m>1$, and $E^{\Xx}_{\zeta+1}=F(M^{\Xx}_{\zeta+1})$ is a subextender of
$F(M^\Ww_{\beta'+1})$, and each apply
to the same premouse, etc.

We claim that $\Ww_{\beta+1}=\Xx$ is as desired. Certainly
$\Xx$ is normal, and its models are wellfounded, since they embed
into models of $\Ww$.
Now let $H_0=F_\beta$ and let $H_{i+1}=F(M^{*\Xx}_{\zeta+i+1})$
for $i<m$. Then using \ref{cor:nass}, $F(M^\Xx_{\zeta+m})$ is
\begin{equation}\label{eqn:Hiassoc} (\ldots((H_0\com
H_1)\com
H_2)\com\ldots\com H_m) = H_0\com H_1\com\ldots\com H_m. \end{equation}
Let $\alpha_D$ be such that $\core_D(G)\in\es_+(M^\Ww_{\alpha_D})$.
Let $\iota=\pred^\Xx(\zeta+m)$. So
\[ \iota=\max([\alpha_D,\alpha_G)_\Ww\inter\beta'+1). \]
Let
$X=\{\delta'\mid\delta'+1\in(\alpha_D,\iota]_\Ww\}$.
Then
\begin{equation}\label{eqn:H_m}H_m=\Ult_0(\core_D(G),\left<E^\Ww_{
\delta'}\right>_ {
\delta'\in X}).
\end{equation}
By induction, for $\delta'\in X$,
$E^\Ww_{\delta'}=\Ult_0(F_\delta,\Ff\rest(\delta,\eps))$, where
$\Ff\rest[\delta,\eps)$ is the Dodd decomposition of $E^\Ww_{\delta'}$ (so $F_\delta=\core_D(E^\Ww_{\delta'})$).
But then
\[
\Ult_0(M^{*\Ww}_{\delta'+1},E^\Ww_{\delta'})=\Ult_0(M^{*\Ww}_{\delta'+1},\Ff\rest[\delta,
\eps)),\]
 and the ultrapower maps agree. This fact is a straightforward extension of
\ref{lem:extass}; we omit the details.
Therefore $H_m=\Ult_0(\core_D(G),\Ff\rest[1,\xi))$,
where $\Ff\rest[1,\xi)$ is the concatenation of the Dodd decompositions of $E^\Ww_{\delta'}$ for $\delta'\in X$, and
the ultrapower map is that corresponding to (\ref{eqn:H_m}). We get a similar
representation for each $H_i$, thus partitioning
$\Ff\rest[1,\beta]$ into $m+1$ intervals. Finally, using
(\ref{eqn:Hiassoc}) and the extension of \ref{lem:extass}, then
$F(M^{\Xx}_{\zeta+m})$ is
\[ \Ult_0(\core_D(G),\Ff\rest[1,\beta+1))=G_{\beta+1}. \]
We leave the remaining details and limit case to the reader.
\renewcommand{\qedsymbol}{$\Box$(\ref{lem:Ddecomp})}\end{proof}

\begin{lem}\label{lem:Dcoreseq}
Adopt the hypotheses and notation of \ref{lem:Ddecomp}.
Let $P$ be an $m$-sound seg-pm, with $G$ over
$P$ and $\crit(G)<\rho_m^P$. Suppose that $\Ult_m(P,G)$ is wellfounded.
Then for $\alpha\leq\lambda$, we have
$\Ult_m(P,G_\alpha) = \Ult_m(P,\Ff\rest\alpha)$
and $i^{P,m}_{G_\alpha}=i^{P,m}_{\Ff\rest\alpha}$, and the degree $m$ core sequence of
$i^{P,m}_G$ is
\[
\left<Q_\alpha=\Ult_m(P,\Ff\rest\alpha),i^{Q_\alpha,m}_{\Ff\rest[\alpha,\lambda)}
\right>_ {
\alpha\leq\lambda}.\]
In particular, $i^{Q_\alpha,m}_{\Ff\rest[\alpha,\lambda)}$ is not superstrong.

Suppose that $P$ is $(m+1)$-solid.
Then $i^{P,m}_{\Ff\rest\alpha}$ and
$i^{Q_\alpha,m}_{\Ff\rest[\alpha,\lambda)}$ are $m$-embeddings,
preserve $p_{m+1}$, and are cofinal at $\rho_{m+1}$.
\end{lem}

\begin{proof}[Proof Sketch]
For the (inductively established) characterization of the core sequence of $i_G^{P,m}$,
given $Q_\alpha$ and factor embedding $j_\alpha=i^{Q_\alpha,m}_{\Ff\rest[\alpha,\lambda)}$ as above, note that
\[ \Ult_m(P,G)=\Ult_m(Q_\alpha,\Ff\rest[\alpha,\lambda)), \]
and show that the natural factor embedding
\[ k:\Ult_m(Q_\alpha,F_\alpha)\to\Ult_m(P,G)\]
maps the maximal fragments of $F_\alpha$ in $\Ult_m(Q_\alpha,F_\alpha)$ (corresponding to $(s_{F_\alpha},\sigma_{F_\alpha})$) to those of
the extender derived from $i^{Q_\alpha,m}_{\Ff\rest[\alpha,\lambda)}$. This follows the argument for \ref{lem:Dsp5}, and
that $\sigma_{F_\alpha}<\crit(F_{\beta})$ when $\alpha<\beta$.
For the second paragraph, use \ref{cor:k+1solid} and commutativity.
\end{proof}

\begin{newdfn}\label{dfn:potentialtree} Let $\Ww$ be a $k$-maximal tree on a
$k$-sound premouse $N$, of length $\zeta+1$. Let $P=M^\Ww_\zeta$ and
$\gamma\leq\OR^P$, with $\lh(E^\Ww_\alpha)<\gamma$ for every
$\alpha+1<\lh(\Ww)$. Let $E$ be an extender such that $(P||\gamma,E)$ is a
premouse.

The \emph{potential $k$-maximal tree}
$\Ww\conc\left<E\right>$ is the ``putative iteration tree'' $\Ww^+$ on $N$ of
length $\zeta+2$, extending $\Ww$, with $E^{\Ww^+}_\zeta=E$, and
with
$\pred^{\Ww^+}(\zeta+1)$, etc, determined by the rules for
$k$-maximality. We say the tuple $(k,\Ww,\zeta,P,E,\Ww^+)$ is \emph{potential
for $N$}.

As usual $\Phi(\Ww)$ denotes the phalanx associated to
$\Ww$ (see \cite{cmip}). 
Let $(k,\Ww,\ldots)$ be potential for $N$. Let $U=M^{\Ww^+}_{\zeta+1}$ and $d=\deg^{\Ww^+}(\zeta+1)$.
Let $\iota=\lgcd(P||\lh(E))$.
We write $\Phi(\Ww,\iota,E)$ for the phalanx
\[ \left<(\Phi(\Ww),{<\iota}),(U,d),\lh(E)\right>.\]
Normal trees $\Uu$ on $\Phi(\Ww,\iota,E)$, indexed with ordinals $\geq\zeta+1$
(the first model of $\Uu$ is $M^\Uu_{\zeta+1}$), must satisfy the usual conditions for $k$-maximality,
except/including that (i) $M^\Uu_{\zeta+1}=U$, (ii) $\lh(E)<\lh(E^\Uu_{\zeta+1})$,
(iii) if $\crit(E^\Uu_\alpha)<\iota$ then $\pred^\Uu(\alpha+1)=\delta\leq\zeta$
and $M^{*\Uu}_{\alpha+1}\ins M^\Ww_\delta$ and $\deg^\Uu(\alpha+1)$ are as usual (like for $\Phi(\Ww)$), and
(iv) if $\crit(E^\Uu_\alpha)=\iota$
then $\pred^\Uu(\alpha+1)=\zeta+1$ and $M^{*\Uu}_{\alpha+1}=U$ and $\deg^\Uu(\zeta+1)=d$.
\end{newdfn}

We now  establish a phalanx iterability criterion guaranteeing that an extender is on the sequence of a premouse.
The lemma should be compared with \cite[8.6]{cmip}
(to which it is very similar); there, the exchange ordinal
is $\nu_E$ instead of $\iota$. The lower exchange ordinal leads to our need to appeal to the
Dodd-structure analysis in the proof. (And in our application of the lemma later, our iterability proof
only seems to give iterability with respect to $\iota$.)

\begin{newlem}\label{lem:extendermaximality} Let $M$ be an $\om$-sound premouse
projecting to $\om$. Suppose that every $E\in\es_+^M$ is Dodd sound. Let $\cc<\OR^M$ and let $\Sigma$ be an above-$\cc$,
$(\om,\om_1+1)$-strategy for $M$. Let
\[ (\om,\Ww,\zeta,P,E,\Ww^+)\]
be potential for $M$, with $\Ww$ via $\Sigma$.  Let
$\iota=\lrgcrd(P||\lh(E))$. Suppose that $\cc<\lh(E)$ and $\cc$ is a
cutpoint of $P|\lh(E)$.\footnote{Note we don't assume $\cc\leq\crit(E)$. But if $P|\lh(E)$ is active then $\cc\leq\crit(F^{P|\lh(E)})$.}

Then $E\in\es_+^P$ iff the phalanx $\pP=\Phi(\Ww,\iota,E)$ is normally $(\om_1+1)$-iterable.\end{newlem}

\begin{proof}
We have $\zeta+1=\lh(\Ww)$ and $P=M^\Ww_\zeta$ and $(P||\lh(E),E)$ is a premouse. Let 
$\delta=\pred^{\Ww^+}(\zeta+1)$ and
$Q=M^{*\Ww^+}_{\zeta+1}\ins M^\Ww_\delta$ and $m=\deg^{\Ww^+}(\zeta+1)$ and
\[ i_E=i^{Q,m}_E:Q\to \Ult_m(Q,E)=M^{\Ww^+}_{\zeta+1}.\]

The proof that $\pP$ is iterable assuming $E\in\es_+^P$ we mostly
leave to the reader since we won't use this fact. The point is
that
$\Phi(\Ww^+)=\Phi(\Ww,\nu(E),E)$ is iterable since $M$ is iterable above $\cc$, and there
is a direct correspondence between normal trees $\Uu$ on $\pP$ and normal trees $\Vv$ on
$\Phi(\Ww^+)$. (For example when $E$ is type 2:
Let $\beta+1<\lh(\Uu)$ with $\pred^\Uu(\beta+1)=\zeta+1$ and $\crit(E^\Uu_\beta)=\iota$.
Then $M^{*\Uu}_{\beta+1}=M^{\Ww^+}_{\zeta+1}$ and $\deg^\Uu(\beta+1)=m$, whereas
$\pred^\Vv(\beta+1)=\zeta$ and $M^{*\Vv}_{\beta+1}=P|\lh(E)$ and $\deg^\Vv(\beta+1)=0$.
It follows that 
\[ M^\Uu_{\beta+1}=\Ult_m(Q,F(M^\Vv_{\beta+1})),\]
and in particular, $M^\Uu_{\beta+1}$ is wellfounded.)
 
So assume $\pP$ is iterable. In comparison of 
$\pP$ vs $\Phi(\Ww)$, the resulting trees are above $\cc$
(but we allow $\crit(E)<\cc$).
For if $P|\lh(E)$ is active then $\crit(F^{P|\lh(E)})\geq\cc$ because $\cc$ is a cutpoint of $P|\lh(E)$.
So suppose there is an active premouse $N$ such that $P||\lh(E)\pins N$
and $\lh(E)$ is a cardinal of $N$ and $\crit(F^N)<\cc$.
Then $\crit(F^N)=\iota<\cc$, by the ISC and since $\cc$ is a cutpoint of $P|\lh(E)$.
So $\iota$ is inaccessible in $P|\lh(E)$, so $E$ is type 2 or 3,
so
\[ \crit(E)\text{ is }{<\iota}\text{-strong in }P||\lh(E)\text{, as witnessed by }
 \es^{P||\lh(E)}.\]
The ultrapower map $j:P||\lh(E)\to \Ult(P||\lh(E),F^N)$
preserves this statement. But then by the ISC,
$P||\lh(E)$ does not have a cutpoint above $\iota$, contradiction.

So we get a successful $\om$-maximal comparison $(\Xx',\Yy')$,
with $\Yy'$ such that $\Ww\conc\Yy'$ is via $\Sigma$. Write
\[ \Xx=\Ww^+\conc\Xx'\text{ and }\Yy=\Ww\conc\Yy',\]
so $\Yy$ is an $\om$-maximal tree on $M$ with $\Yy\rest(\zeta+1)=\Ww$,
and $\Xx\rest(\zeta+2)=\Ww^+$ and $\Xx$ is $\om$-maximal except that $E^\Xx_\zeta=E$
so maybe $E\notin\es_+(M^\Xx_\zeta)$,
and maybe the exchange ordinal $\iota<\nu_E$.
So we can't apply the Closeness Lemma \cite[6.1.5]{fsit} to $\Xx$. But
for all $\alpha+1<\lh(\Xx)$, $E^\Xx_\alpha$ is weakly amenable over
$M^{*\Xx}_{\alpha+1}$. Let $\Qfin=\fin^\Xx$. Clearly
$\Qfin\ins\fin^\Yy$. By \ref{cor:k+1solid},
$\Qfin$ is unsound, so $\Qfin=\fin^\Yy$. Let $\alpha+1\in b^\Xx$ be
least such that $(\alpha+1,\infty]_\Xx$ does not drop in model or degree,
and let $n=\deg^\Xx(\alpha+1)$.
Since $\Qfin=\fin^\Yy$, therefore $\Qfin$ is $n+1$-solid, so by
\ref{cor:k+1solid},
\[ p_{n+1}^\Qfin=i^{*\Xx}_{\alpha+1,\infty}(p_{n+1}(M^{*\Xx}_{\alpha+1})) \]
and $\rho_{n+1}^\Qfin=\rho_{n+1}(M^{*\Xx}_{\alpha+1})$. We have $\Ww=\Yy\rest(\zeta+1)$.
Standard
arguments now show that $b^\Xx$ is above $U$ and $b^\Xx$ does
not drop in model or degree above $U$ (that is, $\zeta+1\in b^\Xx$ and $(\zeta+1,\infty]_\Xx$ does not drop in model or degree),
so
\[ j^\Xx\eqdef i^{*\Xx}_{\zeta+1,\infty}:U\to \Qfin.\]
Similar arguments show that
$\Yy\neq\Ww$, and letting
$\beta+1=\min(b^\Yy\cut(\zeta+1))$, that $\delta=\pred^\Yy(\beta+1)$,
$M^{*\Yy}_{\beta+1}=\Rstar=M^{*\Ww^+}_{\zeta+1}$
and
$\deg^\Yy(\beta+1)=m$
and $(\beta+1,\infty]_\Yy$ does not drop in model or degree, so
\[ j^\Yy=i^{*\Yy}_{\beta+1,\infty}:Q\to Z.\]
We have $j^\Xx\com i_E=j^\Yy$, since these maps preserve $p_{m+1}$ and do not move the generators that generate $\Rstar$.
If $\nu_E\leq\crit(j^\Xx)$, standard arguments now show
$E\in\es_+^P$. So assume $E$ is type 2, with largest generator $\gamma$, and
\[ \crit(j^\Xx)=\iota\leq\gamma<(\iota^+)^U=\lh(E).\]
Let $G=E^\Yy_\beta$. So $G\rest\nu_G\sub E_{j^\Yy}$.

\setcounter{clm}{0}

\begin{clm}{\label{clm:1extG}} $G$ is the only
extender used on $b^\Yy$, and $\nu_G=j^\Xx(\gamma+1)$.
\begin{proof}
Let
$\sigma\in[\zeta+1,\infty]_\Xx$ be least such that
\[ \sigma=\infty\text{ or }\crit(i^\Xx_{\sigma\infty})>i^\Xx_{\zeta+1,\sigma}(\gamma).\]
Let $(\iota_1,\gamma_1)=i^\Xx_{\zeta+1,\sigma}(\iota,\gamma)$. Now $E\rest\iota\in U$
and
\[ E_{j^\Yy}\rest\iota_1=j^\Xx(E\rest\iota)\in\Qfin,\]
since $\crit(E)<\crit(j^\Xx)$ and $j^\Xx\com
i_E=j^\Yy$. It follows that $\iota_1<\nu_G$
and $G\rest\iota_1\in\Qfin$.
Since $p_{m+1}^\Qfin=j^\Xx\com i_E(p_{m+1}^\Rstar)$, 
$\Qfin$
has the $(m+1)$-hull property at $\gamma_1+1$, i.e.,
\begin{equation}\label{eqn:hullprop}
\Hull_{m+1}^\Qfin((\gamma_1+1)\un\{p_{m+1}^\Qfin\}
)\inter\pow(\gamma_1+1)=\Qfin\inter\pow(\gamma_1+1) \end{equation}
(cf.~\cite[Example 4.3 and following Remark]{cmip}).
But $\gamma_1<((\iota_1)^+)^\Qfin$ and $\gamma_1$ is a generator of $G$. For
if $f:[\mu]^{<\om}\to\mu$ with $f\in \Rstar$, then
\[ \gamma\notin i_E(f)``[\gamma]^{<\om}\implies
\gamma_1\notin
j^\Xx(i_E(f))``[\gamma_1]^{<\om}\implies
\gamma_1\notin
i^{Q,m}_G(f)``[\gamma_1]^{<\om}.\]
Therefore $\gamma_1+1=\nu_G$.
It follows
that
\[ M^\Xx_\sigma=\Hull_{m+1}^\Qfin((\gamma_1+1)\un\{p_{m+1}^\Qfin\})=M^\Yy_{\beta+1
},\]
but then in fact $M^\Xx_\sigma=\Qfin=M^\Yy_{\beta+1}$, proving the claim.
\renewcommand{\qedsymbol}{$\Box$(Claim \ref{clm:1extG})}\end{proof}\end{clm}

By Claim \ref{clm:1extG}, $E$ is a subextender of $G$. We will refine this
observation, using the Dodd structure analysis.
Now \ref{lem:Ddecomp},
\ref{lem:Dcoreseq} apply to $\Yy,G$. Let
$\Ff=\left<F_\alpha\right>_{\alpha<\lambda}$ be the Dodd decomposition of $G$.
Let \[G_\alpha=\Ult_0(\id,\Ff\rest\alpha).\]
Let
$\left<\Qfin_\alpha,j_\alpha\right>_{\alpha\leq\lambda}$ be the degree $m$ core
sequence of $i^{Q,m}_G=j^\Yy$. For $\alpha\leq\beta$ let
$j_{\alpha\beta}=j_\beta^{-1}\com j_\alpha$.
 By \ref{lem:Dcoreseq},
\[\Qfin_\alpha=\Ult_m(\Rstar,\Ff\rest\alpha),\]
\[ j_{\alpha\beta}=i^{Z_\alpha,m}_{\Ff\rest[\alpha,\beta)}\text{ and }j_\alpha=i^{Z_\alpha,m}_{\Ff\rest[\alpha,\lambda)},\]
and $j_{\alpha\beta},j_\alpha$ are $m$-embeddings which preserve $p_{m+1}$.

\begin{clm}\label{clm:factor} There is $\eps\leq\lambda$ such that
$\Qfin_\eps=U$ and $G_\eps=E$.\end{clm}

\begin{proof}\setcounter{case}{0}
We will inductively define $m$-embeddings
$i_\alpha:\Qfin_\alpha\to U$
such that (see Figure \ref{fgr:commuting} for a partial summary):
\begin{enumerate}[label=--]
 \item $j_\alpha=j^\Xx\com i_\alpha$ (so $i_\alpha$ preserves $p_{m+1}$),
 \item $i_\beta\com j_{\alpha\beta}=i_\alpha$ for $\alpha\leq\beta$,
 \item if $\alpha>0$ then $\iota,\gamma\in\rg(i_\alpha)$,
 \item if $\Qfin_\alpha\neq U$ or $i_\alpha\neq\id$ then $\crit(i_\alpha)<\iota$ (so if $\alpha>0$ then $i_\alpha(\crit(i_\alpha))\leq\iota$).
\end{enumerate}

\begin{figure}
\[
\begin{picture}(100,75)(50,0)
\put(46,0){$\Qfin_\alpha$}
\put(58,11){\vector(1,4){15}} 
\put(61,11){\vector(1,1){15}} 
\put(62,8){\vector(3,1){70}} 
\put(75,30){$\Qfin_\beta$}
\put(75,75){$U$}
\put(135,30){$\Qfin$}
\put(80,40){\vector(0,1){30}} 
\put(90,35){\vector(1,0){37}} 
\put(85,70){\vector(3,-2){47}} 
\put(52,40){$\scriptstyle i_{\alpha}$}
\put(95,10){$\scriptstyle j_{\alpha}$}
\put(109,58){$\scriptstyle j^\Xx$}
\put(84,51){$\scriptstyle i_{\beta}$}
\put(100,40){$\scriptstyle j_{\beta}$}
\end{picture}
\]
\caption{\label{fgr:commuting}\textit{Commuting maps for $\alpha\leq\beta$}}
\end{figure}
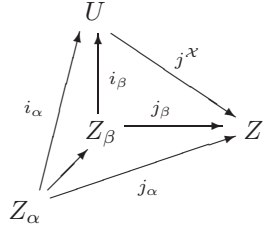

\begin{case} $\alpha=0$.

We have $\Qfin_0=\Rstar$ and $j_0=j^\Yy$. Set $i_0=i_E$.\end{case}

\begin{case}\label{case:alpha=1} $\alpha=1$.

Observe first that $E\rest\gamma\in U$. This follows the ISC if
$E\rest\gamma$ is not type Z, so suppose otherwise.\footnote{We
can't quote \cite{deconstruct} here since we don't know that $(P||\lh(E),E)$
is iterable.} Then $\iota$ is the largest generator of $E\rest\gamma$, and
\[ E\equiv_{\iota,\{\iota\},\{\iota_1\}}G.\]
But
$G\rest\gamma_1\in \Qfin$, by \cite{deconstruct}, so
$G\rest\iota\un\{\iota_1\}\in \Qfin$. Since $\pow(\iota)\inter
\Qfin=\pow(\iota)\inter U$,
then
$E\rest\iota+1\in U$, so $E\rest\gamma\in U$ as required.

Let $s=s_E$ and $\sigma=\sigma_E$. Then $\gamma=\max(s)$ since $\nu_E=\gamma+1$ and $E\rest\gamma\in
U$. Also, $\sigma\leq\iota$ since $E\approx
E\rest\iota\un\{\gamma\}$. So as above, $G\rest\sigma\un j^\Xx(s)\notin
\Qfin$,
but for each $X$ such that $E\rest X\in U$, we have $G\rest
j^\Xx(X)=j^\Xx(E\rest X)\in \Qfin$. Therefore $s_G=j^\Xx(s)$ and
$\sigma_G=\sigma$.
It follows that
\[ \Qfin_1=\cHull_{m+1}^U(\sigma\un\{s\}\un\{p_{m+1}^U\}).\] Let
$i_1:\Qfin_1\to U$ be the uncollapse. Then $\gamma=\max(s)\in\rg(i_1)$, so $\iota\in\rg(i_1)$.

Now suppose $\iota\leq\crit(i_1)$. Since
$\gamma\in\rg(i_1)$, then $\gamma<\crit(i_1)$. Since $i_1$ is an $m$-embedding preserving
$p_{m+1}$, therefore $U\sub\rg(i_1)$, so $\Qfin_1=U$ and $i_1=\id$.\end{case}

\begin{case} $\alpha=\beta+1>1$.

Suppose $\Qfin_\beta\neq U$ and $i_\beta:\Qfin_\beta\to U$ with
\[ \kappa=\crit(i_\beta)=\crit(j_\beta)=\crit(F_\beta)<\iota.\]
Now $F_\beta$ is over $\Qfin_\beta$, $F_\beta$ is on the extender sequence of a
premouse (by \ref{rem:Dcore}) and $\Qfin_\beta\inter\pow(\kappa)=Z\inter\pow(\kappa)=U\inter\pow(\kappa)$.
Therefore $\kappa$ is inaccessible in $\Qfin_\beta$
and in $U$. Since $\iota\in\rg(i_\beta)$ we have
$\kappa'=i_\beta(\kappa)\leq\iota$.

Now $E_{i_\beta}\notin U$, for
otherwise $j^\Xx(\kappa')=j_\beta(\kappa)$ and
$E_{j_\beta}=j^\Xx(E_{i_\beta})\in
\Qfin$, contradicting \ref{lem:Dcoreseq}.
So $i_\beta$ is Dodd-appropriate. 
Then $j^\Xx(s_{i_\beta})=s_{j_\beta}$ and $\sigma_{i_\beta}=\sigma_{j_\beta}$ because $\sigma\leq\kappa'\leq\iota$ and $j^\Xx\com
i_\beta=j_\beta$. (In fact, therefore $s_{i_\beta}=s_{j_\beta}$ as $s_{i_\beta}\sub\iota$.) Now proceed as in
Case \ref{case:alpha=1}.\end{case}

\begin{case} $\alpha$ is a limit.

This case follows from the commutativity of the maps before stage
$\alpha$.\end{case}

This defines all $i_\alpha$. Now since $\Qfin_\lambda=\Qfin$ and
$j_\lambda=\id$,
there is $\eps\leq\lambda$ such that $\Qfin_\eps=U$ and $i_\eps=\id$, so
$i^{Q,m}_{G_\eps}=i_E$, so $G_\eps=E$.\renewcommand{\qedsymbol}{$\Box$(Claim
\ref{clm:factor})}\end{proof}

Fix $\eps$ as in Claim \ref{clm:factor}. Let
$\Ww_\eps$ be as in
\ref{lem:Ddecomp}. Then
$E=G_\eps\in\es_+(\fin^{\Ww_\eps})$. But $\Ww$ and $\Ww_\eps$
are both
normal trees via $\Sigma$, using only extenders $F$ with
$\lh(F)<\lh(E)$, and $\fin^\Ww||\lh(E)=\fin^{\Ww_\eps}||\lh(E)$.
Therefore
$\Ww=\Ww_\eps$, and $E\in\es_+^P$, as required.
\renewcommand{\qedsymbol}{$\Box$(\ref{lem:extendermaximality})}\end{proof}

\section{Extender maximality}\label{sec:cohering}
Let $N$ be a $(k+1)$-sound mouse with $\om<\rho_{k+1}^N$. Our proofs of Theorems
\ref{thm:easy_coh}-\ref{thm:cohering} require the formation 
of $(k+1)$-sound,
$\rSigma_{k+1}$-elementary, proper hulls of $N$, containing a
given parameter. The following lemma helps with this.

\begin{newlem}\label{lem:sound_hull}
Let $N,n$ be such that $N$ is an
$(n+1)$-sound premouse and either
\begin{enumerate}[label=\tu{(}\roman*\tu{)}]
 \item $N$ is $(n,\om_1,\om_1+1)$-iterable, or
 \item $N$ is $(n+3)$-sound and $(n+3,\om_1+1)$-iterable.
\end{enumerate}

Let $\theta\in[\om,\rho_{n+1}^N)$ be an $N$-cardinal and
$x\in\core_0(N)$. Then $\exists q\in\core_0(N)$ such that letting
\[ H=\Hull_{n+1}^N(\theta\un\{q\}) \]
and $M=\trcoll(H)$ and $\pi:M\to N$ be the
uncollapse, then $x\in H$, $M\pins N$, $\rho_{n+1}^M=\theta$, $q=\pi(p_{n+1}^M)$, and $p_{n+1}^N=q\cut\alpha$
for some $\alpha$.\footnote{Given $N,n,x,\theta,q$ as in \ref{lem:sound_hull},
we say that $q$ witnesses
\ref{lem:sound_hull} with respect to $(N,n,x,\theta)$.}
\end{newlem}
\begin{proof}We may assume $N$ is countable, and if (ii) holds, then by replacing $N$
with $\Hull_{n+4}^N(\emptyset)$, we may assume $\rho_{n+4}^N=\om$ and $N$
is $(n+4)$-sound.
For the assertion in the second paragraph of the lemma $\rPi_{n+4}$, given $\rho_{n+3}>\om$.
This uses that $\{p_{n+1}^N\}$ is an $\rPi^N_{n+3}$
singleton. Also, $\{u_n^N\}$ is an $\rPi^N_{n+2}$ singleton.
We prove both of these facts by induction on $n$. Suppose that $\{u_n^N\}$ is $\rPi^N_{n+2}$;
we show that $\{p_{n+1}^N\}$ is $\rPi^N_{n+3}$. By the induction hypothesis,
we may use $u=u_n^N$ as a parameter. But $p_{n+1}^N$ is the unique $p\in[\OR(\core_0(N))]^{<\om}$
such that
\begin{enumerate}[label=(\alph*)]
 \item\label{item:p_n+1-solid_for_N}
$p$ is $(n+1)$-solid for $N$,
and \item\label{item:N_is_Hull_n+1^N} $N=\Hull_{n+1}^N(\rho_{n+1}^N\cup\{p\})$
\end{enumerate}
(note that in both \ref{item:p_n+1-solid_for_N} and \ref{item:N_is_Hull_n+1^N}, $u$ is an implicit parameter).
Condition \ref{item:p_n+1-solid_for_N} is $\rSigma^N_{n+2}(\{u\})$,
because we only need to assert the existence of generalized solidity witnesses.
And \ref{item:N_is_Hull_n+1^N} holds iff for every $x\in\core_0(N)$ there is $\alpha\in\OR^N$
such that $x\in\Hull_{n+1}^N(\{p,\alpha\})$ and there is a generalized $(n+1)$-witness
for $p\cup(\alpha+1)$ (apply the latter condition with $x=p_{n+1}^N$ and note that $\alpha<\rho_{n+1}^N$);
this is $\rPi^N_{n+3}(\{u\})$. It follows easily that $\{u_{n+1}^N\}$
is also $\rPi^N_{n+3}$.

For notational simplicity, we assume $n=1$, but the proof easily generalizes. Let $w_0$ be the set of
$2$-solidity witnesses for $p_2^N$. Let $p\in(\rho_2^N)^{<\om}$ be $<_\lex$-least
such that
\[ w_0,x\in\Hull_{2}^N(\theta\un\{p,p_{2}^N\}).\] If $p=\emptyset$
then $q=\emptyset$ witnesses the lemma (by degree $2$ condensation\footnote{\label{ftn:d2con}If (i)
fails then we don't have the usual iterability assumptions
for degree $2$ condensation,
but one can easily modify the proof thereof using the fine structural assumptions in (ii).}), so assume $p\neq\emptyset$.
Let $\gamma=\max(p)$.
Let
\[ H_0=\cHull_{2}^N((\gamma+1)\un\{p_{2}^N\}).\] We have $H_0\in N$
since $\gamma<\rho_{2}^H$. Let $\pi:H_0\to N$ be the uncollapse. Let
$\gamma_1=\crit(\pi)=(\gamma^+)^{H_0}$. Let
$\delta=\card^N(\gamma)=\card^N(\gamma_1)=\rho_{2}^{H_0}$.
Let $R\ins N$ be least such that $\gamma_1\leq\OR^R$ and $R$ projects to
$\delta$. For a theory $t$ and parameters $a,b$, write $t_{a/b}$ for the theory resulting
by replacing $a$ with $b$.

\begin{clm*} Let  $t=\pTh_{2}^N((\gamma+1)\un\{p_{2}^N\})$. Then $t_{p_2^N/\dot{p}}$
is
$\bfrSigma_2^R$, where $\dot{p}$ is some constant symbol.\end{clm*}

\begin{proof}
 Suppose (i) holds. Let $\pP$ be the phalanx
$((N,1,\delta),(H_0,1),\gamma_1)$ (see \ref{ssec:ntn}). Compare $\pP$ with $N$, forming
normal/$1$-maximal trees $\Tt$ on $\pP$ and $\Uu$ on $N$. The details, including
iterability and the analysis of the comparison, mostly follow the proof of solidity in
\cite[\S8]{fsit} and \cite{outline}, using 
weak Dodd-Jensen as in \cite{outline}.

We get that $b^\Tt$ is above $H_0$ and
non-dropping (with $\deg^\Tt(b^\Tt)=1$), and $\fin^\Tt\ins\fin^\Uu$. If
$\fin^\Tt\pins\fin^\Uu$ then $\fin^\Tt$ is sound, and it follows\footnote{Here the traditional argument involves showing that $E^\Tt_\alpha$ is close
to $M^{*\Tt}_{\alpha+1}$, which takes some extra work because we are iterating a phalanx.
An alternative is to use \ref{cor:k+1solid} to deduce that if $\Tt$ is non-trivial then $M^\Tt_\infty$ is not sound.} that
\[ M^\Tt_\infty=H_0=R\pins N,\]
which suffices.
Otherwise let $Q=\fin^\Tt=\fin^\Uu$. Note that $\Uu$ is non-trivial.
In fact, $b^\Uu$ drops in model, since $\rho_{2}^Q\leq\delta<\rho_{2}^N$,
and by \ref{cor:k+1solid}.\footnote{In fact, $\rho_2^Q=\delta$,
again either by closeness of extenders or by
\ref{cor:k+1solid}.}
Moreover,
letting $\alpha+1=\min(b^\Uu\cut\{0\})$, $b^\Uu$ drops in model at
$\alpha+1$, but not in model or degree beyond there. So letting
$R^*=M^{*\Uu}_{\alpha+1}\pins N$, we have
$\delta=\rho_{2}^{R^*}\leq\crit(i^{\Uu}_{R^*,Q})$, so
$R^*=R$; also $\deg^\Uu(b^\Uu)=1$. So all $\bfrSigma_{2}^Q$ subsets of
$\delta$ are
$\bfrSigma_{2}^R$, which suffices.

Now suppose (ii) holds. Let $\pP=((N,3,\delta),(H_0,1),\gamma_1)$. Compare $\pP$ with $N$, again producing $\Tt$ on
$\pP$ and $\Uu$ on $N$, with $\Uu$ being $3$-maximal. Argue as in the previous case, but using the fact that $N$
is $(n+4)$-sound and
$\rho_{n+4}^N=\om$, with standard fine structure, in place
of weak Dodd-Jensen.\renewcommand{\qedsymbol}{$\Box$(Claim)}\end{proof}

Now $H_0=\cHull_{2}^N(\gamma_1\un\{p_{2}^N\})$ (recall $\gamma_1=(\gamma^+)^{H_0}$), so $H_0$ is a
$2$-solidity witness for $p_{2}^N\cup\{\gamma_1\}$. Let
$q_1=p_{2}^N\un\{\gamma_1\}$ and $w_1=w\un\{H_0\}$. Note that $R\cup\{R\}\sub\Hull_{1}^N(\delta\un\{\gamma_1\})$ by choice of $R$,
so by the claim and choice of $\gamma,\gamma_1$,
\[ x,w_1\in\Hull_2^N(\delta\cup\{q_1\}).\] 
Let $p'\in[\OR]^{<\om}$ be $<_\lex$-least with
\[ x,w_1\in\Hull_{2}^N(\theta\un\{p',q_1\}).\]
So $p'\sub\delta$. If
$p'=\emptyset$ we set $q=q_1$. Otherwise, repeat
the preceding argument with $q_1,w_1,p'$ in place of $q_0=p_2^N,w_0,p$, and
so on, producing
\[ q=p_{2}^N\un\{\gamma_1,\gamma_2,\ldots,\gamma_k\},\] with
$\gamma_1>\gamma_2>\ldots>\gamma_k$ (and $k$ as large as possible). Let $M=\cHull_2^N(\theta\cup\{q\})$
and $\pi:M\to N$ be the uncollapse. We have
$\rho_\om^M=\rho_{2}^M=\theta$ is a cardinal of $N$, $M$ is sound and $M\in
N$. Degree $2$ condensation (arguing as in Footnote \ref{ftn:d2con}
in case (ii)) gives $M\pins N$. 
\renewcommand{\qedsymbol}{$\Box$(\ref{lem:sound_hull})}
\end{proof}

\begin{newdfn}\label{dfn:reasonable}
Let $N$ be an $(n+1)$-sound premouse.
We say that $N$ is $n$-\emph{reasonable} if every $E\in\es_+^N$ is Dodd sound
and $N$ satisfies the conclusion of \ref{lem:sound_hull} with respect to $n$.
We say that $N$ is \emph{reasonable} iff $N$ is $\om$-sound and $n$-reasonable for all $n$.
\end{newdfn}

Note that $n$-reasonableness is first order,
and a consequence of $(n+1)$-soundness and $(n,\om_1,\om_1+1)$-iterability.
And by condition (ii) of \ref{lem:sound_hull},
it is also a consequence of $\om$-soundness and $(\om,\om_1+1)$-iterability,
and therefore:

\begin{newcor}
 Let $N$ be a $(0,\om_1+1)$-iterable premouse.
 Then every proper segment of $N$ is reasonable.
\end{newcor}

We now state the most central theorems of the paper. They share the
same basic theme, but differ in certain details.
The first was stated in the introduction:

\begin{tm}\label{thm:easy_coh}
Let $N$ be a $(0,\om_1+1)$-iterable premouse. Let $E\in N$ be such that
$(N||\lh(E),E)$ is a premouse and $E$ is total over $N$.
Then $E\in\es^N$.
\end{tm}

\ifbool{standardmode}{}{
\begin{olddfn}\end{olddfn}}

\begin{tm}\label{thm:finite_coh}
Let $N$ be a $(0,\om_1+1)$-iterable premouse. Suppose
there are cofinally many $\gamma<\OR^N$ such that $N|\gamma$ is admissible.
Suppose that if $N$ has a largest
cardinal $\kappa$ then $N\sats$``$\cof(\kappa)$ is not measurable''.

Let $(0,\Ww,\zeta,P,E,\Ww^+)$
be potential for $N$, with $E\in N$ and $\zeta<\om$. Suppose that
$b^{\Ww^+}$ does not drop in model.
Then $E\in\es^P_+$.
\end{tm}

Note that in \ref{thm:finite_coh} it is possible that $\lh(E)=\OR^P$, given
that $[0,\zeta]_\Vv$ drops. In this case, the conclusion is that $E=F^P$.

\begin{tm}\label{thm:strategy_coh_proj}
Let $N$ be a premouse and $\om<\eta<\OR^N$ with $\eta$ an $N$-cardinal and
 $N\sats$``$\eta^{++}$ exists''. Let $\R\in N$ be a reasonable
premouse with $\lpole\R\rpole=
(\her_{\eta^+})^N$.\footnote{$\lpole M\rpole$ denotes the universe of $M$.}
Let $\cc\in\OR^N$ and $\Sigma\in N$\footnote{Note that $\Sigma\sub
N|(\eta^{++})^N$, so it makes sense to have $\Sigma\in N$.}  with
\[ N\sats\text{``}\Sigma\text{ is an above-}\cc,\
(\om,\eta^+)\text{-strategy for }R\text{''.}\]
Let
$(\om,\Ww,\zeta,P,E,\Ww^+)$
be potential for $R$ and such that:
\begin{enumerate}[label=--]
\item $\Ww,E\in N$ and $\Ww$ is above $\eta$, of length $<(\eta^+)^N$ and via $\Sigma$,
\item $\cc<\lh(E)$ and $\cc$ is a cutpoint of $P|\lh(E)$,\footnote{The hypothesis ``$\cc<\lh(E)$'' is
redundant if $\Ww$ is non-trivial, as $\Ww$ is above $\cc$.}
\item $\crit(E)<\eta<\lh(E)$ \textup{(}therefore $b^{\Ww^+}$ does not drop\tu{)}.\footnote{The hypothesis ``$\eta<\lh(E)$'' is
redundant if $\Ww$ is non-trivial, as $\Ww$ is above $\eta$.}
\end{enumerate}

Then $E\in \es^P_+$.\end{tm}

In both \ref{thm:strategy_coh_proj} above and \ref{thm:strategy_coh_card} below,
$\Ww$ is above $\cc$ because it is via $\Sigma$, but we do not assume 
$\cc\leq\crit(E)$ explicitly (but $\cc\leq\crit(E)$
follows from the conclusion that $E\in\es_+^P$). Note that in \ref{thm:strategy_coh_proj},
although $\Sigma$ is an above-$\cc$ strategy
and possibly $\cc<\eta$, we assume also that $\Ww$ is above $\eta$,
and since $R$ has largest cardinal $\eta$,
this is a serious restriction on $\Ww$. It ensures that
$\Ww$ is equivalent to a tree $\Vv$ on some $R'\pins R$, and $\Vv\in R$.
However, \ref{thm:strategy_coh_card} makes no such  assumption.

\begin{tm}\label{thm:strategy_coh_card}
Let $N$ be a premouse and $\eta\geq\om$
regular in $N$ with $N\sats$``$\eta^+$ exists''. Let
$\R\in N$ be a reasonable premouse with $\lpole\R\rpole = \her_\eta^N$. Let $\cc<\eta$ be a cutpoint of $\R$. Suppose
that if $\tau$ is the largest cardinal of $R$
and $\cc\leq\cof^N(\tau)$ then $\cof^N(\tau)$ is not the critical 
point of any $R$-total $F\in\es^R$.

Let $\Sigma\in N$ be such that $N\sats$``$\Sigma$ is an above-$\cc$,
$(\om,\eta)$-strategy for $R$''. Let $(\om,\Ww,\zeta,P,E,\Ww^+)$
be potential for $\R$ with
$\Ww,E\in N$ and $\Ww$ via $\Sigma$,
$\lh(\Ww)<\eta$ and
$\cc<\OR^P$,
and $b^{\Ww^+}$ non-dropping.
Then $E\in\es^P_+$.\end{tm}

In the following two variants, the tree $\Ww\in N$,
and $\Ww$ is via $\Sigma$,
but $\Sigma$ is external, not assumed to be in $N$.
The first is actually a corollary of the second, but it contains the main point and its statement is simpler,
so we give it first:

\begin{newtm}\label{thm:cohering_simple}
Let $N$ be a premouse and $\Sigma$ a $(0,\theta^++1)$-strategy
for $N$, where $\theta\geq\om$. Let $\eta<\OR^N$ be an $N$-cardinal.
Let $(0,\Ww,\zeta,P,E,\Ww^+)$
be potential for $N|\eta$, such that $\Ww,E\in N$, and
$\Ww$ follows the strategy for $N|\eta$ induced by $\Sigma$,
$\lh(\Ww)<\theta^+$,
and
$b^{\Ww^+}$ does not drop.
Then $E\in\es^P_+$.
\end{newtm}

Theorem \ref{thm:cohering} below is
finer than the previous results. Hypothesis
\ref{item:Wdefinable} of
\ref{thm:cohering}, which asserts that
$\Ww$ and $E$ are suitably definable over $N$, will be made precise in 
\ref{rem:clarification}.

\begin{tm}\label{thm:cohering}
Let $k<\om$. Let $N$ be a $(k+1)$-sound $k$-reasonable premouse and $\cc\in\OR$ be
a cutpoint of $N$.
Let $\theta\in\OR$ and $\Sigma$ be an above-$\cc$,
$(k,\theta^++1)$-strategy for $N$.\footnote{Actually if $\theta\geq\om_1$
then a $(k,\theta^+)$-strategy suffices.} Let $(k,\Ww,\zeta,P,E,\Ww^+)$
be potential for $N$, such that:
\begin{enumerate}[label=\textup{(}\alph*\textup{)},
ref=\textup{(}\alph*\textup{)}]
 \item $\Ww$ is via $\Sigma$ with $\lh(\Ww)<\min(\theta^+,\rho_0^N)$,
 \footnote{$\lh(\Ww)<\rho_0^N$ will actually follow from the assumption that $\Ww$ is $\bfrSigma_{k+1}^N$ in the codes.}
 \item $\cc<\lh(E)$,
 \item $b^{\Ww^+}$ does not drop, and
$\crit(E)<\rho_{k+1}(M^\Ww_\pr)$ where $\pr=\pred^{\Ww^+}(\zeta+1)$,
 \item\label{item:Wdefinable} $\Ww$
and  $E$ are $\bfrSigma_{k+1}^N$ in the
codes.\footnote{\label{ftn:clarify}Cf.
\ref{rem:clarification}}
\end{enumerate}

Then $E\in\es^P_+$.
\end{tm}

Note that in \ref{thm:cohering}, we might have $\lh(E)=\OR^P$ even with $[0,\zeta]_\Ww$ non-dropping.
By \ref{cor:k+1solid}, the assumptions of \ref{thm:cohering} imply
$\rho_{k+1}(M^\Ww_\pr)=\sup i^\Ww_{0,\pr}``\rho_{k+1}^N$.

\begin{rem}\label{rem:superstrong}The literal generalization
of these results to ``premice'' $M$ with enough extenders of superstrong
type on $\es_+^M$ fails. For suppose $E$ is a total
type 2 extender on $\es^M$, $U=\Ult(M,E)$, $\kappa=\lgcd(M|\lh(E))$,
$F\in\es_+^U$, $\kappa=\crit(F)$, $F$ is of superstrong type and $U$-total.
Then $\Ult(M|\lh(E),F)$ and $U|\lh(F)$ are distinct
active premice, with the same reduct. If $M$ is iterable and $M\sats\ZFmin$,
then the other hypotheses of \ref{thm:cohering} also hold.

This doesn't appear to be a strong failure of \ref{thm:cohering}, however, since
both extenders are on the extender sequence of normal iterates of $M$.
(Also, in Jensen's $\lambda$-indexing, these extenders are not indexed at
the same point.)\end{rem}

\ifbool{standardmode}{}{
\begin{oldlem}\end{oldlem}}

In the proof of \ref{thm:strategy_coh_proj} we will need to consider ostensibly illfounded ultrapowers,
and fine structural embeddings between them. Toward this we define generalizations of the usual fine structural notions
for such possibly illfounded ``premice'', at least to the extent that these notions are well defined.
The definitions and calculations are almost the usual ones; however when $U=\Ult_m(M,E)$ is illfounded
there is a subtlety in the definition of $\rho_m^U,T_m^U$ and $\rSigma_{m+1}^U$,
as in this case it seems possible that the natural candidate for $\rho_m^U$ is a proper segment of $\OR^U$
which is however not a cut ``in'' $U$,
i.e. there is no $\rho\in\OR^U$ such that
\[ \rho_m^U=\{\alpha\in U\mid U\sats\alpha\in\rho\}.\]
So we give a detailed description of things in order to handle this issue.
\footnote{In an earlier version of this paper,
the author had not noticed the issue with $\rho_m^U$ and
simply claimed that there was no problem
in adapting the fine structure to illfounded structures,
and omitted any discussion thereof.
A question from the anonymous referee regarding the topic
lead to the author's noticing the issue.}
\begin{newdfn}\label{dfn:illffs}
Let $N=(\univ{N},\es,F)$ be a structure in the premouse language, possibly illfounded.
Suppose that $N\sats$``I am a premouse'' and $N$ has wellfounded $\omega$;
so $N$ is correct about what formulas are.
We define the fine structure of $N$ as far as possible. Write
\[ \OR^N=\{\alpha\in\univ{N}\mid N\sats\alpha\in\OR\}.\]

Suppose first that $N\sats$``I am not type 3''. Define $\rho_0^N=\OR^N$. We say that $N$ is \emph{$0$-sound}
and \emph{$0$-feasible}. Define $\rSigma_1^N$ as usual, and for $X\sub N$,
define $\Hull_1^N(X)$ as usual: the set of all $y\in N$ such that for some $\Sigma_1$ formula $\varphi$ and $\xvec\in X^{<\om}$,
$y$ is the unique $y'\in N$ such that $N\sats\varphi(\xvec,y')$.

Define $\rho_1^N$ as the set of all $\alpha\in\OR^N$ 
such that either $\alpha<\om$ or for all $p\in N$, we have $\Th_1^N(\alpha\cup\{p\})\in N$. Note that $\rho_1^N$
is an $\in^N$-initial segment of $\OR^N$, but we might have $\rho_1^N\psub\OR^N$ but $\rho_1^N$ not ``in'' $N$.
Given $p\in[\OR^N]^{<\om}$, we define the \emph{$1$-solidity} of $p$ for $N$ as usual.
We say that $N$ is \emph{$1$-sound} iff there is $p\in[\OR^N\cut\rho_1^N]^{<\om}$
such that $p$ is $1$-solid for $N$ and $N=\Hull_1^N(\rho_1^N\cup\{p\})$.

Suppose $N$ is $1$-sound. Define $p_1^N$ to be the unique $p$ witnessing this,
and define $u_1^N$ as usual.
(Uniqueness: Suppose $p<_\lex^N q$ are both such. By the $1$-solidity of $q$,
there is a ``limit ordinal'' $\alpha\in\OR^N$ such that ``$\rho_1^N\sub\alpha$'' and $\Th_1^N(\alpha\cup\{p\})\in N$.
But $N=\Hull_1^N(\alpha\cup\{p\})$, and the usual diagonalization argument then constructs
a new subset of $\alpha^{<\om}$, a contradiction.)
Define $T_1^N$ as usual; that is, the set of 
all $t=\Th_1^N(\alpha\cup\{q\})$ for some $\alpha\in\rho_1^N$ and $q\in N$.
So $T_1^N\sub N$.
Define $\rSigma_2^N$ and $\Hull_2^N(X)$ from  $T_1^N$ as usual.
(So $u_1^N\in\Hull_2(X)$ by definition.)

We say that $N$ is \emph{$1$-feasible} iff $N$ is $1$-sound and either $\rho_1^N=\OR^N$ or $\rho_1^N$ is ``in'' $N$. 
If $N$ is $n$-feasible where $n\geq 1$ and $\rho_n^N>\om$, define $\rho_{n+1}^N$, $(n+1)$-solidity and $(n+1)$-soundness as above.
If $N$ is $(n+1)$-sound define $p_{n+1}^N,u_{n+1}^N,T_{n+1}^N,\rSigma_{n+2}^N,\Hull_{n+2}^N$
(note these definitions do not require $(n+1)$-feasibility) and also $(n+1)$-feasibility
as above. 

If $N$ is instead type 3 then these things are adapted in the obvious manner to $\core_0(N)=N^\sq$.
(Here $N^\sq$ is well-defined as $\nu(F^N)$ is well-defined by the premouse axioms.)

Let $M,N$ be as above and $\pi\colon\core_0(M)\to\core_0(N)$.
We say that $\pi$ is a \emph{virtual \tu{(}weak/near\tu{)} $m$-embedding} given the usual conditions on $\pi$,
using the parameters etc defined as above. For example, for virtual weak $m$:\footnote{Cf.~the definition of weak $m$-embedding in \S\ref{ssec:ntn}.}
\begin{itemize}[label=--]
 \item  $M,N$ are $m$-sound,
\item for $k\leq m$ we have $\pi(p_k^M)=p_k^N$,
\item for $k<m$ we have either
\begin{itemize}[label=--]
\item $\rho_k^M=\OR(\core_0(M))$ and $\rho_k^N=\OR(\core_0(N))$, or
\item $\rho_k^M<\OR(\core_0(M))$ and $\pi(\rho_k^M)=\rho_k^N$,
\end{itemize}
\item $\pi$ is $\rSigma_m$-elementary,
\item there is a cofinal $X\sub\rho_m^M$ such that $\pi$ is $\rSigma_{m+1}$-elementary on $\Hull_{m+1}^M(X)$.
\end{itemize}
Because $M,N$ are $m$-sound (hence $(m-1)$-feasible if $m>0$) we have defined all the fine structural notions needed to make sense of these requirements.

We also extend the definition of \emph{reasonable} to possibly
illfounded structures $N$ as above, in the obvious fashion.
\end{newdfn}
\begin{newdfn}
 Let $P,Q$ be active premice and $\sigma:\core_0(P)\to\core_0(Q)$ a weak $0$-embedding.
 Then $\psi_\sigma:\Ult_0(P,F^P)\to\Ult_0(Q,F^Q)$ denotes the map induced by $\sigma$ through the Shift Lemma.
 (Note $\sigma\sub\psi_\sigma$.)
\end{newdfn}

\begin{newlem}\label{lem:illffs}
 Let $m\leq n\leq\om$. Let $M,N$ be premice, $M$ is $m$-sound, $N$ is $n$-sound,
$\pi:M\to N$ a weak $m$-embedding.
 Let $P,Q$ be active premice,
 $\sigma:P\to Q$
 a weak $0$-embedding.
 Let $E=F^P$, $F=F^Q$, $\kappa=\kappa_E$, $\mu=\kappa_F$.
 Suppose
 \[ M||(\kappa^+)^M=P|(\kappa^+)^P\text{ and  }N||(\mu^+)^N=Q|(\mu^+)^Q \]
 and $\pi\rest(\kappa^+)^M\sub\sigma$.
 Suppose $\kappa<\rho_m^M$ and $\mu<\rho_n^M$.
 Let
 \[ U=\Ult_m(M,E)\text{ and }W=\Ult_n(N,F)\]
 \tu{(}$U,W$ might be illfounded\tu{)}. Let
 $i=i^{M,m}_E$ and $j=i^{N,n}_F$ and
$\psi:U\to W$
 be the Shift Lemma map.
 
 Then $U$ is $m$-sound, $W$ is $n$-sound,
 $i$ a virtual $m$-embedding, $j$ a virtual $n$-embedding, $\psi$ a virtual weak $m$-embedding,
$\psi\com i=j\com\pi$,
 \[ \psi_\sigma\rest(\OR^P+1)\sub\psi.\]
  Moreover, if $M$ is $\kappa$-sound and $\rho_{m+1}^M\leq\kappa$ then
 \[ U=\Hull_{m+1}^U(\nu(E)\cup\{i(p_{m+1}^M)\}).\]
 \end{newlem}
 \begin{proof}[Proof Sketch]
 For simplicity we assume $m=2$ and $M,N$ are not type 3.
 Note that ($U$ could be illfounded and) $\rho_2^U$ might not be ``in'' $U$.
 
 To verify that $U$ is $2$-sound and $i$ is a virtual $2$-embedding, one proves:
 \begin{enumerate}
  \item  $\rSigma_0$- and $\rSigma_1$-Los theorem holds. Hence $i$ is $\rSigma_1$-elementary.
 \item Let $b\in[\nu(E)]^{<\om}$ and
 \[ f,g:[\kappa]^{|b|}\to M \]
 be $\bfrSigma_2^M$ with $\rg(g)\sub\rho_1^M$.
 Define
 \[ h:[\kappa]^{|b|}\to M,\]
 \[ h(u)=\Th_1^M(g(u)\cup\{f(u)\}).\]
 Note that $h$ is $\bfrSigma_2^M$.
 Let $\widetilde{f}=[f,b]^{M,2}_E$ and likewise $\widetilde{g},\widetilde{h}$.
 Then
 \[ U\sats\text{``}\widetilde{h}=\Th_1(\widetilde{g}\cup\{\widetilde{f}\})\text{''.}\]
 \item If $\rho_1^M=\OR^M$ then $\rho_1^U=\OR^U$, and if $\rho_1^M<\OR^M$ then $i(\rho_1^M)\leq\rho_1^U$.
 \item $i(p_1^M)$ is $1$-solid for $U$, as witnessed by $i(w_1^M)$ ($w_1^M$ is the set of $1$-solidity witnesses for $M$).
 \item $U=\Hull_1^U(i(\rho_1^M)\cup\{i(p_1^M)\})$.
 \item $U$ is $1$-sound and $1$-feasible with $p_1^U=i(p_1^M)$ and $u_1^U=i(u_1^M)$
 and either $\rho_1^M=\OR^M$ and $\rho_1^U=\OR^U$, or $\rho_1^M<\OR^M$ and $\rho_1^U=i(\rho_1^M)$ (use diagonalization to see $\rho_1^U\leq i(\rho_1^M)$).
 \item $\rSigma_2$-Los theorem holds. Hence $i$ is $\rSigma_2$-elementary.
 \item\label{item:2-theories} Let $f:[\kappa]^{|b|}\to M$ be $\rSigma_2^M(\{q\})$ and $\kappa\leq\alpha<\rho_2^M$.
 Let $t=\Th_2^M(\alpha\cup\{q\})$. Define $h:[\kappa]^{|b|}\to M$ by $h(u)=$ the $\rSigma_2$-theory
 in parameters $\alpha\cup\{f(u)\}$ determined by $t$ and substitution. Then $h$ is $\rSigma_2^M(\{(q,t)\})$
 and
 \[ U\sats\text{``}\widetilde{h}=\Th_2(i(\alpha)\cup\{\widetilde{f}\})\text{''.}\]
 \item $\sup i``\rho_2^M\leq\rho_2^U$.
 \item $i(p_2^M)$ is $2$-solid for $U$, as witnessed by $i(w_2^M)$.
 \item Let $b\in[\nu(E)]^{<\om}$ and $\varphi$ be $\rSigma_2$ and $q\in M$ be such that $\varphi(q,\cdot,\cdot)$
 defines a function $f:[\kappa]^{|b|}\to M$ over $M$. Let $\widetilde{f}=[f,b]^{M,2}_E$. Then
 \[ U\sats\text{``}\widetilde{f}\text{ is the unique }y\text{ such that }\varphi(i(q),b,y)\text{''}.\]
 \item\label{lem:U_is_2-Hull} $U=\Hull_2^U(\rg(i)\cup\nu(E))=\Hull_2^U(i``\rho_2^M\cup\{i(p_2^M)\}\cup\nu(E))$.
 \item $U$ is $2$-sound with $p_2^U=i(p_2^M)$ and $i(u_2^M)=u_2^U$
 and $\rho_2^U=\sup i``\rho_2^M$ (but maybe $U$ is not $2$-feasible).
 \item $i$ is $\rSigma_3$-elementary; this follows from properties
 \ref{item:2-theories}, \ref{lem:U_is_2-Hull} and $\rSigma_1$-elementarity,
 just like when $U$ is wellfounded.
 \end{enumerate}

 Similarly, $j:N\to W$ is also a virtual $n$-embedding, and the $\rSigma_n^N$-Los theorem holds, etc.
 The usual calculations using all these facts (including $\rSigma_2$-Los theorem for both ultrapowers),
 now give that $\psi$ is a virtual weak $2$-embedding, commutativity holds etc.

 Finally suppose that $M$ is $\kappa$-sound and $\rho_3^M\leq\kappa$.
 So $M=\Hull_3^M(\kappa\cup\{p_3^M\})$.
 By the $\rSigma_3^M$-elementarity of $i$,
 \[ \rg(i)=\Hull_3^U(\kappa\cup\{i(p_3^M)\}). \]
 But 
 $U=\Hull_2^U(\rg(i)\cup\nu(E))$ (as above), so
 $U=\Hull_3^U(\nu(E)\cup\{i(p_3^M)\})$.
 \end{proof}
\begin{newlem}\label{lem:illflift}
 Let $\pi:\core_0(M)\to\core_0(N)$ be a virtual weak $k$-embedding.
 Let $\varphi$ be $\rSigma_{k+1}$ and $x\in\core_0(M)$. If $\core_0(M)\sats\varphi(x)$
 then $\core_0(N)\sats\varphi(\pi(x))$.
\end{newlem}
\begin{proof}
 If $k=0$ this is immediate. If $k>0$ it holds because $\pi``T_k^M\sub T_k^N$,
 which holds just as for a weak $k$-embedding.\footnote{That is, with \emph{weak $k$-embedding} defined as in \S\ref{ssec:ntn}.}
\end{proof}
\begin{proof}[Proof of \ref{thm:strategy_coh_proj}]\setcounter{clm}{0}
We argue overall as in  the proof of the ISC in
\cite[\S10]{fsit}. First, for motivation,
make the following simplifying assumptions: that $\R=N|(\eta^+)^N$, $\Ww$
is trivial, so $E$ coheres $\es^N$, and that
we can compare $\Ult(N,E)$ with
$N$, producing trees $\Uu,\Vv$ such that
$\fin^\Uu=\fin^\Vv$,
$b^\Uu$ and $b^\Vv$ do not drop,
$i^\Uu\com i_E=i^\Vv$, and
$\nu_E\leq\crit(i^\Uu)$.
Then standard
arguments using the ISC show that $E$ is used in $\Vv$, and
$E\in\es^N$, as required.

Now drop the above simplifying assumptions. We don't even know, for
instance, that $\Ult(N,E)$ is wellfounded. So
we will instead arrange a situation similar to that in the preceding paragraph,
but with
$\R$ replaced by a certain hull $M$, which will be $k+1$-sound for some $k$,
with $\rho_{k+1}^M=\om$. Denoting the resulting collapses of $E$, $\Ww$, etc, by
$\Ebar$, $\Wwbar$, etc, and letting $\iota=\lgcd(P|\lh(E))$, we will
show that the phalanx
$\pP=\Phi(\Wwbar,\iotabar,\Ebar)$ (see \ref{dfn:potentialtree}) is sufficiently
iterable in $N$ to allow comparison with $M$.
If
$\iotabar=\nu(\Ebar)$ (that is, $E$ is type 1 or 3), then the proof is completed as in the previous paragraph.
In general we will rely on the Dodd structure analysis from
\S\ref{sec:Dodd}, appealing to \ref{lem:extendermaximality}.

We now commence with the details. Because $\Ww$ is above $\eta$,
it is equivalent to a tree $\Vv$ on some $R'\pins R$ such that $\rho_\om^{R'}=\eta$.
Note then that $\Vv,E\in
\R$. Let $k<\om$ be large enough to reflect the facts about $x=(\Vv,E,\cc)$
into the model $M$ defined below.
Now by assumption $\R$ is reasonable. So we can fix $q\in\R^{<\om}$ witnessing \ref{lem:sound_hull}
with
respect to $(\R,k,x,\theta=\om)$. Let
\[ M=\cHull_{k+1}^{\R}(\{q\})\]
and $\pi:M\to\R$ be the uncollapse. So $M$ is sound, reasonable, $\rho_{k+1}^M=\om$,
$p_{k+1}^M=\pi^{-1}(q)$, and $\pi$  is a near
$k$-embedding. Since $R\sats\ZFC^-$ and $\pi,\Sigma\in N$ and $\om_1^N<(\eta^+)^N$,
\[ N\sats\text{``}M\text{ is above-}\bar{\cc},\ (\om,\om_1+1)\text{-iterable''.}\] 

Let $\Vvbar=\pi^{-1}(\Vv)$, etc. Let $\Wwbar$ be the same
tree except with $M^\Wwbar_0=M$ (we have $M^{\Vvbar}_0\pins M$).
So
$(\om,\Wwbar,\zetabar,\Pbar,\Ebar,\Wwbar^+)$ is potential for
$M$ and $\bar{\cc}$ is a cutpoint of $\Pbar|\lh(\Ebar)$.
Let
\[ \bar{U}=\Ult_k(M,\bar{E})=M^{\Wwbar^+}_\infty.\]
We will show that $\bar{U}$ is wellfounded, and define the phalanx
$\pP=\Phi(\Wwbar,\iotabar,\Ebar)$. (See \ref{dfn:potentialtree}. Note that
$\pP$ has last model $\bar{U}$, at degree $k$.)

\begin{clm*}\label{clm:pP_iter} $\bar{U}$ is wellfounded and $N\sats$``$\pP$ is
$(\om_1+1)$-iterable.''\end{clm*}

\begin{proof}
We first find $M'\pins P||\lh(E)$ and a virtual weak $k$-embedding $\psi':\bar{U}\to M'$
such that $\psi'\in N$ and $\pi\rest(\max(\iotabar,\bar{\cc})+1)\sub\psi'$.
It follows that $\bar{U}$ is wellfounded,
and so $\psi'$ is in fact a weak $k$-embedding.
Working in $N$, we will then lift
normal trees $\Tt$ on $\pP$ to above-$\cc$, normal trees $\Uu$ on $\Phi(\Ww)$.
Since $\Ww$ is via $\Sigma\in N$, this suffices. 
We will use $\pi,\psi'$, which are in $N$, as initial copying maps, using $\pi$ to
lift the models of $\Phi(\Wwbar)$ to (some) models of $\Phi(\Ww)$ and $\psi'$ to lift
$\Ubar$ to $M'$.

Let $U=\Ult_0(R,E)=\Ult_\om(R,E)$. (We don't assume $U$ is wellfounded.) Let $j=i^R_E$, $\jbar=i^{M,k}_{\Ebar}$ and
$\psi:\Ubar\to U$
be the Shift Lemma map. By \ref{lem:illffs} and in the sense of \ref{dfn:illffs},
$\Ubar$ is $k$-sound, $U$ is fully sound,
$\psi$ is a virtual weak $k$-embedding, $\psi\com\jbar=j\com\pi$ and
$\pi\rest(\lh(\Ebar)+1)\sub\psi$.

By elementarity, $U$ is reasonable (in the sense of \ref{dfn:illffs}). 
Let $x'=(j(q),\nu(E),\cc)$ and $q'\in U^{<\om}$ ``witness
\ref{lem:sound_hull}'' with respect to
$(U,k,x',\iota)$ (in the obvious sense).
Let
\[ M'=\cHull_{k+1}^{U}(\iota\un q')\] and let
$\pi':M'\to U$
be the uncollapse, so $\pi'$ is a virtual near $k$-embedding. Then $M'\pins U|\lh(E)$ since $\lh(E)$ is a
cardinal in
$U$, so $M'\pins P||\lh(E)$ and $M'$ is wellfounded. Moreover,
\[ (\nu(E)+1)\un\{\cc,j(q)\}\sub\rg(\pi').\]

We claim that $\rg(\psi)\sub\rg(\pi')$. First let us verify that
$\rg(\psi\com\jbar)\sub\rg(\pi')$. Let $y\in M$.
Let $\varphi$ be $\rSigma_{k+1}$ with
$y=\minterm^M_\varphi(\qbar)$ (see \ref{dfn:minterm}).
Then by commutativity and elementarity of $j$,
\[ \psi(\jbar(y))=j(\pi(y))=\minterm^U_\varphi(j(q))\in\rg(\pi'),\]
as desired. But
\[ \bar{U}=\Hull_k^{\bar{U}}(\rg(\jbar)\cup\nu(\Ebar)) \]
(see the proof of \ref{lem:illffs}) and $\psi$
is $\rSigma_k$-elementary and
$\psi``\nu(\Ebar)\sub\nu(E)\sub\rg(\pi')$,
which suffices.

It follows that
$\psi':\bar{U}\to M'$
is a virtual weak $k$-embedding
where
\[ \psi'(x)=(\pi')^{-1}(\psi(x)).\]
But then $\bar{U}$ is wellfounded, so $\psi'$ is in fact a weak $k$-embedding.

So we have $M',\psi'$ as desired. We now
lift trees $\Tt$ on $\pP$ to $\Uu$ on $\Phi(\Ww)$, via $\pi,\psi'$.
The details are mostly standard; we just make a
couple of key points.

Note
$\pi\rest(\max(\iotabar,\bar{\cc})+1)\sub\psi'$ and
\[\psi'(\iotabar,\bar{\cc})=\iota,\cc<\psi'(\lh(\Ebar))<\lh(E)=\pi(\lh(\Ebar)).\]
For the models $M^\Wwbar_\alpha$ we use copy map
\[ \pi\rest M^\Wwbar_\alpha:M^\Wwbar_\alpha\to M^\Ww_{\pi(\alpha)},\]
and for $M^\Tt_0=\Ubar$ we use
\[ \pi_0=\psi':\Ubar\to M'\pins M^\Ww_\zeta=M^\Uu_0.\]
For $\alpha>0$ we then proceed to define $\pi_\alpha:M^\Tt_\alpha\to M^\Uu_\alpha$
inductively as usual.
 Now $\lh(\Ebar)<\lh(E^\Tt_0)$, so
\[ \iota,\cc<\psi'(\lh(\Ebar))<\lh(E^\Uu_0)\leq\OR^{M'}<\lh(E),\]
and $\psi'\rest(\lh(\Ebar)+1)\sub\pi_\alpha$ for all $\alpha\geq 0$.
It follows that if $\crit(E^\Tt_\alpha)<\iotabar$ then
$\Troot^\Uu(\alpha+1)=\pi(\Troot^\Tt(\alpha+1))$ and $\Tt,\Uu$ agree about drops and degree at $\alpha+1$, and if there is
a drop in model then $\pi(M^{*\Tt}_{\alpha+1})=M^{*\Uu}_{\alpha+1}$. If
$\crit(E^\Tt_\alpha)=\iotabar$ then $M^{*\Tt}_{\alpha+1}=\Ubar$ and
$M^{*\Uu}_{\alpha+1}=M'$ and
$\deg^\Tt(\alpha+1)=\deg^\Uu(\alpha+1)=k$.\footnote{If we had
used $\iotabar+1$ in place of $\iotabar$ as an exchange ordinal in defining
$\pP$, and $\iotabar=\crit(E^\Tt_\alpha)$, then
$P'=M^{*\Tt}_{\alpha+1}\ins\Pbar$ and $M^{*\Uu}_{\alpha+1}=M'$. We have
not constructed appropriate embeddings for lifting such $P'$ to $M'$, so
the copying process would break down.}

Because $\bar{\cc}<\lh(\Ebar)$ is a cutpoint of $\Pbar|\lh(\Ebar)$,
we get that $\bar{\cc}$ is a cutpoint of $M^\Tt_\alpha$ for all $\alpha$ 
just as at the start of the proof of \ref{lem:extendermaximality}.
It follows that $\Uu$ is above $\cc$, so we can use $\Sigma$ for forming $\Uu$.\end{proof}

Now since $M$ is above-$\bar{\cc}$, $(\om,\om_1+1)$-iterable in $N$, the claim above combined
with Lemma \ref{lem:extendermaximality} implies
$\Ebar\in\es^{\Pbar}_+$, so $E\in\es^P_+$, completing
the proof.\footnote{In \ref{lem:extendermaximality}, $b^{\Ww^+}$ is allowed to drop,
but we needed $b^{\Ww^+}$ to be non-dropping in
order to prove that $\pP$ is iterable.}
\end{proof}
\ifbool{standardmode}{}{
\begin{oldlem}\end{oldlem}}

The proofs of the variants are 
similar, proceeding by defining an iterable phalanx $\pP$
and appealing to Lemma \ref{lem:extendermaximality}.  So we just give a sketch.
\begin{proof}[Proof of \ref{thm:easy_coh}]
We may assume $N$ is passive, and note that every extender in $\es^N$ is Dodd sound
and every proper segment of $N$ is reasonable. Let $\gamma<\OR^N$ be such that $E$ is definable over $N|\gamma$ (so
$\lh(E)\leq\gamma$). We have $\crit(E)<\rho_\om^{N|\gamma}$ as $E$ is $N$-total.
Using \ref{lem:sound_hull} let $k<\om$ and $q\in[\gamma]^{<\om}$ be such that $M=\cHull_{k+1}^{N|\gamma}(\{q\})$
is sound (and $\rho_{k+1}^M=\om$) and the facts about $E$ reflect to an extender $\Ebar$ defined similarly over $M$.
We have $\crit(\Ebar)<\rho_k^M$.
Let $\Ubar=\Ult_k(M,\Ebar)$ and $\iotabar=\lgcd(M||\lh(\Ebar))$
and define the phalanx
\[\pP=((M,k,\iotabar),(\Ubar,k),\lh(\Ebar)).\]
It suffices to see that $\pP$ is $(\om_1+1)$-iterable,
as then we can use is \ref{lem:extendermaximality}.
But we have the Shift Lemma map
$\psi:\Ubar\to i^{N,0}_E(N|\gamma)$, and we  can find $M'\pins N||\lh(E)$ and a suitable weak $k$-embedding $\psi':\Ubar\to M'$ as
in the preceding proof.
\end{proof}

\begin{proof}[Proof of \ref{thm:finite_coh}]
We may assume that $N$ is countable with largest cardinal $\kappa$. Since
$E\in N$, $\lh(E)<\OR^N$, so we may also assume $N$ is passive.
Let $\lambda$ be such that $\kappa<\lambda<\OR^N$ and $E\in N|\lambda$ 
and $\rho_\om^{N|\lambda}=\kappa$ and $\cof^N(\kappa)=\cof^{N|\lambda}(\kappa)$.
Let $\gamma$ be least $>\lambda$ such that $N|\gamma$ is admissible.
We have $\Ww$ on $N$ and $\lh(\Ww)<\om$.
We claim that there is a $0$-maximal tree $\Vv$ on $N|\gamma$
using the same extenders (with the same indices) as does
$\Ww$, and moreover,
$\Ww,\Vv$
have the same tree, drop and degree structure, and
letting $l=\{\kappa,\lh(E^\Ww_0),\ldots,\lh(E^\Ww_{\zeta-1})\}$,
then for $m<\lh(\Ww)$,
\begin{enumerate}[label=--]
\item $M^\Vv_m$ is a $\Delta_1^{N|\gamma}(l)$-definable transitive class of $N|\gamma$ and:
 \item if $[0,m]_\Ww$ does not drop then $M^{\Vv}_m=i^\Ww_{0m}(N|\gamma)$ has height $\gamma$, and
\item if $[0,m]_\Ww$ drops then $M^{\Vv}_m=M^{\Ww}_m\in N|\gamma$.
\end{enumerate}
For suppose $[0,m]_\Ww$ does not drop. Then $i^\Ww_{0m}$ is continuous at $\kappa$
and $i^\Ww_{0m}(N|\kappa)=\Ult_0(N|\kappa,F)$ where $F$ is the branch extender,
since $N\sats$``$\cof(\kappa)$ is not measurable''. Therefore if $\rho_\om^{N|\alpha}=\kappa$
then $i^\Ww_{0m}(N|\alpha)$ is determined by $i^\Ww_{0m}(t)$ where $t\sub\kappa$ is the appropriate theory,
and $i^\Ww_{0m}(t)=\bigcup_{\beta<\kappa}i^\Ww_{0m}(t\rest\beta)$. This gives a $\Delta_1^{N|\gamma}(l)$ description of
the function $N|\alpha\mapsto i^\Ww_{0m}(N|\alpha)$ for such $\alpha<\gamma$.
Therefore $M^{\Vv}_m$ is $\Delta_1^{N|\gamma}(l)$.
Using the admissibility of $N|\gamma$,
it follows that $\OR(M^{\Vv}_m)=\gamma$ and $M^\Vv_m$ is admissible.
So $M^{\Vv}_m=i^\Vv_{0m}(N|\gamma)$ by the minimality of $\gamma$.

Now considerations as above show that if the conclusion of the theorem fails,
the failure is a first-order
sentence satisfied by $N|\gamma$ (taking $l,E$ to be minimal counterexamples with respect to $N|\gamma$ and with properties as above,
then they also yield a counter-example with respect to $N$). So we may assume
that $l,E$ are minimal. Fix
$0<k<\om$ such that the failure
and all properties established reflect into
$M=\cHull_{k+1}^{N|\gamma}(\emptyset)$. Let
$\pi:M\to N|\gamma$ be the
uncollapse. Write $\pi(\bar{E})=E$ etc.

Now $\Vvbar$ is a $0$-maximal tree on $M$, but is in fact equivalent to a $k$-maximal tree.
For $\rho_1^M=\rho_k^M=\bar{\kappa}$, so it suffices to see that the relevant $\bfrSigma_k^M$ functions
with range $\sub\bar{\kappa}$ are bounded on measure one sets.
So let $\mu<\bar{\kappa}$ and $F\in\es^M$ be $M$-total with $\crit(F)=\mu$. Let $b\in[\nu(F)]^{<\om}$.
Let $f:[\mu]^{|b|}\to\bar{\kappa}$ be an $\bfrSigma_k^M$ function.
We want $A\in F_b$ such that $f``A$ is bounded in $\bar{\kappa}$.
If $\kappa$ is singular in $N$ then $\bar{\kappa}$ is singular in $M$
with cofinality $\neq\mu$, which easily gives such an $A$.
Suppose $\kappa$ is regular in $N$. Let $\pi(f)$ be the $\bfrSigma_k^{N|\gamma}$-function
defined over $N|\gamma$ as $f$ is over $M$. Because $\pi$ is a near $k$-embedding and $\pi(\bar{\kappa})=\kappa=\rho_k^{N|\gamma}$,
we have $\pi(f):[\pi(\mu)]^{|b|}\to\kappa$, but $\pi(f)\in N$, so $\pi(f)$ is bounded in $\kappa$.
But the fact that there is a bound on $\pi(f)$ is an $\rSigma_{k+1}$ assertion about the relevant parameters,
and therefore $f$ is bounded in $\bar{\kappa}$, so $A=\mu^{|b|}$ suffices.

Now let $\Vvbar^+=\Vvbar\conc\left<E\right>$ as a $k$-maximal potential tree.
Let $\Ubar=M^{\Vvbar^+}_\infty$.
As in the proof
of \ref{thm:strategy_coh_proj} we obtain a suitable lifting map $\psi':\Ubar\to
M'$ with $M'\pins M^{\Vv}_\infty||\lh(E)$. (Let $\Ww^+=\Ww\conc\left<E\right>$ and $\Vv^+=\Vv\conc\left<E\right>$,
both $0$-maximal. As above, $M^{\Vv^+}_\infty=i^{\Ww^+}_{0\infty}(N|\gamma)$
and $i^{\Vv^+}_{0\infty}=i^{\Ww^+}_{0\infty}\rest N|\gamma$.
So $i^{\Vv^+}_{0\infty}$ is fully elementary and we
can define $M'$ as a hull of $M^{\Vv^+}_\infty$.) Define the phalanx
$\pP=\Phi(\Vvbar,\iotabar,\Ubar)$, where $\iotabar=\lgcd(\Pbar|\lh(\Ebar))$, and
show that $\pP$ is iterable in $V$, contradicting
\ref{lem:extendermaximality}.
\end{proof}

\begin{proof}[Proof of \ref{thm:strategy_coh_card}]
Because $\eta$ is regular in $N$ and by our assumptions about $\cof^\R(\tau)$
(when $\R$ has a largest cardinal $\tau$),  there is $\gamma<\eta$ and an $\om$-maximal tree $\Vv$ on $N|\gamma$
equivalent to $\Ww$. (Choose $\gamma$ so that: For $\alpha<\lh(\Ww)$, if $[0,\alpha]_\Ww$ drops let $Q_\alpha=M^\Ww_\alpha$
and otherwise let $Q_\alpha=i^\Ww_{0\alpha}(N|\gamma)$; then $M^\Vv_\alpha=Q_\alpha$
and the factor map $\sigma_\alpha\colon M^{\Vv}_\alpha\to Q_\alpha$ is just the identity.)
Because $\lh(\Ww)<\eta$, we have
$\Vv,E\in\R$. We set $M=\cHull_{k+1}^R(\{q\})$ with an appropriate $q,k$,
and $\pi:M\to R$ the uncollapse with $\Vv,E\in\rg(\pi)$. The
rest is much like before.
\end{proof}

\begin{proof}[Proof of \ref{thm:cohering_simple}]
This is basically like for \ref{thm:strategy_coh_proj},
but the iterability is established in $V$, not in $N$.
Let $\gamma<\OR^N$ with $\Ww,E\in\J(N|\gamma)$.
Let $x\in N|\gamma$ with $\Ww,E$ definable over $N|\gamma$ from $x$.
Let $k$ be sufficiently large and let
$q\in N|\gamma$ witness \ref{lem:sound_hull} with
respect to $(N|\gamma,k,x,\theta=\om)$. Let
\[ M=\cHull_{k+1}^{N|\gamma}(\{q\})\]
and $\pi:M\to N|\gamma$ be the uncollapse.
Let $\pi(\bar{x})=x$ etc.
Let $\Wwbar,\Ebar$, defined over $M$ from $\bar{x}$,
be the preimages of $\Ww,E$.
Let $\Xx$ be the $k$-maximal tree on $N|\gamma$ determined by $\Ww$;
so $\Xx$ is via the strategy induced by $\Sigma$.
Similarly, let $\Xxbar$ be the $k$-maximal tree on $M$ determined by $\Wwbar$.
So $\Phi(\Xxbar)$ is $(\theta^++1)$-iterable (in $V$), since $\pi$ lifts its models appropriately
to models of $\Phi(\Xx)$.
And
$(k,\Xxbar,\zetabar,\Pbar^*,\Ebar,\Xxbar^+)$ is potential for
$M$ (with the obvious definitions for $\Pbar^*,\Xxbar^+$).
Let $\bar{U}=M^{\Xxbar^+}_\infty$. Then
$\bar{U}$ is wellfounded and the phalanx $\pP=\Phi(\Xxbar,\iotabar,\Ebar)$ is
$(\theta^++1)$-iterable, where $\iotabar=\lgcd(\Pbar^*||\lh(\Ebar))$.
For we can find
$M'\pins P|\lh(E)$ and an appropriate lifting map $\psi:\bar{U}\to M'$,
essentially as in the proof of \ref{thm:strategy_coh_proj}.
So by \ref{lem:extendermaximality},
$\Ebar\in\es_+^{\Pbar}$, so $E\in\es_+^P$, as desired.
\end{proof}

None of the remainder of \S\ref{sec:cohering} is needed in
\S\S\ref{sec:stacking},\ref{sec:bicephali}.

We next proceed toward the proof of
\ref{thm:cohering}. But first we need to
make its statement precise, in terms of the manner in which
$\Ww$ and $E$ are definable over $N$. This involves coding
iteration trees on $N$, definably over $N$.

\begin{rem}\label{rem:type3}
Squashing of premice complicates our tree coding. If $N$ is type 3,
we will only have direct representations for elements of $N^\sq$. But a tree
$\Tt$ on $N$ can have $\nu(F^N)<\lh(E^\Tt_0)<\OR(N)$, so
$E=E^\Tt_0\notin N^\sq$. One option is to represent such an $E$ with a
pair $(a,f)$ such that $i_{F^N}(f)(a)=E$. 
But instead, we will adjust the rules of normal iteration trees, to
rule out such $\Tt$.\end{rem}

\begin{dfn}\label{dfn:adjusted}
Given an iteration tree $\Tt$ and $\alpha+1<\lh(\Tt)$,
$\alpha$ is \emph{$\Tt$-exceptional} iff $M^\Tt_\alpha$ is type 3 and
$\nu(M^\Tt_\alpha)<\lh(E^\Tt_\alpha)<\OR(M^\Tt_\alpha)$. A tree $\Tt'$ on a
$k$-sound premouse $N$ is \emph{pre-adjusted $k$-maximal} iff there is a
$k$-maximal
tree $\Tt$ on $N$ such that $\Tt'$ has index set
\[ \lh(\Tt)\un\{(\alpha,0)\mid\alpha+1<\lh(\Tt)\ \&\ \alpha\text{\ is\
}\Tt\text{-exceptional}\}, \]
ordered with $\alpha<(\alpha,0)<\beta$ for $\alpha<\beta$,
and if $\alpha$ is $\Tt$-exceptional, then
$E^{\Tt'}_\alpha=F(M^\Tt_\alpha)$ and $E^{\Tt'}_{(\alpha,0)}=E^\Tt_\alpha$,
and otherwise $E^{\Tt'}_\alpha=E^{\Tt}_\alpha$, and
$<_{\Tt'}$, $D^{\Tt'}$, $\deg^{\Tt'}$ are determined as for $k$-maximal trees.

Given $N,\Tt,\Tt'$ as above, $\adjust(\Tt)$ denotes the
unique tree $\Uu$ on $N$ with $\Uu\iso\Tt'$ and the index set of $\Uu$ an
ordinal. Such trees $\Uu$ are \emph{adjusted $k$-maximal}.

Given $\Ww,P,\gamma,E,\Ww^+$ as in \ref{dfn:potentialtree}, we define the
corresponding \emph{adjusted potential $k$-maximal tree}
$\Uu=\adjust(\Ww^+)$ similarly. In particular, if
$P$ is type 3 and $\nu(F^P)<\gamma<\OR(P)$, then for some $\alpha$,
$\lh(\Uu)=\alpha+3$, $E^\Uu_\alpha=F^P$ and $E^\Uu_{\alpha+1}=E$.
\end{dfn}

\begin{rem}
We make some remarks on the preceding definition; see \cite{copy_con} for more details. We 
use notation as in \ref{dfn:adjusted}. For
$\Tt'$-indices $x<y$ with $E^{\Tt'}_y$ defined, we have
$\nu^{\Tt'}_x\leq\nu^{\Tt'}_y$, and if $\nu^{\Tt'}_x=\nu^{\Tt'}_y$ then
$x=\alpha$ and $y=(\alpha,0)$ for some $\Tt$-exceptional $\alpha$. The tree
structure of $\Tt'\rest\OR$ is essentially that of $\Tt$;  but for example if
$\pred^\Tt(\beta+1)=\alpha$ and $\alpha$ is 
$\Tt$-exceptional and
$\nu(M^\Tt_\alpha)\leq\crit(E^\Tt_\beta)$ then
$\pred^{\Tt'}(\beta+1)=(\alpha,0)$. (Thus, we also might have $\gamma<\lh(\Tt)$ such that 
$\gamma>_\Tt\beta+1>_\Tt\alpha$ but $\gamma\not>_{\Tt'}\alpha$.) 
In this case, $\Tt$ drops in model at $\beta+1$
(but with $M^{*\Tt}_{\beta+1}\notin(M^\Tt_\alpha)^\sq$), and
$M^{*\Tt'}_{\beta+1}=M^{*\Tt}_{\beta+1}$. So for each $\alpha<\lh(\Tt)$,
$M^{\Tt'}_\alpha=M^\Tt_\alpha$, and if $\alpha<_\Tt\beta$ and $(\alpha,\beta]_\Tt$ does not drop in 
model then $\alpha<_{\Tt'}\beta$ and $(\alpha,\beta]_{\Tt'}$ does not drop in model and 
$i^{\Tt'}_{\alpha\beta}=i^\Tt_{\alpha\beta}$. Moreover, for $\Tt$-exceptional $\alpha$, whenever 
$\Tt'$ forms an ultrapower of some $P\ins M^{\Tt'}_{(\alpha,0)}$, we have
$P\pins M^{\Tt'}_{(\alpha,0)}|\OR(M^\Tt_\alpha)$.

The map $\Tt\mapsto\adjust(\Tt)$ is well-defined and 1-1, and, for e.g.,
$(k,\om_1,\om_1+1)$-iterability is equivalent to the corresponding iterability for adjusted 
$k$-maximal trees. By coding only adjusted
$k$-maximal trees, we avoid the problem of
\ref{rem:type3}.

Given an iteration tree $\Tt$ on a premouse $N$, and
$\alpha+1<\lh(\Tt)$, let $\lambda^\Tt_\alpha=\nu(E^\Tt_\alpha)$ if
$M^\Tt_\alpha$ is type 3 and $E^\Tt_\alpha=F(M^\Tt_\alpha)$, and
let
$\lambda^\Tt_\alpha=\lh(E^\Tt_\alpha)$ otherwise. So
$\nu(E^\Tt_\alpha)\leq\lambda^\Tt_\alpha$. If $\Tt$ is adjusted
$k$-maximal, and $\alpha<\beta<\lh(\Tt)$, then $\lambda^\Tt_\alpha$ is a
cardinal of $M^\Tt_\beta$,
$M^\Tt_\alpha||\lambda^\Tt_\alpha=M^\Tt_\beta|\lambda^\Tt_\beta$, and if
$\beta+1<\lh(\Tt)$ then
$\lambda^\Tt_\alpha\leq\nu(E^\Tt_\beta)$.\end{rem}

\ifbool{standardmode}{}
{\begin{olddfn}\end{olddfn}
\begin{oldlem}\end{oldlem}}

The precise meaning of the following lemma is described in the definition to follow:

\begin{newlem}[Coding of finite trees]\label{lem:codes}
Let $N$ be a $k$-sound premouse and $\Tt$ be a finite, adjusted $k$-maximal tree
on $N$.
In the codes, the models and embeddings of $\Tt$ are $\bfrDelta_{k+1}^N$,
and the $\rSigma_{\deg^\Tt(n)+1}$ satisfaction relation for $M^\Tt_n$ is
$\bfrSigma_{k+1}^N$.\footnote{The $\rDelta_{k+1}^N$,
$\rSigma_{\deg^\Tt(n)+1}$
and $\rSigma_{k+1}^N$ referenced here are all pure, not generalized. Likewise
elsewhere in this paper, unless explicitly stated otherwise.}\end{newlem}

\begin{dfn}\label{dfn:codes}
Let $N,k,\Tt$ be as as above. We define
$\bfrDelta_{k+1}^N$ relations
\[ \M_n(x),\ \equality_n(x,y),\ \Ein_n(x,y),\ \es_n(x),\ F_n(x),\ i_{nm}(x,y)\]
(e.g. $\M_n(x)$ is in two variables $n,x$) such that properties
\ref{a}-\ref{last} below hold. (The intended meaning is as follows: $\M_n\sub\core_0(N)$
is a class of codes for elements of $\core_0(M^\Tt_n)$,
and $\equality_n,\Ein_n,\es_n,F_n$ represent equality,
membership, the internal extender sequence and active extender predicates of $\core_0(M^\Tt_n)$
with respect to the coding, and $i_{nm}$ codes $i^\Tt_{mn}$.) We will have:
\begin{enumerate}[label=(\arabic*),ref=(\arabic*),series=codelist]
\item\label{a} $\M_n\sub\core_0(N)$, and $\equality_n$ is an equivalence relation on $\M_n$.
\item $\Ein_n\sub \M_n\cross \M_n$ and $\es_n,F_n\sub \M_n$, and
$\Ein_n,\es_n,F_n$ each respect $\equality_n$.
\item $M^\Tt_n=(|M^\Tt_n|,\in,\es(M^\Tt_n),F(M^\Tt_n))\iso(\M_n,\Ein_n,
\es_n,F_n)/\equality_n$.
\item $i_{nm}\neq\emptyset$ just when $i^\Tt_{nm}$ is
defined, and in this case, $i_{nm}\sub \M_n\cross \M_m$ and respects
$(\equality_n,\equality_m)$, and $i^\Tt_{nm}\iso
i_{nm}/(\equality_n,\equality_m)$.
\end{enumerate}

We start with
$(\M_0,\equality_0,\Ein_0,\es_0,F_0)=(\core_0(N),{=},{\in},\es^{\core_0(N)},F^{\core_0(N)})$ and
$i_{00}=\id:\M_0\to\M_0$. Suppose we have defined $\M_n$, etc.
We make some definitions.
Let
\[ \picode_n:\M_n\to \core_0(M^\Tt_n)\] be the natural surjection, that is,
$\picode_n(x)$ is the image of $[x]_{\equality_n}$ under Mostowski collapse. 
If $\picode_n(x)=x'$,
we
say $(x,n)$ is a \emph{code} for $(x',n)$ (or just $x$ is a code for $x'$ if
$n$ is understood).

Let $\ttt=\ttt^\Tt=(<_\Tt,\dropset^\Tt,\deg^\Tt,d^\Tt)$ where\footnote{$d^\Tt$
is presently redundant, but included for consistency with later notation for transfinite trees.}
\[ d^\Tt=\{(\alpha,\beta)\in\lh(\Tt)^2\mid \alpha<_\Tt\beta\text{ and }(\alpha,\beta]_\Tt\inter D^\Tt\neq\emptyset\}.\]
Let $\lh(\ttt)=\lh(\Tt)$.
Let $\rho_{k-1}^{-N}=\rho_{k-1}^N$ if
$\rho_{k-1}^N<\rho_0^N$, and $\rho_{k-1}^{-N}=\emptyset$ otherwise.
We
say $p$ is \emph{good for $(\Tt,m)$}\footnote{We have encoded in good $p$
more information than strictly necessary, so as to skip some calculations.} iff
$m<\lh(\Tt)$ and
\[
p=(p_k^N,u_{k-1}^N,\rho_{k-1}^{-N},\ttt,\xvec)
\]
for some $\xvec=(x_0,\ldots,x_{\lh(\Tt)-2})\in N^{\lh(\Tt)-1}$ and for each $n<m$,
we have (i) $\M_n(x_n)$,
(ii) if $E^\Tt_n=F(M^\Tt_n)$ then $\picode_n(x_n)=\emptyset$, and
(iii) if $E^\Tt_n\neq F(M^\Tt_n)$ then
$\picode_n(x_n)=\lh(E^\Tt_n)$. We say $p$ is \emph{good for $\Tt$}
iff $p$ is good for $(\Tt,\lh(\Tt)-1)$. Note that ``good for $(\Tt,m)$''
makes reference to $\M_n$, etc, only for $n<m$.
We require $p$ to be good for $(\Tt,m)$
in order to define $\M_m$, $\equality_m$, etc.
It will become clear that we can find $p$ which is good for $\Tt$.

We continue with the properties of the coding,
giving upper bounds for its complexity, and
the complexity of the satisfaction relation for models of $\Tt$.
In order to assist in the calculations,
we also define functions which
yield standard parameters and translation procedures associated to $\Tt$,
in the codes:
$\critical_n$ is a code for $\crit(E^\Tt_n)$;
$\starlevel_{n+1}$ is a code for $\OR(M^{*\Tt}_{n+1})$;
given $m\neq n$,
$\shiftsupport$ translates codes for ordinals within the support of $E^\Tt_{\min(m,n)}$
between $\M_m$ and $\M_n$;
$\definition$  converts a given definition  over $\core_0(M^\Tt_{n+1})$ for a set $A\sub\crit(E^\Tt_n)$
 to a definition for $A$ over $\core_0(M^{*\Tt}_{n+1})$
 (both definitions are of the relevant complexity, from parameters);
and $\meas$ converts a given finite set $a$ of $E^\Tt_n$-generators
into a  $\bfSigma_1^{M^{*\Tt}_{n+1}}$-definition of $(E^\Tt_n)_a$.
The latter two functions are just effective implementations of the proofs
of \cite[4.5]{fsit} and the Closeness Lemma \cite[6.1.5]{fsit},
which the reader should probably have in mind:
\begin{enumerate}[resume*=codelist]
\item The definitions of $\M_n,\equality_n$, etc, are $\rDelta_{k+1}^N(\{p\})$ in any
parameter $p$ good for $(\Tt,n)$; the definitions used only depend on $k$ and the
type of $N$.
 \item \emph{Satisfaction}: Let
$\satisfaction_n(\varphi,x)$ iff $n<\lh(\Tt)$, $\varphi$ is an $\rSigma_{\deg^\Tt(n)+1}$ formula,  $x=(x_0,\ldots,x_{\ell-1})$ for some $\ell<\om$,
$\M_n(x_i)$ holds for each $i<\ell$, and
\[ \core_0(M^\Tt_n)\sats\varphi(\picode_n(x_0),\ldots,\picode_n(x_{\ell-1})).\]
 Then $\satisfaction$
is $\rSigma_{k+1}^N(\{p\})$ for good $p$, uniformly.

The set of
triples $(n,\varphi,x)$ such that
\[ n<\lh(\Tt),\ \dropset_{\deg}^\Tt\inter(0,n]_\Tt\neq\emptyset\text{ and }\satisfaction_n(\varphi,x)\]
is $\rDelta_{k+1}^N(\{p\})$ for good $p$, uniformly.
\item \emph{Parameters}: There are functions
$\critical$, $\starlevel$, uniformly $\rSigma_{k+1}^N(\{p\})$ for good $p$, such
that:
\begin{enumerate}
\item $\dom(\critical)=\lh(\Tt)-1$. We have
$\M_n(\critical_n)$ and $\picode_n(\critical_n)=\crit(E^\Tt_n)$.
 \item $\dom(\starlevel)=\lh(\Tt)\cut\{0\}$. If $n+1\notin
\dropset^\Tt$ then $\starlevel_{n+1}=\emptyset$; otherwise
\[ \M_{\pred^\Tt(n+1)}(\starlevel_{n+1})\text{ holds and } 
\picode_{\pred^\Tt(n+1)}(\starlevel_{n+1})=\OR
(M^{*\Tt}_{n+1}).\]
\end{enumerate}
\item\label{last} \emph{Translations}: There are functions $\shiftsupport$,
$\definition$,
$\meas$, uniformly $\rSigma_{k+1}^N(\{p\})$ for good $p$, such
that:
\begin{enumerate}
 \item \emph{Shifting generators}: $\dom(\shiftsupport)$ is the set of tuples
$(m,n,x)$ such
that
\begin{enumerate}[label=(\roman*)]
\item $m,n<\lh(\Tt)$ and $m\neq n$ and $\M_m(x)$, and
\item $\picode_m(x)\in\lambda^\Tt_{\min(m,n)}$.
\end{enumerate}

If
$\shiftsupport_{mn}(x)=y$ then $\M_n(y)$ and $\picode_n(y)=\picode_m(x)$.
\item \emph{Definitions of sets; cf.} \cite[4.5]{fsit}:  We have $\dom(\definition)=$
\[ \{(n+1,x,\varphi)\mid n+1<\lh(\Tt)\ \&\ \M_{n+1}(x)\ \&\ \varphi 
 \text{ is }
\rSigma_{\deg^\Tt(n+1)+1}\}.\]
Let $m=\pred^\Tt(n+1)$. Then $\definition_{n+1}(x,\varphi)=(x',\varphi')$ where:
\begin{enumerate}[label=(\roman*)]
\item $\M_m(x')$ with $\picode_m(x')\in\core_0(M^{*\Tt}_{n+1})$,
\item $\varphi'$ is an $\rSigma_{\deg^\Tt(n+1)+1}$ formula, and
\item for all $\alpha<\crit(E^\Tt_n)$, we
have
\[ \core_0(M^\Tt_{n+1})\sats\varphi(\picode_{n+1}(x),\alpha)\iff
\core_0(M^{*\Tt}_{n+1})
\sats\varphi'(\picode_m(x'),\alpha). \]
\end{enumerate}

\item \emph{Definitions of measures; cf.} \cite[6.1.5]{fsit}:
We have $\dom(\meas)=$
\[ \{(n,a)\mid n+1<\lh(\Tt)\ \&\ \M_n(a)\ \&\ \picode_n(a)\in[\lambda^\Tt_n]^{<\om}\},\]
and $\meas_n(a)=(x,\varphi)$ is such that, letting $k=\pred^\Tt(n+1)$ and $C=\core_0(M^{*\Tt}_{n+1})$, we have:
\begin{enumerate}[label=(\roman*)]
\item $\M_k(x)$ and $\picode_k(x)\in C$
and $\varphi$ is $\rSigma_1$, and
\item $(E^\Tt_n)_{\picode_n(a)}=\{A\in C\mid C\sats\varphi(A,\picode_k(x))\}$.
\end{enumerate}
\end{enumerate}
\end{enumerate}

This completes the list of properties.
Now suppose $p$ is good for $(\Tt,n)$,
and we have defined the following things:
\begin{itemize}[label=--]
 \item  $(\M_j,\equality_j,\Ein_j,\es_j,F_j)$, $i_{jk}$ and $\satisfaction_j$ for $j,k\leq n$,
\item $\critical_j$, $\meas_j$, $\starlevel_{j+1}$, $\definition_{j+1}$ for $j<n$, and
\item $\shiftsupport_{jk}$ for $j,k\leq n$ with $j\neq k$.
\end{itemize}
Suppose $n+1<\lh(\Tt)$, and $p$ is actually good for $(\Tt,n+1)$. We want to define $\M_{n+1}$ etc.
We have $\M_n(x_n)$ and $x_n$ determines $E^\Tt_n$. Let $\critical_n$ be the natural code in $\M_n$,
determined by $x_n$, for $\crit(E^\Tt_n)$. Let $k=\pred^\Tt(n+1)$;
note that $p$ determines $k$ and determines whether $k\in\dropset^\Tt$.
If $k=n$ let $c=\critical_n$; otherwise
let\[ c=\shiftsupport_{nk}(\critical_n).\]
So $\M_k(c)$ and $\picode_k(c)=\crit(E^\Tt_n)$
in general. If $k\in\dropset^\Tt$, let $\starlevel_{n+1}$ be the natural code in
$\M_k$ for $\OR(M^{*\Tt}_{n+1})$, determined by $c$ and $x_k$. Using these codes, take $\M_{n+1}(x)$ to
be the natural formula asserting ``$x=((q,t),b)$ where:
\begin{enumerate}[label=--]
 \item $\M_k(q)$ and $\picode_k(q)\in\core_0(M^{*\Tt}_{n+1})$,
 \item $t$ is a generalized $\deg^\Tt(n+1)$ term,
 \item $\M_n(b)$ and $\picode_n(b)\in[\lambda^\Tt_n]^{<\om}$.''
\end{enumerate}
Here we use $\satisfaction_k$ to assert ``$\picode_k(q)\in\core_0(M^{*\Tt}_{n+1})$'', etc.
Now to define $\equality_{n+1}$, $\satisfaction_{n+1}$, etc, we  Los' Theorem.
For this we need to translate statements about measure one sets
of $E^\Tt_n$ to statements over $M^\Tt_k$. For this, the main arguments follow the proofs of
\cite[4.5]{fsit} and \cite[6.1.5]{fsit}, observing that everything is
sufficiently effective, and in particular, that we get $\rSigma_{k+1}^N(\{p\})$ functions.
We will use degree $k+1$ min-terms to uniformize $\rSigma_{k+1}^N(\{p\})$ relations
and obtain our functions.
We will just sketch the definition of
$\meas_n$, assuming $k<n$; the case $k=n$ is much easier. Let $\kappa=\crit(E^\Tt_n)$, so $\kappa<\nu(E^\Tt_k)$.

Assume $E^\Tt_n\neq F(M^\Tt_n)$, so $\lambda^\Tt_n=\lh(E^\Tt_n)$. Given
 $a\in\M_n$ with $\picode_n(a)\in[\lambda^\Tt_n]^{<\om}$, we can identify
via $\rDelta_1$-satisfaction for $M^\Tt_n$ (and our good parameter $p$),
those codes $c\in\M_n$
such that $\gamma=\picode_n(c)$
 is the position of $(E^\Tt_n)_{\picode_n(a)}$ in the canonical wellorder of $M^\Tt_n$
(so $\gamma\in(\kappa^{++})^{M^\Tt_n}$).

Using $\shiftsupport$, such codes $c$ are translated to
$\M_{k+1}$-codes $c'$ for
$\gamma$, since if $k+1<n$ then
$(\kappa^{++})^{M^\Tt_n}\leq\lambda^\Tt_{k+1}$. Write
$c'=((q,t),b)$.
Let $\mu=\crit(E^\Tt_k)$ and
\[ \ell=\pred^\Tt(k+1).\]
Now
$\gamma<(\kappa^{++})^{M^\Tt_n}\leq(\kappa^{++})^{M^\Tt_{k+1}}$, so there is
$f:[\mu]^{|\picode_k(b)|}\to\mu$ such that $f\in \core_0(M^{*\Tt}_{k+1})$ and
\[ i^{*\Tt}_{k+1}(f)(\picode_k(b))=\gamma=t^{M^\Tt_{k+1}}(i^{*\Tt}_{k+1}
(\picode_\ell(q)),
\picode_k(b)).\]
We can identify the $\M_\ell$-codes $d$ for such functions $f$
sufficiently effectively; the main
point here
is that we can refer to $\satisfaction_{k+1}$ (or just $\equality_{k+1}$)
to check equality (we also have the $\M_k$-code $\crit_k$ for $\mu$,
so can check that $f:[\mu]^{|\picode_k(b)|}\to\mu$). Such $\M_\ell$-codes $d$
can then be converted to $\M_k$-codes $d'$ using $\shiftsupport$.

The set of tuples $(a,b,d')$ as above is
$\rSigma_{k+1}^N(\{p\})$, so there is an $\rSigma_{k+1}^N(\{p\})$ function
selecting some such $b,d'$ as a function of $a$.
Let $b^*=\picode_k(b)$ and $f=\picode_k(d')$.
Note that
\[ (E^\Tt_n)_{\picode_n(a)}\text{ is }\rSigma_1^{M^\Tt_k|\lh(E^\Tt_k)}(\{b^*,f\}),\]
and uniformly so.
Now $M^\Tt_k|\lh(E^\Tt_k)\ins M^{*\Tt}_{n+1}$
and $p$ incorporates the code $x_k$ for $\lh(E^\Tt_k)$, so this almost suffices.
However, it might be that $M^{*\Tt}_{n+1}$ is active type 3 and $M^\Tt_k|\lh(E^\Tt_k)\pins M^{*\Tt}_{n+1}$
but $M^\Tt_k|\lh(E^\Tt_k)\nins\core_0(M^{*\Tt}_{n+1})$,
in which case the relevant parameters are not available directly to $\core_0(M^{*\Tt}_{n+1})$.
But because $\Tt$ is adjusted $k$-maximal, note that $M^{*\Tt}_{n+1}\pins M^\Tt_k$,
and then it is easy to convert our (parameters for) definitions of measures over $M^\Tt_k|\lh(E^\Tt_k)$
into (parameters for) definitions for those measures over $\core_0(M^{*\Tt}_{n+1})$, using the method above.
This completes the definition of $\meas_n(a)$ assuming $E^\Tt_n\neq F(M^\Tt_n)$.

If instead $E^\Tt_n=F(M^\Tt_n)$, use the proof of
\cite[6.1.5]{fsit}, effectivized through the function $\definition$,
combined with the preceding argument. We leave the remaining details to the reader.

Using standard calculations with Los Theorem,
it is now straightforward to define
$\equality_{n+1}$, etc,
and $\satisfaction_{n+1}$, in terms of $\satisfaction_k$ and $\meas_n$,
and see the relevant properties.
The definition of $\satisfaction_{n+1}$
also easily yields $\definition_{n+1}$. For $\shiftsupport_{n,n+1}$,
we convert an $\M_n$-code $c$ for an ordinal $<\lambda^\Tt_n$ to the $\M_{n+1}$-code $((\dot{\emptyset},t_\id),c)$,
where $\dot{\emptyset}$ is the natural $\M_k$-code for the empty set and $t_\id$ the term for the identity map.
For $\shiftsupport_{n+1,n}$, an $\M_{n+1}$-code $((q,t),b)$ for an ordinal $\alpha<\lambda^\Tt_n$
can be converted to some $\M_n$-code for $\alpha$, using min-term uniformization to make the selection.
\end{dfn}

To summarize, we have:
\begin{newlem}
 Let $N,k,\Tt$ be as above. Then there is $p$ which is good for $\Tt$.
For any $p$ which is good for $\Tt$, the coded satisfaction relation $\satisfaction$,
defined relative to $p$, is $\rSigma_{k+1}(\{p\})$.
\end{newlem}

We next extend the coding to infinite trees $\Tt$. The plan is as follows.
A code for an object $z\in\core_0(M^\Tt_\alpha)$
will be of the form $(\alpha,x)$, and $x$ will essentially consist
of some finite support for $z$, in terms of functions and generators coming from finitely many earlier models $M^\Tt_\beta$.
We demand that the
full tree structure $\ttt$ of $\Tt$, and a function $\lhfunc$ specifying the sequence of extenders used in the tree
are given in advance, with both suitably definably over $N$ from some parameter.
The specification of the extenders must naturally be given through the coding,
analogous to the finite tree case; that is,
$(\alpha,\lhfunc(\alpha))$ will be  a code for $\lh(E^\Tt_\alpha)$ (or for $\emptyset$ if $E^\Tt_\alpha=F(M^\Tt_\alpha)$).
(So $\lhfunc$ is the analog of the parameter $\xvec$ used in the finite case.)
We now proceed to the details.

\begin{dfn}\label{dfn:coherent}
Let $N$ be a $k$-sound premouse. Let $1\leq\lambda\leq\rho_0^N$ and
\[ \ttt=(<_\ttt,D^\ttt,\deg^\ttt,d^\ttt) \]
where $<_\ttt$ is an iteration tree order,
$D^\ttt$ a drop structure and $\deg^\ttt$ a degree structure, each on $\lambda$,
and $d^\ttt$ as before.

Suppose $\lambda<\om$. For $\xvec\in N^{\lambda-1}$ let
\[ p=p(\ttt,\xvec)=(p_k^N,u_{k-1}^N,\rho_{k-1}^{-N},\ttt,\xvec); \]
We say $p$ is
\emph{coherent \textup{(}for $N$\textup{)}} iff there is an adjusted $k$-maximal
tree $\Tt$ on $N$ such that $p$ is good for $\Tt$.

Now remove the restriction that $\lambda<\om$, but
suppose that $\ttt$ is $\bfrDelta_{k+1}^N$. Fix an
$\bfrSigma_{k+1}^N$ function $\lhfunc$ with domain $\lambda$.
The intention is that $\lhfunc$ is a sequence of codes for extender lengths, for a tree $\Tt$
with structure $\ttt$. The elements $z\in\core_0(M^\Tt_\alpha)$ are coded with pairs $(\alpha,x)$,
where $x$ describes a finite support for $z$ in terms of codes for earlier models in $\Tt$.
We define $\Tt\rest(\alpha+1)$, and the collection $C_\alpha$ of all $x$ such that $(\alpha,x)$ is a code,
and the interpretation $\picode(\alpha,x)\in\core_0(M^\Tt_\alpha)$, by recursion on $\alpha$.

First, $C_0=\core_0(N)$, and $\picode(0,x)=x$.

Now let $\eta\geq 1$, and suppose:
\begin{itemize}[label=--]
 \item we have defined the adjusted $k$-maximal tree $\Tt\rest\eta$ on $N$, with structure $\ttt\rest\eta$,
 \item for all $\beta<\eta$, we have defined $C_\beta$ and $\picode(\beta,x)$ for all $x\in C_\beta$, and
\[ \core_0(M^\Tt_\beta)=\{\picode(\beta,x)\mid x\in C_\beta\}.\]
\end{itemize}
Then we say that $(\ttt,\lhfunc)$ is \emph{$\eta$-coherent}. (Note that $1$-coherence is trivial.)

Suppose $\eta$ is a limit and $\eta<\lambda$. We set $[0,\eta)_\Tt=[0,\eta)_\ttt$.
We set $C_\eta$ to be the set of all pairs $(\beta+1,x)$ such that $\beta+1<_\ttt\eta$
and $(\beta+1,\eta)_\ttt$ does not drop in model and $x\in C_{\beta+1}$.
We define
\[ \picode(\eta,(\beta+1,x))=i^\Tt_{\beta+1,\eta}(\picode(\beta+1,x)).\]
Assuming that $M^\Tt_\eta$ is wellfounded, note that $(\eta+1)$-coherence follows.

Suppose $\eta=\alpha+1<\lambda$ and that $\lhfunc(\alpha)\in C_\alpha$.
Let $\gamma=\picode(\alpha,\lhfunc(\alpha))$. Suppose either:
 \begin{itemize}[label=--]
  \item $\gamma=0$ and $E=F(M^\Tt_\alpha)\neq\emptyset$, or
  \item $\gamma<\rho_0(M^\Tt_\alpha)$ and $E=F(M^\Tt_\alpha|\gamma)\neq\emptyset$.
 \end{itemize}  
Then we set $E^\Tt_\alpha=E$; suppose this determines an adjusted $k$-maximal
tree $\Tt\rest(\alpha+2)$ (including wellfoundedness).
Let $\beta=\pred^\ttt(\alpha+1)$ and $M^*=M^{*\Tt}_{\alpha+1}$.
Then $C_{\alpha+1}$ is the set of pairs of the form
\[ x=((\beta,q,u),(\alpha,a)) \]
such that $q\in C_\beta$ and
\[ q'=\picode(\beta,q)\in\core_0(M^*), \]
$u$ is an $\rSigma_n$-Skolem term, where $n=\deg^\Tt(\alpha+1)$,
$a\in C_\alpha$ and
\[ a'=\picode(\alpha,a)\in[\lambda^\Tt_\alpha]^{<\om}.\]
For these objects, we define
\[ \picode(\alpha+1,x)=[f_{q',u},a']^{M^*,n}_{E^\Tt_\alpha}.\]
Note then that $(\eta+1)$-coherence follows.

A \emph{code} (relative to $(\ttt,\lhfunc)$) is a pair $(\alpha,x)$ such that $(\ttt,\lhfunc)$ is $(\alpha+1)$-coherent
and $x\in C_\alpha$. (So $\picode(\alpha,x)$ is defined for codes $(\alpha,x)$.)
We say that $(\ttt,\lhfunc)$ is \emph{coherent} iff it is $\lambda$-coherent.
\end{dfn}

\begin{newrem}\label{rem:hsstm_adap}
For the definability of the coding, and the associated satisfaction relation, 
etc, we will use the fact that adjusted $k$-maximal trees $\Tt$ are
natural direct limits of finite such trees $\bar{\Tt}$.
In \cite[\S2.3]{hsstm}, finite supports for $k$-maximal iteration trees were discussed.
We assume the reader is familiar with such methods.
Adjusted trees were not considered there,
but it is straightforward to adapt those methods  to adjusted trees.
In particular, we have that for each adjusted $k$-maximal tree
$\Tt$ of length $\lambda$ and for each finite set $J\sub\lambda\cross V$
with $J_\alpha\sub\core_0(M^\Tt_\alpha)$, there is a \emph{finite support} $S\sub\lambda\cross V$
(defined much as in \cite{hsstm}) with $J\sub S$.
Moreover, for any such $S$, letting $\bar{\Tt}=\bar{\Tt}_S$ be the
corresponding finite hull of $\Tt$ (finite meaning that $\lh(\bar{\Tt})$ is finite),
the natural ``copy map''
\[ \pi_{\bar{\alpha}}:M^{\bar{\Tt}}_{\bar{\alpha}}\to M^\Tt_{\alpha} \]
is a near $\deg^\Tt(\alpha)$-embedding
(here $\alpha\in I_S$ where $I_S$ is the projection of $S$ on the left coordinate,
and $\deg^{\bar{\Tt}}(\bar{\alpha})=\deg^\Tt(\alpha)$).\footnote{However, as in \cite{hsstm},
those $\alpha\in I_S$ such that $\alpha<\max(I_S)$ but $\alpha+1\notin I_S$,
do not correspond to any $\bar{\alpha}$; this is also discussed further below.}
The iteration maps commute in the obvious manner with the maps $\pi_{\bar{\alpha}}$.

If $\Tt$ is our coded tree, then given $\bar{\Tt},\bar{\alpha},\alpha$ as above and $\bar{z}\in\core_0(M^{\bar{\Tt}}_{\bar{\alpha}})$
and $z=\pi_{\bar{\alpha}}(\bar{z})$, we will convert 
the code $(\alpha,x)$ for $z$ into a code $c$ for $\bar{z}$;
here $c$ is a code in the sense of our earlier coding for finite trees (\ref{dfn:codes}).
We will write $c=\tc(\alpha,x)$ ($\tc$ for \emph{transitive collapse});
we remark that $c$ depends on the choice of $\bar{\Tt}$ though
(which below is determined through a \emph{support} $S$).
This code conversion will be appropriately effective, 
and since $\pi_{\bar{\alpha}}$ is a near $\deg^\Tt(\alpha)$-embedding,
we can define the $\rSigma_{\deg^\Tt(\alpha)+1}$ satisfaction relation
for $M^\Tt_\alpha$ in terms of that for such models $M^{\bar{\Tt}}_{\bar{\alpha}}$.
The details, to follow, are basically an effective version of
material from \cite[\S2.3]{hsstm}.
\end{newrem}

\begin{newdfn}\label{dfn:codes_trans}
Let $N,k,\ttt,\lhfunc$ be as in \ref{dfn:coherent} (we don't assume coherence),
with $N$ iterable for finite $k$-maximal trees.
Let $S\sub\lambda\cross\core_0(N)$ and let
$I_S$ be the projection of $S$ on its left coordinate.
We say $S$ is a \emph{support} with respect to $(\ttt,\lhfunc)$ iff:
\begin{enumerate}
 \item $S$ is finite and $0\in I_S$.
 \item If $(\alpha+1,x)\in S$ then $(\alpha,\lhfunc(\alpha))\in S$ and $x=((\beta,q,u),(\alpha,a))$ where:
\begin{itemize}[label=--]
 \item $\beta=\pred^\ttt(\alpha+1)$,
  \item $u$ is an $\rSigma_{\deg^\ttt(\alpha+1)}$ Skolem term, and
 \item  $(\alpha,a),(\beta,q)\in S$ and $\beta+1\in I_S$.
\end{itemize}
 \item If $(\alpha,x)\in S$ and $\alpha$ is a limit
 then
 \[ \max(\alpha\inter I_S)=\gamma+1<_\ttt\alpha\]
  and $x=(\beta+1,y)\in S$
 where $\beta+1<_\ttt\alpha$ and $(\beta+1,\alpha)_\ttt$ does not drop in model or degree.
 Moreover, if $\alpha+1\in I_S$ and $x=\lhfunc(\alpha)$
 then $\beta+1<\gamma+1$.\footnote{The last requirement helps ensure
 that when $S$ is a support, then $\ttt_S$ is an adjusted $k$-maximal tree;
 see \cite{hsstm} for further explanation.}
\end{enumerate}

If $(\alpha,x)\in S$, say that $(\alpha,x)$ is a \emph{potential code}.
Now we want to translate from the codes for infinite trees discussed 
in \ref{dfn:coherent} and our earlier coding
for finite trees in \ref{dfn:codes}. Toward this, we compute the \emph{transitive collapse} $\tc(\alpha,x)$ of $(\alpha,x)$,
with respect to a given support $S$ with $(\alpha,x)\in S$ (though $(\alpha,x)$ might not actually be a code in the sense of \ref{dfn:coherent}).
Let
\[ E_S=\{\alpha\in I_S\mid\alpha+1\in I_S\},\]
\[ P_S=I_S\cut(E_S\un\{\max(I_S)\}).\]
Let $\ttt_S=\ttt\rest I_S$. We
consider $\ttt_S$ as a padded tree structure, indexed with ordinals in $I_S$, with padding occurring at
precisely the ordinals in $P_S$, and extenders used at the ordinals in $E_S$.\footnote{$P$ for \emph{Padding} and $E$ for \emph{Extenders}.}
Let $\trcollt(\ttt_S)$ denote the non-padded tree structure
which is otherwise isomorphic to $\ttt_S$. We enlarge $S$
to a support which includes canonical potential codes for the empty set
with respect to each $\alpha\in I_S$,
and which is closed under (coded) images under iteration maps.
So, recursively on $\alpha\in I_S$, we define $e_\alpha$, such that $(\alpha,e_\alpha)$
codes the empty set: Set $e_0=\emptyset$.
Given $\alpha+1\in I_S$ and $\beta=\pred^\ttt(\alpha+1)$ (so $\beta\in I_S$ also), let
\[ e_{\alpha+1}=((\beta,e_\beta,\mathrm{c}),(\alpha,e_\alpha)), \]
where $\mathrm{c}=\mathrm{c}(v)$ is the Skolem term for the constant function with value $v$.
Given a limit $\alpha\in I_S$ and $\gamma+1=\max(I_S\inter\alpha)$, let
\[ e_\alpha=(\gamma+1,e_{\gamma+1}).\]
Now for $\gamma,\delta\in I_S$ such that $\gamma\leq_\ttt\delta$ and $(\gamma,\delta]_\ttt\inter D^\ttt=\emptyset$,
and given a pair of the form
$d=(\gamma,x)$,
we define a pair $d^{\gamma\delta}$ of the form $(\delta,y)$.
Set $d^{\gamma\gamma}=d$. Suppose $\gamma<\delta$.
Suppose first that $I_S\inter(\gamma,\delta)_\ttt=\emptyset$.
If $\delta=\alpha+1$ then note that $\gamma=\pred^\ttt(\alpha+1)$, and we set
\[ d^{\gamma,\alpha+1}=(\alpha+1,((\gamma,x,\mathrm{c}),(\alpha,e_\alpha))) \]
with $\mathrm{c}$ as before. If $\delta$ is a limit (so $\gamma=\max(I_S\inter\delta)$
and $\gamma$ is a successor) then we set
\[ d^{\gamma\delta}=(\delta,(\gamma,d)).\]
Now if instead $I_S\inter(\gamma,\delta)_\ttt\neq\emptyset$ then letting
$\beta=\max(I_S\inter(\gamma,\delta)_\ttt)$, we set
\[ d^{\gamma\delta}=(d^{\gamma\beta})^{\beta\delta}.\]

Now let $S'=S\cup\{(\alpha,e_\alpha)\mid \alpha\in I_S\}$ and $S''$ the closure
of $S'$ under $d\mapsto d^{\gamma\delta}$ (for $\gamma,\delta\in I_S$).
Note $S',S''$ are (finite) supports with $I_{S''}=I_{S'}=I_S$.

For $(\alpha,x)\in S''$ we will define
$\trcollt(\alpha,x)$, recursively on $\alpha$.
Here if $(\ttt,\lhfunc)$ is coherent, determining tree $\Tt$,
and $(\alpha,x)$ is a code,
then $\ttt_S$ is a finite length adjusted $k$-maximal tree $\bar{\Tt}$,
a sub-tree of $\Tt$,
and with $\M_n$ the coding of $M^{\bar{\Tt}}_n$ defined as in \ref{dfn:codes},
then we will have $\trcollt(\alpha,x)\in\M_n$. The definition is as follows:
\begin{enumerate}
 \item $\trcollt(0,x)=x$,
 \item $\trcollt(\alpha+1,((\beta,q,u),(\alpha,a)))=((\trcollt(\beta,q),u),\trcollt(\alpha,a))$,
 \item for limit $\alpha$,
$\trcollt(\alpha,d)=\trcollt(\gamma+1,d^{\beta+1,\gamma+1})$ where $d=(\beta+1,y)$ and $\gamma+1=\max(I_S\inter\alpha)$.
\end{enumerate}
Let $p(\ttt,S) = p(\trcollt(\ttt_S),\xvec)$ where
\[ \xvec=\left<x_i\right>_{i<\card(E_S)}
=\left<\trcollt(\alpha,\lhfunc(\alpha))\right>_ {\alpha\in E_S}.\]

We say $(\ttt,\lhfunc)$ is \emph{pre-coherent \textup{(}for $N$\textup{)}} iff:
\begin{itemize}[label=--]
\item for every finite $J\sub\lambda$ there is a support $S$ such that $J\sub I_S$, and
 \item  $p(\ttt,S)$ is coherent for every support $S$.
\end{itemize}

By essentially the arguments in \cite{hsstm}, if $(\ttt,\lhfunc)$ is coherent,
then it is pre-coherent.
Note that if $S_1,S_2$ are supports with $I_{S_1}\sub I_{S_2}$ then $S_1\cup S_2$
is a support. Therefore if $(\ttt,\lhfunc)$ is pre-coherent
then the collection of all supports is directed under $\sub$.

Suppose $(\ttt,\lhfunc)$ is pre-coherent. Then we can make sense of ``$M^\Tt_\alpha$''
for each $\alpha<\lambda$, even if this structure is illfounded,
as the direct limit of models $M^{\bar{\Tt}}_{\bar{\alpha}}$, where $\bar{\Tt}$ is a finite tree determined by a support $S$
and $\bar{\alpha}$ the corresponding collapse of $\alpha$.
We write $\M^*_\alpha$, etc, for the codeset for elements of $\core_0(M^\Tt_\alpha)$,
etc.
So, we define
$\rDelta_{k+1}^N(N)$ relations $\M^*_\alpha(x)$, etc, and
$\rSigma_{k+1}^N(N)$ relations $\satisfaction^*_\alpha(\varphi,x)$, etc (we
define analogues of all relations defined in \ref{dfn:codes}) as follows. For
coherent $p$, let $\M^p$, etc, denote the relations $\M$, etc, defined as in
\ref{dfn:codes} from $p$, and let $\Tt^p$ be the corresponding tree on $N$. For
$x\in N$, set $x\in\M^*_\alpha\iff$
\[ \ex S\left[S\text{ is a support for }(\ttt,\lhfunc)\ \&\
(\alpha,x)\in S\ \&\ \M^{p(\ttt,S)}_{\ot(\alpha\inter
E_S)}(\trcollt(\alpha,x))\right]. \]
The remaining relations ($\satisfaction^*_\alpha$, etc) are defined similarly.
\end{newdfn}

\begin{rem}\label{rem:lhfunctree}
Let $N,k,\ttt,\lhfunc$ be as in \ref{dfn:coherent} with $(\ttt,\lhfunc)$
pre-coherent. Then there is a unique adjusted putative $k$-maximal tree $\Tt$ on
$N$, such that:
\begin{itemize}[label=--]
 \item  $\lh(\Tt)\leq\lambda$ and either $\lh(\Tt)=\lambda$ or $\Tt$ has a last, illfounded model,
 \item $\ttt\rest\lh(\Tt)=\ttt^\Tt$,
\item for each $\alpha<\lh(\Tt)$, we have
\[ M^\Tt_\alpha\iso(\M^*_\alpha,e^*_\alpha,\es^*_\alpha,
F^*_\alpha)/\equality^*_\alpha,\]
and letting
$\varsigma^*_\alpha:\M^*_\alpha\to M^\Tt_\alpha$
be the
natural surjection, if $E^\Tt_\alpha=F(M^\Tt_\alpha)$ then
$\varsigma^*_\alpha(\lhfunc(\alpha))=\emptyset$, and otherwise
$\varsigma^*_\alpha(\lhfunc(\alpha))=\lh(E^\Tt_\alpha)$.
\end{itemize}
Moreover,
$\satisfaction^*_\alpha$ defines $\rSigma_{\deg^\Tt(\alpha)+1}$ satisfaction for
$M^\Tt_\alpha$ (in the codes), etc, as in \ref{dfn:codes}.
Here the natural surjection $\varsigma^*_\alpha$ results from the considerations
in \ref{rem:hsstm_adap}, and the fact that $\satisfaction^*_\alpha$ is correct
also follows from
\ref{rem:hsstm_adap}.

We continue to assume that $(\ttt,\lhfunc)$ is pre-coherent.
Note then that $(\ttt,\lhfunc)$ is coherent
iff the tree $\Tt$ above has wellfounded models (hence length $\lambda$).
And if $(\ttt,\lhfunc)$ is coherent then
$C_\alpha=\{x\mid M^*_\alpha(x)\}$
and
$\picode(\alpha,x)=\varsigma^*_\alpha(x)$
for all $x\in C_\alpha$ (where $C_\alpha$ and $\picode$ are as in \ref{dfn:coherent}).

The set of all coherent $p$ (in the sense for codes for finite trees)
is a Boolean combination of $\rSigma_{k+1}^N(\{q\})$ sets, where
$q=(p_k^N,u_{k-1}^N,\rho_{k-1}^{-N})$,
uniformly in
$N$ ($k$-sound, of the same type). For example, with such complexity, one can
assert that for each $n$, $\M_n(x_n)$ and $x_n$ codes either $\emptyset$ or the
index of an extender on $\es(\core_0(M^\Tt_n))$, and that drops occur exactly
where
$\ttt$ specifies. (Note that asserting that an extender does not cause a drop in model
can require an $\rPi_{k+1}^N(\{q\})$ assertion.)

Let $X$ be the set of tuples $(x,\varphi_0,\varphi_1,\varphi_2)\in\core_0(N)$ such that
each $\varphi_i$ is an $\rSigma_{k+1}$ formula and
for some $\lambda\leq\rho_0^N$:
\begin{enumerate}[label=--]
 \item $\varphi_0(x,\cdot)$ defines over $\core_0(N)$ an iteration tree structure
$\ttt$ on $\lambda$, and $\varphi_1(x,\cdot)$ defines its complement,
 \item $\varphi_2(x,\cdot,\cdot)$ 
 defines over $\core_0(N)$ a function $\lhfunc$
with domain $\lambda$, and
 \item $(\ttt,\lhfunc)$ is pre-coherent.
\end{enumerate}
Note that $X$ is $\Pi_1$-over-$\rSigma_{k+1}^N(\{q\})$, for
\[ q=(\varphi_0,\ldots,\varphi_3,x,p_k^N,u_{k-1}^N,\rho_{k-1}^{-N}), \]
uniformly in $N$. Therefore this property (when true) is passed downward under
near $k$-embeddings.

The following definition completes the description of our coding.
\end{rem}

\begin{dfn}\label{dfn:infinitegood} Let $\Tt$ be a $k$-maximal tree on $k$-sound $N$.
Then
\[ (x,\varphi_{\lhfunc},\varphi_\ttt,\varphi_{\neg\ttt})\in
\core_0(N)\] is
\emph{transfinite-good for $\Tt$} iff there is $(\ttt,\lhfunc)$ which is coherent and yields $\Tt$
as in \ref{rem:lhfunctree} and $\lhfunc,\ttt,\neg\ttt$ are
$\rSigma_{k+1}^N(\{x\})$, via the formulas
$\varphi_{\lhfunc},\varphi_\ttt,\varphi_{\neg\ttt}$.\end{dfn}

\begin{rem}\label{rem:clarification}
We now clarify hypothesis
\ref{thm:cohering}\ref{item:Wdefinable} and
$\Xx^+=\adjust(\Ww^+)$. Let
$\adjlast+2=\lh(\Xx^+)$. So
$E^{\Xx^+}_{\adjlast}=E$ and either [$\lh(E)=\OR^P$ and $P=M^{\Xx^+}_{\adjlast}$]
or
\[ \lh(E)<\OR^P\text{ and }P|\lh(E)=\core_0(M^{\Xx^+}_{\adjlast})|\lh(E).\]
Let $\Xx=\Xx^+\rest(\adjlast+1)$.
Then \ref{thm:cohering}\ref{item:Wdefinable} first asserts that if $\lh(\Xx)\geq\om$ then there is
$p$ which is transfinite-good for $\Xx$.
If $\lh(\Xx)\geq\om$, let $\M^*_{\adjlast}$ be
defined from
$p$ via \ref{dfn:infinitegood} and \ref{dfn:coherent}. If $\lh(\Ww)<\om$, let
$p$ be good for $\Xx$ and let $\M_{\adjlast}$ be defined from $p$ as in
\ref{dfn:codes}.
Let $E^*\sub\M^*_{\adjlast}$ (or $E^*\sub\M_{\adjlast}$) be the
set of codes for elements of $\widetilde{E}$ (the $P||\lh(E)$-amenable predicate for
$E$).
Then \ref{thm:cohering}\ref{item:Wdefinable} secondly asserts that $E^*$ is
$\bfrSigma_{k+1}^N$.
\end{rem}

\begin{proof}[Proof of \ref{thm:cohering}]
We follow the notation used in the preceding discussion of coding.
Let $\Xx^+,\Xx,\adjlast,p$ be as in
\ref{rem:clarification}. Let
$e\in N$ be such that ($\dagger$)  $E^*$ is $\Sigma_{k+1}^N(\{e\})$.
Using $k$-reasonableness let
$q_0\in\OR(N)^{<\om}$ witness \ref{lem:sound_hull}
with respect to
$(N,k,x,\om)$ where
\[ x=(\cc,p,\adjlast,e,p_{k+1}^N).\]
Let $M=\cHull_{k+1}^N(\{q_0\})$, let $\pi:M\to N$ be the uncollapse, let
$\pi(\pbar)=p$ and let $\Xxbar$ be defined
over $M$ from $\pbar$ as $\Xx$ is defined over $N$ from $p$; ``bars''
denote preimages under $\pi$ in general.

\setcounter{clm}{0}
\begin{clm}\label{clm:propertiesreflect} We have:
\begin{enumerate}[label=\textup{(}\alph*\textup{)},ref=(\alph*)]
 \item\label{item:propertiesreflect} The hypotheses of
\ref{thm:cohering} and properties mentioned in \ref{rem:clarification}
and \tu{(}$\dagger$\tu{)}
regarding
$N$, $\Ww$, $\Xx$, $p$, etc, excluding
``$\crit(E)<\rho_{k+1}(M^\Ww_\pr)$'', hold regarding
$M$, $\Wwbar$, $\Xxbar$, $\pbar$,  etc.
Moreover,
$\crit(\Ebar)<\rho_k(M^\Wwbar_{\prbar})$.
\item\label{item:EbarE} $\Ebar\in\es_+(\fin^\Xxbar)$ iff
$E\in\es_+(\fin^\Xx)$.
\item\label{item:TbarTcopy} $\Xxbar$ is a hull of $\Xx$ \textup{(}see\textup{
\cite{sargsyan}}\textup{)}, via hull maps
$\pi\rest\lh(\Xxbar):\lh(\Xxbar)\to\lh(\Xx)$ and $\pi_\alpha:M^\Xxbar_\alpha\to
M^\Xx_{\pi(\alpha)}$, where
$\pi_\alpha\com\picodebar_\alpha=\picode_{\pi(\alpha)}
\com\pi\rest\Mbar_\alpha$.
\item\label{item:copymaps}
$\pi_\alpha$ is a near $\deg^\Xxbar(\alpha)$-embedding.
\end{enumerate}\end{clm}
\begin{proof}[Proof sketch] The premousehood
of $(\Pbar||\lh(\Ebar),\Ebar)$ uses that the
premouse axioms are Q-formulas, $E^*$ is $\rSigma_{k+1}^N(\{e\})$,
$\M^\Xx_{\adjlast}$ is $\rDelta_{k+1}^N(\{p,\adjlast\})$, and
$\pi$ is $\rSigma_{k+1}$
elementary. The rest of \ref{item:propertiesreflect} and \ref{item:EbarE} are
similar.
Parts \ref{item:TbarTcopy} and \ref{item:copymaps} are proved by induction
through $\lh(\Xxbar)$; cf. \cite[\S2.3]{hsstm}. Here part
\ref{item:copymaps} involves an adaptation of the argument in \cite{fstim}, or alternatively uses the uniformity of the
definition of $\satisfaction^*$ in \ref{dfn:codes_trans}.\footnote{The proof of uniformity itself
used the argument of \cite{fstim}.}
\end{proof}

Let $\iota=\lrgcrd(P||\lh(E))$. Let $\pP$ be the phalanx
$\Phi(\Xxbar,\iotabar,\Ebar)$ and $\pP'$ the phalanx $\Phi(\Wwbar,\iotabar,\Ebar)$.
Note that $\pP,\pP'$ have the same last model $\Ubar=\fin^{\overline{\Xx^+}}=\fin^{\overline{\Ww^+}}$.
By \ref{lem:extendermaximality}, the following claim completes the proof.

\begin{clm}\label{clm:pPiterable}$\Ubar$ is wellfounded and $\pP'$ is
$(\om_1+1)$-iterable.\end{clm}
\begin{proof}
The $(\om_1+1)$-iterability of $\pP'$ follows easily from that
of $\pP$, and we prove the latter.
We lift trees on $\pP$ to trees on $\Phi(\Xx)$, using the
maps $\pi_\alpha$ from Claim
\ref{clm:propertiesreflect}\ref{item:TbarTcopy}, and
a weak $k$-embedding
$\psi:\Ubar\to M'$
with some $M'\pins P|\lh(E)$. We will get $M'$ and $\psi$
through a variant of the earlier arguments. Let $\upsilon'=\pred^\Xx(\zeta'+1)$.
We have
\[ \crit(E)<\rho_{k+1}(M^\Xx_{\upsilon'})=\sup i^\Xx_{0\upsilon'}``\rho_{k+1}^N.
\]
Let $\mu<\rho_{k+1}^N$ be a cardinal of $N$ with $\crit(E)\leq
i^\Xx_{0\upsilon'}(\mu)$. Again using $k$-reasonableness let $q_1\in\OR(N)^{<\om}$ witness \ref{lem:sound_hull}
with respect to $(N,k,q_0,\mu)$. Let
\[K_\mu=\cHull_{k+1}^N(\mu\un\{q_1\}).\]
So $K_\mu\pins N||\rho_{k+1}^N$. We have the natural near $k$-embeddings
\[ \sigma:M\to K_\mu\text{ and }\tau:K_\mu\to N,\]
with
$\pi=\tau\com\sigma$. Let $\tau(\qhat_0)=q_0$. Now
\[ N\sats\all\varphi\left[\varphi\text{ is }\rSigma_{k+1}\shortimplies\all z\in\mu^{<\om}\left[\varphi(q_0,z,u_k)\shortimplies
K_\mu\sats\varphi(\qhat_0,z,u_k^{K_\mu})\right]\right],\]
and this is an $\rPi_{k+1}$ assertion
$\psi(q_0,\qhat_0,\mu,K_\mu,u_k)$.\footnote{It's not clear that this would hold if
we allowed $\varphi$ to be \emph{generalized} $\rSigma_{k+1}$.}
Let $U=M^{\Xx^+}_\infty$. We have virtual $k$-embeddings
\[ j=i^{\Xx^+}:N\to U\text{ and }\jbar=i^{\Xxbar^+}:M\to\Ubar. \]
So
$j(K_\mu)\pins U$ and
$U\sats\psi(j(z))$ where $z=(q_0,\qhat_0,\mu,K_\mu,u_k)$.

Let $\varrho:\Ubar\to U$ be given by the Shift Lemma. So
$\varrho$ is a virtual weak $k$-embedding.\footnote{It's not clear yet that $\varrho$ is
a virtual near $k$-embedding, because the proof of \cite{fstim} depends on strong
closeness with respect to $(M^\Xxbar_{\prprime},\Ebar),(M^\Xx_{\prprime},E)$,
but here, the
measures of $E,\Ebar$ are ostensibly not definable at all over
$M^\Xx_{\prprime},M^\Xxbar_{\prprime}$.}
Define a (putative) virtual weak $k$-embedding
$\varrho':\Ubar\to j(K_\mu)$
as follows. First set
\[ \varrho'\com\jbar=j\com\sigma.\]
Now
$\Ubar=\Hull_{k}^\Ubar(\rg(\jbar)\un\nu(\Ebar))$
and
\[ \rg(\varrho)=\Hull_k^U(\rg(\varrho\com\jbar)\cup\varrho``\nu(\Ebar)).\]
We set
\[ \rg(\varrho')=\Hull_k^{j(K_\mu)}(\rg(\varrho'\com\jbar)\cup\varrho``\nu(\Ebar)).\]
Note that the two hulls are isomorphic (to $\Ubar$),
and $\varrho'$ is a virtual weak $k$-embedding, because $U\sats\psi(j(z))$
and $\varrho``\nu(\Ebar)\sub j(\mu)$.

Now $U\sats$``$j(K_\mu)$ is reasonable''.
So let $q_2\in j(K_\mu)^{<\om}$ witness this for
$(j(K_\mu),k,r,\iota)$
where $\iota=\lgcd(P||\lh(E))$ and $r=(j(\qhat_0),\nu(E),\cc)$.
Let
\[ M'=\cHull_{k+1}^{j(K_\mu)}(\iota\cup\{q_2\}) \]
and
$\varsigma:M'\to j(K_\mu)$
be the uncollapse.
Note
$M'\pins P|\lh(E)$ (hence is wellfounded)
and
$\rg(\varrho')\sub\rg(\varsigma)$.
Define $\psi:\Ubar\to M'$ by
$\psi=\varsigma^{-1}\com\varrho'$.
Then $\psi$
 is a weak $k$-embedding,
and the rest is like before.
\renewcommand{\qedsymbol}{$\Box$(Claim
\ref{clm:pPiterable})(\ref{thm:cohering})}\end{proof}
\renewcommand{\qedsymbol}{}\end{proof}

\section{Definability of $\es$}\label{sec:stacking}
Given a mouse $M$ and $x\in M$, we can ask whether $\es^M$ is definable
over $\univ{M}$ from $x$. If so,\footnote{One might want to assume
that $\univ{M}\sats\ZFC$, in order to know that $\HOD_x$ is defined as usual.}
\[ \univ{M}\sats\text{``}V=\HOD_x\text{''}. \]
Steel showed that
for $n\leq\om$, $K^{M_n}=M_n$, so $\es^{M_n}$ is $\lpole M_n\rpole$-definable
(from $x=\emptyset$, i.e.~without parameters); see
\cite{cmpw} for $n<\om$. In this context, $K$ is fully iterable
``in intervals between Woodin cardinals''. As we show in \ref{prop:it_error}, this degree of self-iterability
fails for mice with a measurable limit of Woodins.
The author does not know if one can prove a generalization
of Steel's result at this level.

We next prove, in \ref{thm:es_def} and \ref{thm:es_def_2}, that $\es^M$
is nonetheless
$\lpole M\rpole$-definable for various mice $M$
with measurable limits of Woodins,
and also give a new proof for
$M_n$.\footnote{Of course, the fact that $K^{M_n}=M_n$
gives more information than just the fact that $\es^{M_n}$ is definable over $\univ{M_n}$;
and significantly, the ``$K$'' in Steel's result
has a corresponding generic absoluteness which does not seem to be present
in our results here.}
Both theorems require a certain degree of self-iterability;
\ref{thm:es_def_2} requires less, but requires also
that the internal strategy agrees with a fuller external strategy for the mouse.
We then prove in \ref{thm:tamesi} that \ref{thm:es_def_2}
applies in particular to $M_{\nt}|\delta$, where $M_{\nt}$ is the least
non-tame mouse (see \ref{dfn:tame}) and
$\delta$ is its largest (Woodin) cardinal.
This approach to proving
``$V=\HOD$'' (in a mouse) is essentially limited to tame mice
(cf. \ref{prop:nontame}),
as the degree of self-iterability required
typically fails beyond there.

We remark that, using a quite different approach, Woodin has proved the
following weaker version of \ref{thm:es_def}. Let $M$ be a tame mouse satisfying
$\PS$, and having no Woodin cardinals. Suppose $M\sats$``Every proper
segment of me is fully iterable''. Then $\univ{M}\sats$``$V=\HOD$''.
However, Woodin's
proof does not seem to show that $\es^M$ is definable over $\univ{M}$. Also,
Nam Trang and Martin Zeman have recently found a proof of
self-iterability sufficient to apply \ref{thm:es_def}, very
different from what we do here, and which works in a
different context.

\begin{lem}\label{lem:unique_iterable}
Let $N$ be a premouse. Let $\gamma<\OR^N$ be an $N$-cardinal. Suppose
$N\sats$``$\gamma^{++}$ exists''. Let $\R,\Ss\in N$ be reasonable premice such that
\[ \univ{\R}=\univ{\Ss}=(\her_{\gamma^+})^N.\]
Let
$\cc,\dd<(\gamma^+)^N$ be cutpoints of $\R,\Ss$ respectively.
Suppose $R|\mu=\Ss|\mu$ where $\mu=\max(\cc,\dd)$,
and $N\sats$``$\R$ \tu{(}$\Ss$\tu{)}
is above-$\cc$ \tu{(}above-$\dd$\tu{)}, $(\om,\gamma^++1)$-iterable''.
Then
$\R=\Ss$.
\end{lem}
\begin{proof}
We work in $N$. Suppose
$\R\neq\Ss$, and let $\eta$ be least such that
\[ \R|(\eta^+)^N\neq\Ss|(\eta^+)^N.\]
Let $\R'=\R|(\eta^+)^N$ and $\Ss'=\Ss|(\eta^+)^N$. Let $\Sigma$ be
an above-$\cc$, $(\om,\eta^++1)$-strategy for $\R'$, and $\Gamma$ likewise but above-$\dd$ for $S'$. Consider
the (possibly partial) comparison $(\Tt,\Uu)$ of $\R'$ with $\Ss'$, formed using
$\Sigma,\Gamma$. (If the comparison reaches length $\eta^++1$, then
we stop.) We take $(\Tt,\Uu)$ to be padded as usual. 

We claim $(\Tt,\Uu)$ is above
$\eta$.
For $\eta<\lh(E)$ and
$\max(\cc,\dd)<\lh(E)$ for all extenders $E$ used in $(\Tt,\Uu)$.
Suppose $\alpha$ is such that $(\Tt,\Uu)\rest\alpha+1$ is above $\eta$ but
$\crit(E^\Uu_\alpha)<\eta$. Then by \ref{thm:strategy_coh_proj},  $E^\Uu_\alpha\in\es_+(M^\Tt_\alpha)$, so
$M^\Tt_\alpha|\lh(E^\Uu_\alpha)=M^\Uu_\alpha|\lh(E^\Uu_\alpha)$, contradiction.

Now suppose $\Tt$ is non-trivial. Then there is
$Q\pins\R$ such that $\rho_\om^Q=\eta$ and $\Tt$ can be considered a tree on
$Q$. This is because $\Tt$ is above
$\eta$. Moreover, $b^\Tt$ drops. The same things hold for $\Uu$.
Therefore the comparison is successful (and has length $<(\eta^+)^N$), and
(still assuming $\Tt$ is non-trivial), $\Uu$ is trivial, and
$\Ss\pins\fin^\Tt$. But then $E^\Tt_0\in\univ{\R}=\univ{\Ss}\in\fin^\Tt$,
contradiction.
\end{proof}

\begin{lem}\label{lem:es_def}
Let $N$ be a premouse and let $\eta<\OR^N$ be an $N$-cardinal such that
$N\sats$``$\eta^{++}$ exists''. Suppose $N|(\eta^+)^N$ is reasonable, $\cc$ is a cutpoint of $N|(\eta^+)^N$ and
\[ N\sats\text{``}N|\eta^+\text{ is above-}\cc,\ (\om,\eta^++1)\text{-iterable''.}\]
Let $\Hh=(\Hh_{\eta^{+3}})^N$ \tu{(}if $(\eta^{+3})^N=\OR^N$ then set
$\Hh=\lpole N\rpole$\tu{)}. Then
\begin{enumerate}[label=\tu{(}\alph*\tu{)}]
 \item\label{item:N|eta^+_def}
$\{N|(\eta^+)^N\}$ is a $\Sigma_{2}^\Hh(\{N|\cc\})$
singleton, uniformly in $N$, $\eta$ and $\cc$.
\item\label{item:N|eta^+_def_from_N|eta} If $\cc\leq\eta$ then $\{N|(\eta^+)^N\}$ is a
$\Sigma_2^\Hh(\{N|\eta\})$ singleton, uniformly in $N,\eta$.
\end{enumerate}
\end{lem}

\begin{proof}
Part (a): By \ref{lem:unique_iterable}, $\Hh\sats$``$N|\eta^+$ is the unique reasonable premouse
$\R$ such that
(i) $\univ{\R}=\Hh_{\eta^+}$,
(ii) $\N|\cc\pins\R$,
(iii) $\cc$ is a cutpoint of $\R$,
and (iv) $\R$ is above-$\cc$, $(\om,\eta^++1)$-iterable.''

Part (b): If $\cc\leq\eta$ then by \ref{lem:unique_iterable},
$\Hh\sats$``$N|\eta^+$ is the unique reasonable premouse $\R$ such that  $\N|\eta\pins\R$ and for some
$\cc\leq\eta$, items (i), (iii) and (iv) above
hold.''
\end{proof}

We can now deduce the first main theorem of this section.
In the theorem, if $N\not\sats\ZFC$ then we take
$\HOD^N$ to be the union of all transitive sets coded by a set of ordinals
$X$ such that for some $\gamma<\OR^N$, $X$ is definable from ordinal parameters
over $\Hh_\gamma^N$.

\begin{tm}\label{thm:es_def}
 Let $N$ be a premouse satisfying $\PS$, all of whose proper segments are reasonable. Suppose that for each $N$-cardinal
$\eta<\OR^N$ there is $\cc\leq\eta$ such that $\cc$ is a cutpoint of $N$ and
\[ N\sats\text{``}N|\eta^+\text{ is above-}\cc,\ (\om,\eta^++1)\text{-iterable''.}\]
Then
$\es^N$ is definable over $\univ{N}$ and $\univ{N}\sats$``$V=\HOD$''. In fact,
$\{N|(\eta^+)^N\}$ is $\Sigma_2^\her$ where $\her=(\her_{\eta^{+3}})^N$,
uniformly in $\eta$.
\end{tm}

\begin{proof}
Using \ref{lem:es_def}(b), we can define $N|\eta$, for $N$-cardinals $\eta$, by
induction on $\eta$, uniformly in $\eta$. (Given $N|\eta$, we determine
$N|(\eta^+)^N$ by an application of \ref{lem:es_def}(b).) Moreover, the
induction leading to and producing $N|(\eta^+)^N$ is in fact
$\Sigma_2$-definable over $(\her_{\eta^{+3}})^N$ (from no parameters).
\end{proof}

We can now give an alternate proof of Steel's result
that $\univ{M_n}\sats$``$V=\HOD$'' for $n\leq\om$:
Theorem \ref{thm:es_def} applies to $N=M_n$, using
$\cc$ as the supremum of the Woodin cardinals $\delta$ of $N$ such that
$\delta\leq\eta$. (See
\cite{projwell} for the proof of self-iterability in the $n<\om$ case, and
\cite[\S7]{outline} for the $n=\om$
case.) Now recall:

\begin{newdfn}\label{dfn:tame}
 A premouse $M$ is \emph{non-meek} iff there is $E\in\es_+^M$ such that
 $M|\crit(E)\sats$``There is a proper class of Woodins''.
 A premouse $M$ is \emph{non-tame} iff there is $E\in\es_+^M$ and $\delta$
 such that $\crit(E)\leq\delta\leq\nu(E)$ and $M|\lh(E)\sats$``$\delta$ is Woodin''.
 \emph{Meek/tame} means not non-meek/non-tame.
We write $M_{\nt}$ for the least non-tame, sound, $(\om,\om_1+1)$-iterable mouse,
assuming such exists.
\end{newdfn}

We show below that if a
non-tame mouse exists then $M_\nt$ is well-defined; that is, there is a unique
minimal sound non-tame mouse $M$ which projects to $\om$.
We will show that $\univ{M_\nt|\lambda}$
can define $\es^{M_\nt|\lambda}$,
but we need a variant of the results above to achieve this,
as in this case the self-iterability
hypothesis of
\ref{thm:es_def} fails.

\begin{newdfn}
 Let $N$ be a premouse
 and $\Sigma$ a (possibly partial) iteration strategy for $N$.
 We say that $\Sigma$ is \emph{extender-maximal}
iff for every
\[ (k,\Ww,\zeta,P,E,\Ww^+) \]
 potential for $N$, with $\Ww$ via $\Sigma$,
 $\pred^{\Ww^+}(\zeta+1)=0$, and $b^{\Ww^+}$ does not drop in model or degree, then $E\in\es_+(M^\Ww_\zeta)$.
\end{newdfn}

\begin{newlem}\label{lem:unique_ext-max_it}
 Let $R,S$ be passive sound premice
 with $\univ{R}=\univ{S}$, with largest cardinal $\eta$,
 and $R|\eta=S|\eta$. 
 Suppose there are  extender-maximal, above-$\eta$, $(\om,\eta^++1)$-iteration strategies
 for $R,S$.
 Then $R=S$.
\end{newlem}
\begin{proof}
 This is like  \ref{lem:unique_iterable}.
 Let $\Sigma_R,\Sigma_S$ be strategies witnessing the assumption.
 Compare $(R,S)$ using $(\Sigma_R,\Sigma_S)$, producing
 trees $(\Tt,\Uu)$.
 It suffices to see that $(\Tt,\Uu)$ is above $\eta$.
 This follows from extender-maximality.
\end{proof}

Analogously to Theorem \ref{thm:es_def}, we now conclude:

\begin{newtm}\label{thm:es_def_2}
Let $N$ be a premouse satisfying $\PS$,
and $\Sigma$ a $(0,\theta^++1)$-strategy for $N$,
where $\theta=\card(N)$.
For an $N$-cardinal $\eta$,
let
 $\Sigma_\eta$ be the partial strategy induced by $\Sigma$ for trees
 which are in $N$, on $N|(\eta^+)^N$, above $\eta$, and of length $\leq(\eta^+)^N$.
 Suppose that $\Sigma_\eta\in N$ for each such $\eta$.
 Then $N\sats$``$\Sigma_\eta$ is extender-maximal''
 and $\es^N$ is definable over $\univ{N}$, so $\univ{N}\sats$``$V=\HOD$''.
\end{newtm}
\begin{proof}
Because $N$ is iterable, all its proper segments are reasonable.
Note that because $\Sigma_\eta\in N$, $N$ is closed under $\Sigma_\eta$
and $N\sats$``$\Sigma_\eta$ is an above-$\eta$, $(\om,\eta^++1)$-strategy for $N|\eta^+$''.
The extender-maximality of $\Sigma_\eta$ follows from \ref{thm:cohering_simple}.
 
 Now work in $\univ{N}$; we define $\es^N$. We identify $N|\eta$ inductively on $N$-cardinals $\eta$.
 Suppose we have $N|\eta$.
 Then by \ref{lem:unique_ext-max_it} and the properties of $\Sigma_\eta$,
 note that $N|\eta^+$ is the unique premouse $R$
 such that $\univ{R}=\her_{\eta^+}$
 and $R|\eta=N|\eta$ and there is an extender-maximal,
 above-$\eta$, $(\om,\eta^++1)$-iteration strategy for $R$.
\end{proof}

We will next show that whenever
$\lambda$ is a limit cardinal of $M_{\nt}$,
then the hypotheses of \ref{thm:es_def_2} apply to $M_\nt|\lambda$,
with $\Sigma$ the $(0,\om_1+1)$-strategy for $M_\nt|\lambda$
induced by the unique $(\om,\om_1+1)$-strategy for $M_\nt$.
The general method of proof is
standard, but some non-standard details are involved. The method also works for
many mice simpler than $M_{\nt}$.

\begin{lem}\label{lem:Mnt} 

\textup{(}a\textup{)} Let $N$ be a non-tame premouse with no
non-tame proper segment. Let $X\sub\core_0(N)$. Then $\cHull_1^N(X)$ is
non-tame. Therefore if $N$ is $(0,\om_1+1)$-iterable then $M_\nt$ is
well-defined,
\[ M_\nt=\cHull_1^N(\emptyset),\]
$\rho_1(M_\nt)=\om$ and
$p_1(M_\nt)=\emptyset$.

\textup{(}b\textup{)} Let $M$ be a type 3 premouse with
$\nu(M)$ regular in $M$. Let
$\nu(M)=[a,f]^M_{F^M}$ with $a,f\in M$. Let $\Tt$ be a $0$-maximal tree on $M$
such that $i^\Tt$ exists. Let $R=\fin^\Tt$. Then
\[ \nu(R)=[i^\Tt(a),i^\Tt(f)]^R_{F^R}.\]
Therefore if $M$ is non-tame then so is
$R$.
\end{lem}
\begin{proof}
(a) $N$ is type 3 and $\nu(F^N)=[\{\kappa\},f]^N_{F^N}$ where
$\kappa=\crit(F^N)$ and
\[ f:\kappa\to\kappa \]
is such that $f(\alpha)$ is the
least Woodin $\delta$ of $N$ such that $\delta>\alpha$. Now
$\kappa,f\in\Hull_1^N(X)$. Let $H=\cHull_1^N(X)$ and let
$\kappabar,\fbar\in H$ be the collapses of $\kappa,f$. One can
show
\[ \nu(F^H)=[\{\kappabar\},\fbar]^H_{F^H}.\]
Therefore $H$ is
non-tame.

(b) Let
\[ \psi:\Ult(M^\sq,F^M)\to\Ult(R^\sq,F^R)\]
be induced by
$i^\Tt$; i.e.,
\[ \psi([a,f]^{M^\sq}_{F^M})=[i^\Tt(a),i^\Tt(f)]^{R^\sq}_{F^R}.\]
We need to
see that $\psi$ is continuous at $\nu(F^M)$. So let $\gamma<\psi(\nu(F^M))$. Let
$b,g\in R^\sq$ be
such that $\gamma=[b,g]^{R^\sq}_{F^R}$. Because $\Tt$ is only $0$-maximal, there is
$f\in M^\sq$ and some $a$ such that $(b,g)=i^\Tt(f)(a)$. Therefore there are $h\in M^\sq$ and
$\alpha<\nu(F^R)$ such that
\[ \gamma\in i_{F^R}\com i^\Tt(h)``[\alpha]^{<\om}.\]
There is $\beta<\nu(F^M)$ such that $\alpha<i^\Tt(\beta)$. It follows
that
\[ \gamma\in\psi(i_{F^M}(h)``[\beta]^{<\om}).\]
We may assume
\[ i_{F^M}(h)``[\beta]^{<\om}\sub\nu(F^M).\]
But $\nu(F^M)$ is regular in
$\Ult(M,F^M)$, so $i_{F^M}``[\beta]^{<\om}$ is bounded in $\nu(F^M)$. It
follows that $\gamma<\nu(F^R)$, as required.
\end{proof}

\begin{tm}\label{thm:tamesi} Assume $M_\nt$ exists \tu{(}so is $(\om,\om_1+1)$-iterable\tu{)}.
Let $\Sigma$ be the unique $(\om,\om_1+1)$-strategy for $M_\nt$.
Let $\eta<\OR^{M_\nt}$ be an $M_\nt$-cardinal and $\cc$ be the sup of all $\delta\leq\eta$
such that $\delta=0$ or $\delta$ is Woodin in $M_\nt$.
\tu{(}So $\cc$ is a cutpoint of $M_\nt||\OR(M_\nt)$,
but maybe not a strong cutpoint.\tu{)} Suppose $\cc<\eta$.
Let $\lambda>\eta$ be a limit cardinal of $M_\nt$.

If $\eta=(\cc^+)^{M_\nt}$, let $\xi=\cc$; if $\eta>(\cc^+)^{M_\nt}$,
let $\xi=(\cc^+)^{M_\nt}$. Let
$\Gamma$ be the partial strategy
induced by $\Sigma$ for $0$-maximal above-$\xi$ trees on $M_\nt|\eta$
which are in $M_\nt|\lambda$.

Then $\Gamma$ is a class of $M_\nt|\lambda$.

Therefore $M_\nt|\lambda\sats$``$\Gamma$ is an extender-maximal, above-$\xi$, $(0,\lambda)$-strategy
for $M_\nt|\eta$'', and $\es^{M_\nt|\lambda}$ is definable over $\univ{M_\nt|\lambda}$.
\end{tm}

\begin{newrem}\label{rem:it_error}
A draft of this paper
claimed that $M_\nt|\lambda\sats$``$Z|\eta$ is above-$\cc$, $\eta$-iterable''.
But this is false for some $\cc,\eta,\lambda$, as we will show in \ref{prop:it_error}.
The earlier putative proof of this claim ignored the fact that
if $\cc$ is measurable in $M_\nt$, hence a limit of Woodins, and $E$ is an $M_\nt$-total extender with $\crit(E)=\cc$,
then $\Ult(M_\nt|\eta,E)$ has Woodins above $\lh(E)$,
which means that correct Q-structures for trees using
such an $E$ can be non-tame.
\end{newrem}

\begin{proof}\setcounter{clm}{0}
 The last sentence follows from \ref{thm:cohering_simple}
 (which gives the extender-maximality) and \ref{thm:es_def_2}.
So we just need to see that $\Gamma$ is a class of $M_\nt|\lambda$.
We may easily assume that $\eta$ is a successor cardinal of $M_\nt$
and $\lambda\leq\kappa=\crit(F^{M_\nt})$.
We will consider first the case that $\lambda=\kappa$, and then deduce the general case from this one.
Let $N=M_\nt|\eta$ and $Z=M_\nt|\kappa$.

Note that $\Gamma$ is everywhere Q-structure guided (as $\eta$ is a successor cardinal
and $\cc$ is a cutpoint; $\cc$ might be a measurable limit of Woodins,
but if $(\cc^+)^{M_\nt}<\eta$ then the trees we consider are above $\xi=(\cc^+)^{M_\nt}$).

Let $\Ww\in Z$ be a limit length, above-$\xi$ tree on $N$, via $\Gamma$.
Work in 
$Z$. Let
$\CC=\left<N_\alpha\right>_{\alpha\leq\Omega}$ be the maximal
fully backgrounded $L[\es]$-construction above $N_0=M(\Ww)$,
with all background extenders $E^*$ such that $E^*\in\es^Z$
and $\crit(E^*)>\delta(\Ww)$, and where $\Omega$ is least such that either
$\Omega=\kappa$ or the $\delta(\Ww)$-core $Q$ of $N_\Omega$ is a Q-structure for $M(\Ww)$. (Here each
$N_\alpha$ is a premouse with $M(\Ww)\ins N_\alpha$ and $\delta(\Ww)$ a cutpoint, not just an $M(\Ww)$-premouse.)

If $\Omega<\kappa$ then in $Z$ we can use $Q$ to determine $b$ as usual,
completing the proof in this case (that $\lambda=\kappa$).
So assume $\Omega=\kappa$ for a contradiction. Let
$H=F^{M_{\mathrm{nt}}}$, let $R=i_H(N_\kappa)$, let $\delta=\nu(H)$,
let $G=H\rest R\cross[\delta]^{<\om}$,
let
\[ U=\Ult(R|(\kappa^+)^R,G),\]
let $S=U|(\delta^+)^U=R||(\delta^+)^U$, let $G'$ be the
trivial completion of $G$, and $Q=(S,G')$.
Given an iterate $M'$ of $M_\nt$ and iteration map $j:M_\nt\to M'$ with $\delta(\Ww)<\crit(j)$,
let $Q(M')$ denote ``$j(Q)$'',
that is, the structure defined over $M'$ from $M(\Ww)$ in the same
manner as we have just defined $Q$ over $M_\nt$ from $M(\Ww)$.

\begin{clm*} $Q$ is a non-tame premouse
and is  $(0,\om_1+1)$-iterable above $\delta(\Ww)$.
\end{clm*}

Assume the claim. If $M(\Ww)\neq Q$
then $M(\Ww)\pins Q$
and $\OR^{M(\Ww)}$ is Woodin in $Q$. So comparing $Q$ with the tame $\delta(\Ww)$-sound Q-structure
extending $M(\Ww)$ leads to contradiction, using \ref{lem:Mnt}(b).
So the following completes the proof in this case:

\begin{proof}[Proof of Claim]
In this proof we write $\delta(P)$ to denote $\lrgcrd(P)$, for premice
$P$.

Non-tameness and the ISC for
$Q$ is via the usual proof
that Woodinness is absorbed by $L[\es]$-constructions. For
iterability, we define a $0$-maximal strategy for $Q$, lifting trees $\Tt$
on $Q$ to
$0$-maximal trees $\Tt'$ on $M_\nt$ via our strategy for $M_\nt$, using a variant of the lift/resurrect
procedure of \cite[\S12]{fsit}. We assume the reader is familiar with that
procedure, and just explain the differences.
The differences arise because $Q$ is not built by a standard
$L[\es]$-construction: $S$ is
built by $i_F(\CC)$ in $\Ult(M_\nt,F)$, but $F^Q$ is not available to
$\Ult(M_\nt,F)$. 

We will decompose $\Tt$ into the form $\Uu\conc\Vv$, where $(\Uu,\Vv)$ is a stack
of two normal trees,
and decompose $\Tt'$ into the form $\Uu'\conc\left<F\right>\conc\Vv'$, also a stack of two normal
trees $(\Uu'\conc\left<F\right>,\Vv')$.
We lift $\Uu$ to $\Uu'$ via almost the usual procedure. The tree
$\Vv$ begins at
the least $\theta$ such that the appropriate ancestor for $E^\Tt_\theta$ is not
directly available to $\core_0(M^{\Tt'}_\theta)$. We set
$F=F(M^{\Tt'}_\theta)$, and then lift $\Vv$ to $\Vv'$ via the
usual procedure.

The tree $\Uu'$ results from lifting/resurrection of $\Uu$ to
$M_\nt$, except that for each $\alpha<\lh(\Uu)$, if
$[0,\alpha]_\Uu$ does not drop then the lift of $M^\Uu_\alpha$ is
$Q(M^{\Uu'}_\alpha)$ (the ``image'' of $Q$), whereas if $[0,\alpha]_\Uu$ drops then the lift of
$M^\Uu_\alpha$ is $\core_{n_\alpha}(N_{\gamma_\alpha}^{\CC_\alpha})$, where
\[ \CC_\alpha=i^{M^{\Uu'}_\alpha}_{F(M^{\Uu'}_\alpha)}(
i^{\Uu'}_{0,\alpha}(\CC)),\]
$\gamma_\alpha<\lh(\CC_\alpha)$,
and $n_\alpha=\deg^\Tt(\alpha)$. So the lift is in
$\Ult(M^{\Uu'}_\alpha,F(M^{\Uu'}_\alpha))$,
but may not be in
$M^{\Uu'}_\alpha$. Likewise for partial resurrections. Resurrection,
if necessary, is computed in $\Ult(M^{\Uu'}_\alpha,F(M^{\Uu'}_\alpha))$.

We choose $\Uu$ as long as possible such that for all
$\alpha+1<\lh(\Uu)$,
\begin{enumerate}[label=(\roman*)]
 \item if $[0,\alpha]_\Uu$ does not drop then either
$E^\Uu_\alpha\in\core_0(M^\Uu_\alpha)$ or $E^\Uu_\alpha=F(M^\Uu_\alpha)$,\footnote{As $Q$ is type 3, when $[0,\alpha]_\Uu$ does not drop,
we can have $\nu(M^\Uu_\alpha)<\lh(E^\Uu_\alpha)<\OR(M^\Uu_\alpha)$,
in which case (i) fails.}and
 \item letting $E^*=E^{\Uu'}_\alpha$, the ancestor of the lift of
$E^\Uu_\alpha$,
either $E^*\in\core_0(M^{\Uu'}_\alpha)$ or
$E^*=F(M^{\Uu'}_\alpha)$ (the latter occurs just when $[0,\alpha]_\Uu$ does not
drop and $E^\Uu_\alpha=F(M^\Uu_\alpha)$).
\end{enumerate}

Let $\beta=\pred^\Uu(\alpha+1)$ and $k=\deg^\Uu(\alpha+1)$. Let
\[ \sigma:M^{*\Uu}_{\alpha+1}\to\core_k(N^{\CC_\beta}_\eta)\]
be the partial
resurrection map determined by $\crit(E^\Uu_\alpha)$. If
$\core_k(N^{\CC_\alpha}_\eta)\in\core_0(M^{\Uu'}_\alpha)$ or
$\core_k(N^{\CC_\alpha}_\eta)=\core_0(Q(M^{\Uu'}_\alpha))$ then define
$\pi_{\alpha+1}$ as usual. Suppose otherwise. Then
\[ \Ult_0(\Ult_0(M^{\Uu'}_\beta,F(M^{\Uu'}_\beta)),E^{\Uu'}
_\alpha)=\Ult_0(M^{\Uu'}_{\alpha+1},F(M^{\Uu'}_{\alpha+1}))\] and the ultrapower
maps commute. Let
\[ j:\Ult_0(M^{\Uu'}_\beta,F(M^{\Uu'}_\beta))\to\Ult_0(M^{\Uu'}_{\alpha+1},F(M^{
\Uu' }_{\alpha+1}))\]
be the resulting ultrapower map (via $E^{\Uu'}_\alpha$). Then
we define $\pi_{\alpha+1}$ in the usual way, except that $j$ replaces
$i^{\Uu'}_{\beta,\alpha+1}$.

Note that by our choice of $\Uu$, the first drop in
model along any $\Uu$-branch is to some
$M^{*\Uu}_{\alpha+1}\pins\core_0(M^\Uu_\beta)$, where
$\beta=\pred^\Uu(\alpha+1)$. So $M^{*\Uu}_{\alpha+1}\in\dom(\pi_\beta)$. (But
the resurrection of $\pi_\beta(M^{*\Uu}_{\alpha+1})$
may not be in $M^{\Uu'}_\beta$.)

If $\Vv=\emptyset$ then
$F,\Vv'=\emptyset$.

Suppose $\Vv\neq\emptyset$.
Then $\lh(\Uu)$ is a successor $\theta+1$, and either (i) or (ii) above fails
for $\alpha=\theta$. We set
$E^{\Tt'}_{\theta}=F=F(M^{\Uu'}_\theta)$.
Let $\lambda=i_F(\crit(F))$. We will define $\Tt'=\Uu'\conc\left<F\right>\conc\Vv'$, indexed with
\[ (0,1,\ldots,\theta,\theta*,\theta+1,\ldots). \]
Let
$M=M^{\Uu'}_\theta$. Then $M^{\Tt'}_{\theta*}|\lambda=\Ult(M,F)|\lambda$
and
$i^{\Tt'}_{0,\theta*}(\CC)=C_\theta$. Suppose (i) fails
for $\alpha=\theta$. Let
\[ \psi:\Ult(M^\Uu_\theta,F(M^\Uu_\theta))\to Q(\Ult(M,F))\]
be induced by
$\pi_\theta$, through the Shift Lemma, and let
$E^*=\psi(E^\Uu_\theta)$. By \ref{lem:Mnt}(b),
$\psi(\delta(M^\Uu_\theta))=\delta(M^{\Uu'}_\theta)$, so
$\delta(M^{\Uu'}_\theta)<\lh(E^*)$. If instead (i) holds, use
$\psi=\pi_\theta$ to lift $E^\Uu_\theta$ to an extender $E^*$. In either case
let $E^{\Tt'}_{\theta*}$ be the ancestor of $E^*$, according to
$\CC_\theta$. Note that $\lh(F)$ is a cutpoint of $M^{\Uu'}_{\theta^*}$, and
$\lh(F)<\crit(E^{\Tt'}_{\theta*})$.

Let $\sigma$ be the map resurrecting $E^*$ to $E^{\Tt'}_{\theta*}$. Let
$\gamma\in\OR$ be least such that
$\sigma(\psi(\gamma))\geq\delta(M^{\Uu'}_\theta)$. Then $\gamma$ is a limit
cardinal cutpoint of $M^\Tt_\theta|\lh(E^\Tt_\theta)$. For the limit cardinality
is clear; and if $G\in\es_+(M^\Tt_\theta|\lh(E^\Tt_\theta))$ overlaps
$\gamma$ then the ancestor $G'$ of $\sigma(\psi(G))$ is such that
$\crit(G')<\delta(M^{\Uu'}_\theta)<\lh(G')$, contradicting the tameness of
$\Ult(M,F)|\lambda$. Also, $\gamma<\lh(E^\Tt_\theta)$.

Let $\Vv$ be the remainder of $\Tt$, so $\Tt=\Uu\conc\Vv$. Then by the
previous paragraph, $\Vv$ is on $M^\Uu_\theta$ and is above $\gamma$. We
lift $\Vv$ to a normal tree $\Vv'$ on $M^{\Tt'}_{\theta*}$, above
$\delta(M^{\Tt'}_\theta)$. For this we use the standard process; we need not
lift to any $Q(M^{\Vv'}_\alpha)$. This is because $[0,\beta]_\Tt$ drops in model
for every $\beta$ such that $\theta<_\Tt\beta$.
For suppose  $[0,\theta]_\Tt$ does not drop
and let $\alpha+1<\lh(\Tt)$ with $\pred^\Tt(\alpha+1)=\theta$; then
$\lh(E^\Tt_\theta)$ has cardinality
$\leq\gamma$ in $M^\Tt_\theta$, and therefore $\alpha+1\in D^\Tt$.
If (i) fails this is immediate, and  (i) holds it is because
$\gamma<\lh(E^\Tt_\theta)<M^\Tt_\theta)$
and $\psi(\gamma)<\delta(M^{\Tt'}_\theta)$,
but $\sigma(\psi(\gamma))\geq\delta(M^{\Tt'}_\theta)$,
and $\sigma$ results from resurrection of projecting structures (so if $(\gamma^+)^{M^\Tt_\theta}<\lh(E^\Tt_\theta)$,
then $\sigma(\psi(\gamma))=\psi(\gamma)$, contradiction).
We leave the remaining details to the reader.
\renewcommand{\qedsymbol}{$\Box$(Claim)}\end{proof}

Now consider the case that $\lambda\leq\kappa$ is a limit cardinal of $M_\nt$.
 As before,
we have $\Ww\in Z=M_\nt|\lambda$ and want to compute $\Gamma(\Ww)$.
We simply look for some $R$
with $N\pins R\pins M_\nt|\lambda$
and $\Ww\in R\sats\ZFC$ and $R\sats$``My $L[\es]$-construction
above $M(\Ww)$ reaches a Q-structure $Q$ for $M(\Ww)$
as above''. We  use this $Q$ to compute $b$. Since $M_\nt|\kappa$ 
has these properties and $\lambda\leq\kappa$ is a limit cardinal of $M_\nt$,
by condensation there is also some such $R\pins M|\lambda$, and any
such $R$ computes the correct Q-structure, so we are done. This completes the proof of the theorem.
\end{proof}

\begin{rem}
The methods above can be adapted to many other canonical
tame mice,
for example,
the least proper class mouse with an inaccessible limit of Woodin cardinals, etc.

It is known that under sufficient large cardinal assumptions
and for sufficiently complex reals $x$, $\HOD^{L[x]}\inter\RR=M_1\inter\RR$.
So if $x$ is a sound mouse projecting to $\om$ which extends $M_1^\#$,
then $L[x]$ is a mouse satisfying ``$\RR\not\sub\HOD$''.

Steel asked: (1)
 Suppose $M$ is a mouse satisfying $\ZFC$. Does $\univ{M}$ satisfy
``$V=\HOD_X$ for some $X\sub\om_1^M$''? Naturally we can also consider the
variant of (1) in which we require $X$ to be a real.

Some related questions are: (2) To what extent can the methods
used to prove \ref{thm:tamesi} be adapted to more arbitrary (non-``canonical'') tame mice?
(3) Do or can non-tame mice known enough of their own iteration strategies
in order that one can answer (1) using the methods of this paper?

It turns out that non-tame mice very typically do \emph{not} know enough of
their own iteration strategy to suffice for (3).
This is demonstrated in \ref{prop:nontame} below, which is a simple variant of an observation
possibly due to Steel; see \cite[1.1]{sile}.
However, the author has since answered (1) positively; the proof uses methods quite different to those of this paper.
Moreover, if  $M$ is tame then  we may take $X$ to be a real, but it is not known to the author whether
the same holds for non-tame $M$.
These results are to appear.
The author has also shown
that various canonical non-tame mice, for example $M_{\mathrm{adr}}$, satisfy ``$V=\HOD$''.
\end{rem}

\begin{prop}\label{prop:nontame}
Let $N$ be a countable premouse satisfying $\ZFC$ and
$\Sigma$ an $(\om,\om_1+1)$-strategy
for $N$. Let $P\pins N$ be active and suppose $\tau,\delta,\eta$ are such that
\[ \crit(F^P)<\tau<\delta<\OR^P\leq\eta<(\tau^+)^N, \]
$\tau$ is a cardinal of $N$,
$P\sats$``$\delta$ is Woodin'', and
 $\eta$ is a cutpoint of $N$.
Let $\Sigma'$ be the strategy induced by $\Sigma$, for $0$-maximal trees
$\Tt$ on $P$ such that $\Tt\in N$, $\Tt$ is above $\tau$, and of
length $\leq(\tau^+)^N$.\footnote{That is, $\Sigma'(\Tt)$ is defined for such $\Tt$ of length $\leq(\tau^+)^N$.}
Then
$\Sigma'\notin N$.\end{prop}

\begin{proof}
For simplicity we assume that
$\delta<\nu_E$. The argument in this case can be adapted to the case that
$\delta=\nu_E$, using part of the argument for \ref{lem:Mnt}.

Suppose $\Sigma'\in N$. Let $\kappa$ be the least measurable of $P$ such that
$\kappa>\tau$. Let $\BB$ be the $\delta$-generator extender
algebra of $P$ at $\delta$,
using extenders $E$ with $\crit(E)>\kappa$. Working in $N$, let
$\Tt$ on $P$ be given by first iterating $\kappa$ out past $\eta$
and then making $N|\eta$
generic for $i^\Tt(\BB)$ over $\fin^\Tt$. Because $\Sigma'\in N$, this succeeds
with $\Tt\in N$ of length $<(\tau^+)^N$. Let $E=F(\fin^\Tt)$, so $\crit(E)=\crit(F^P)<\tau$,
$E$ is $N$-total
and $\Ult(N,E)$
is wellfounded, because $\Sigma'$ is induced by $\Sigma$. Moreover, $N|\eta$ is generic over
$\Ult(N,E)$. Let $Q\ins N$ be least such that
$\rho_\om^Q=\tau$ and $i^\Tt(\delta)\leq\OR^{Q}$. Then
\[ Q\notin\Ult(N,E)[N|\eta] \]
since $\delta$ is a cardinal in $\Ult(N,E)[N|\eta]$.
But $Q\in\Ult(N,E)[N|\eta]$ as in the proof of \cite[1.1]{sile},\footnote{Sketch: In the generic extension after collapsing $\eta$,
we can search for a $Q$ of the right form, extending $N|\eta$, and embedded
into $i_E(Q)$. But then by uniqueness we get $Q$ in the ground model.}
contradiction.\renewcommand{\qedsymbol}{$\Box$(\ref{prop:nontame})}\end{proof}

We finally adapt the example above to show that
the above-$(\cc^+)^{M_\nt}$ iterability of $M_\nt|\eta$ in $M_\nt$ 
established in \ref{thm:tamesi}
cannot in general be improved to above-$\cc$ iterability,
and that the same holds for typical non-meek mice. The example also indicates that if non-meek
mice satisfy ``$V=K$'' (whatever this means)
then there must be  new difficulties involved in the proof,
as the core model $K$ involved would be less iterable than that in the case of $M_n$:

\begin{newprop}\label{prop:it_error} Let $N$ be a countable premouse satisfying $\ZFC$
and $\Sigma$ a $(0,\om_1+1)$-strategy for $N$.
Let $\cc$ be a limit of Woodin cardinals of $N$ and
suppose there is $E\in\es^N$ such that
\[ \cc=\crit(E)<(\cc^+)^N<\lh(E).\]
Suppose there is a cutpoint $\dd$ of $N$ such that $\lh(E)\leq\dd<\OR^N$.
Then
\[ N\sats\text{``}N|\lh(E)\text{ is not above-}\cc\text{, }(0,\dd^++1)\text{-iterable''.}\]
\end{newprop}
\begin{proof}
 Suppose otherwise. Let $E\in\es^N$ be least witnessing the assumption.
 Let $\Ubar=\Ult(N|(\cc^+)^N,E)$.
 Then $i^N_E(\cc)$ is a limit of Woodins of $U$.
 Let $\delta$ be the least Woodin of $\Ubar$ such that $\delta>\lh(E)=(\cc^{++})^\Ubar$.
Form an above-$\lh(E)$ tree $\Uu$ on $\Ubar$,
first linearly iterating past $\dd$, and then iterating to make $N|\dd$ generic for the above-$\lh(E)$ extender algebra of $M^\Tt_\infty$
at $\delta^*=i^\Uu(\delta)$.
By the iterability assumption, this succeeds,
with $\lh(\Uu)<(\dd^+)^N$. Let $\Tt=E\conc\Uu$.
Now 
$\Tt$ can be considered a tree $\Tt'$ on $N$, with last model $W$
($\Tt'$ has wellfounded models; otherwise pull a counterexample
to wellfoundedness down below $N|(\cc^+)^N$ for a contradiction).
Let $P\pins N$ be least such that $\delta^*\leq\OR^P$ and $\rho_\om^P<\delta^*$.
Let $P^*=i^{\Tt'}(P)\in W$.

Now $\delta^*$ is regular in $W[N|\dd]$, so $P\notin W[N|\dd]$.
By absoluteness as in the proof of \ref{prop:nontame},
it follows that in $W[N|\dd][G]$, where $G\sub\Coll(\om,\delta^*)$
is $W[N|\dd]$-generic, there are $P_1,P_2,\pi_1,\pi_2$
such that $P_i$ is a sound premouse, $P_1\nins P_2\nins P_1$, $N|\dd\ins P_i$, $\dd$ is a cutpoint of $P_i$,
$\rho_\om^{P_i}<\delta^*$, and
$\pi_i:P_i\to P^*$
is elementary.
The existence of such objects embedded into $P^*$ is therefore forced over $W$ by $\Coll(\om,\delta^*)$.
This pulls back elementarily to $N,P$ and 
some Woodin $\delta'<\cc$. But this contradicts the iterability of
$N$ in $V$.
\end{proof}

\section{Mixed bicephali}\label{sec:bicephali}
In this section we generalize Mitchell and Steel's
``Uniqueness of the next extender''
result from \cite[\S9]{fsit}:

\begin{newtm}\label{thm:uniquemaxLEcon}
Let $R$ be an $(\om_1+1)$-iterable countable transitive model of $\ZFC$.
Work in $R$. Let $\CC=\left<N_\alpha\right>_{\alpha\leq\lambda}$ be a fully backgrounded
$L[\es]$-construction. Suppose there
are fully backgrounded extenders $E,F\neq\emptyset$ such that $(N_\lambda,E)$ and $(N_\lambda,F)$ are
premice. Then $E=F$.\footnote{We leave the precise interpretation of ``fully backgrounded''
to
the reader; it should be enough to guarantee the $(0,\om_1+1)$-iterability of
$B$ in the proof.}\end{newtm}

The assumption that $R\sats\ZFC$ is made for simplicity; one can as usual get by with less than this.
In the case that $E,F$ are either both type 2, or neither is type 2,
this already follows from the bicephalus arguments in \cite{fsit}.
The remaining case, that say $E$ is type 2 and $F$ is not,
was not considered in \cite{fsit} (this case was not necessary for the results there).
However, the uniqueness question in this case is of course also natural, and we verify it here.

In what follows we use $\delta(P)=\delta^P$ to denote $\lrgcrd(P)$.

\begin{dfn}
A \emph{mixed bicephalus} is a structure $(N,E,F)$, where $(N,E)$ is a
type 2 premouse and $(N,F)$ is a type 1 or type 3 premouse.

A \emph{mis-bicephalus} is
a structure $(N,E,F)$, where $(N,E)$ is a
type 2 premouse and $(N,F)$ is a segmented-premouse.\footnote{``mis-''
abbreviates ``mixed segmented-''. \emph{Segmented-premouse} is defined in
\ref{dfn:segmented}.}

Given a mixed or mis-bicephalus $B=(N,E,F)$,
$F^B_0=E$ and $F^B_1=F$.
\end{dfn}

Remarks similar to those in \ref{rem:segmented}, regarding fine structure and
ultrapowers, also apply to mis-bicephali. Ultrapowers of such are
always constructed at the unsquashed level. This works like for type 2
bicephali in \cite[\S9]{fsit}.
The reader should now recall the two examples of proper seg-pms in \ref{rem:segmented};
in the comparison argument to follow we need to deal with these and more general examples.

In a normal iteration tree $\Tt$, the exchange ordinal $\nu^\Tt_\alpha$ associated with extender $E^\Tt_\alpha$
is usually the strict sup of generators $\nu(E^\Tt_\alpha)$.
For our trees on mixed bicephali, it is convenient to tweak this in special cases;
the tweak makes the proof of iterability of bicephali slightly smoother.

\begin{dfn}\label{dfn:iterationgame}Let $N$ be a mixed bicephalus. The
\emph{maximal iteration game on $N$} is defined as usual, except that for
$\alpha+1<\lh(\Tt)$: Set
$\nu^\Tt_\alpha=\nu(E^\Tt_\alpha)$, unless $[0,\alpha]_\Tt$ does not drop
and $E^\Tt_\alpha=F_1(M^\Tt_\alpha)$, in which case set $\nu^\Tt_\alpha=\delta(M^\Tt_\alpha)$.\footnote{Steve Jackson
noticed that we could have instead set $\nu^\Tt_\alpha=\nu(E^\Tt_\alpha)$ in all
cases. However, doing so would slightly complicate the standard proof of iterability of
the bicephalus (when the bicephalus is constructed by fully backgrounded $L[\es]$-construction in an iterable background universe).}
Then $\pred^\Tt(\alpha+1)$ is the least
$\beta$ such that $\crit(E^\Tt_\alpha)<\nu^\Tt_\beta$.
We say $N$ is \emph{$\alpha$-iterable} iff player $\Two$ has a
winning strategy in the maximal game on $N$ of length $\alpha$.
\end{dfn}

\begin{lem}\label{lem:bicephalus} There is no $(\om_1+1)$-iterable mixed
bicephalus.\end{lem}

From now on we write \emph{bicephalus} for \emph{mis-bicephalus}.

\begin{proof}
     Suppose $P$ is otherwise. We may assume $P$ is countable. As
in 
\cite[\S9]{fsit}, we compare $P$ with itself, producing padded $0$-maximal trees
$\Tt,\Uu$.
For $\alpha+1<\lh(\Tt)$, say $E^\Tt_\alpha$ is
\emph{$\Tt$-special} iff
$[0,\alpha]_\Tt$ does not drop and $E^\Tt_\alpha=F_1(M^\Tt_\alpha)$. We may and
do require that if $E^\Tt_\alpha$ is $\Tt$-special then
$E^\Uu_\alpha\neq\emptyset$ (we have
\[ F_0(M^\Tt_\alpha)\neq F_1(M^\Uu_\alpha),\]
because $F_0^P\neq F_1^P$). Likewise for $\Uu$.

     We claim that the comparison fails. This is as in \cite[\S9]{fsit},
but we discuss a detail not discussed there. Suppose the
comparison succeeds, with
\[ Q=\fin^\Tt\ins R=\fin^\Uu.\]
Then $b^\Tt$ drops (otherwise $F_0^Q\neq F_1^Q$ and the comparison can continue). In
\cite{fsit}, it is argued that therefore
$Q$ is unsound, so $Q=R$, so $b^\Uu$ drops, and standard fine structure leads to
contradiction. The appeal to fine structure assumes
that enough extenders used in $\Tt,\Uu$ are close to the models to which
they apply. But \cite[6.1.5]{fsit} was not proven for
bicephali. We deal with this by proving:

\begin{enumerate}[label=($*$\arabic*)]
\item\label{*1} If $[0,\delta+1]_{\Tt}$ drops then $E^{\Tt}_\delta$ is close
to
$(M^*_{\delta+1})^{\Tt}$; likewise for $\Uu$.\footnote{Establishing
``closeness'' in general would require consideration of $\Sigma_1^{(M,F_1,F_2)}$
in the language using both $F_1,F_2$, which we prefer to avoid. We could
alternatively make do with semi-closeness,
using \ref{cor:k+1solid},
but we also need the consequences of the proof of $(*1)$ later.}
\end{enumerate}

     Suppose that $E^\Tt_\alpha$ is
$\Tt$-special. Let
$\kappa=\crit(E^\Tt_\alpha)$. If $\gamma\in[0,\alpha]_\Tt$ is least
such that $\gamma=\alpha$ or
$\crit(i^\Tt_{\gamma\alpha})>\crit(F_1(M^\Tt_\gamma))$, then
$\kappa=\crit(F_1(M^\Tt_\gamma))$. So for $\delta<\gamma$, we have
$\nu^\Tt_\delta<\kappa$, and if $\gamma<\alpha$, then
$\kappa<\crit(i^\Tt_{\gamma\alpha})<\nu^\Tt_\gamma$. Therefore
$\pred^\Tt(\alpha+1)=\gamma$ and
$\crit(E^\Tt_\alpha)=\crit(F_1(M^\Tt_\gamma))$.

     Applying this to all $\Tt$-special extenders, we get:
\begin{enumerate}[resume*]
 \item\label{*2} If $E^\Tt_\alpha$ is $\Tt$-special, $\gamma$ is as above
and $\delta+1\in(\gamma,\alpha]_\Tt$, then $E^\Tt_\delta$ is not $\Tt$-special.
\end{enumerate}

By these and similar considerations regarding those
$\alpha$ such that $[0,\alpha]_\Tt$ does not drop and $E^\Tt_\alpha=F_0(M^\Tt_\alpha)$,
we have:
\[ \text{If }[0,\delta+1]_\Tt\text{ drops then either }\lh(E^\Tt_\delta)<\OR(M^\Tt_\delta)\text{ or }[0,\delta]_\Tt\text{ drops.}\]
By this and the argument of \cite[6.1.5]{fsit}, \ref{*1} follows.

     We examine the generators of $\Tt$-special extenders. Given a
short extender $E$ weakly amenable to a premouse $M$,
let $\varrho(E)$ denote the least
$\varrho\geq(\kappa_E^+)^M$ such
that $E\rest\varrho\notin\Ult_0(M,E)$ and $E\rest\varrho$ is not type Z. (Note that $\varrho(E)$ is independent of the choice of $M$.)

     If $M$ is a premouse and $E\in\es_+^M$, then
$\varrho(E)=\nu(E)$. Suppose $[0,\alpha]_\Tt$ is
non-dropping. Then
$\varrho(F_1(M^\Tt_\alpha)) = \sup i^\Tt_{0\alpha}``\delta(P)$.
Moreover,
\[ \varrho(F_1(M^\Tt_\alpha)) < \nu(F_1(M^\Tt_\alpha)) \]
    iff there is $\beta<_\Tt\alpha$ such that
    \begin{equation}\label{eqn:bicephalus_tp3_projectum}      
     i^\Tt_{0\alpha}``\delta(P)=
     i^\Tt_{0\beta}``\delta(P)\sub\crit(i^\Tt_{\beta\alpha}).
    \end{equation}
    In this case, we have that
    \begin{equation}\label{eqn:bicephalus_tp3_sup_gens}
    \nu(F_1(M^\Tt_\alpha))=\sup_{\gamma<\alpha}(\nu'(E^\Tt_\gamma)),
    \end{equation}
    where $\nu'(E)=\nu(E)$ if $E$ is type 2 or type 3, and $\nu'(E)=\crit(E)+1$
if $E$ is type 1. These statements can be proven inductively along the branch
$[0,\alpha]_\Tt$, using \ref{lem:Dsp3} and \ref{*2} for the successor case, and
$\rPi_1$ elementarity in the limit case.

    It follows that if $[0,\alpha]_\Tt$ is non-dropping, then
$\nu(F_1(M^\Tt_\alpha))\leq\delta(M^\Tt_\alpha)$, and therefore $M^\Tt_\alpha$
is a bicephalus, and player $\Two$ does not win through
\ref{dfn:iterationgame}(a).

     So the comparison fails, reaching length $\om_1+1$.
We take a hull and get $\pi:H\to V_\theta$ elementary with $H$
countable. Let $\kappa=\crit(\pi)$. Then $M^\Tt_\kappa\in H$ and $\pi\rest
M^\Tt_\kappa=i^\Tt_{\kappa\om_1}$, and likewise with $\Uu$. Also
\[
M^\Tt_\kappa|(\kappa^+)^{M^\Tt_\kappa}=M^\Uu_\kappa|(\kappa^+)^{M^\Uu_\kappa},\]
and the $(\kappa,\om_1)$-extenders derived from $i^\Tt_{\kappa,\om_1}$ and $i^\Uu_{\kappa,\om_1}$ are
identical. Denote this extender $G$. We have
$M^\Tt_{\om_1}|\om_1=M^\Uu_{\om_1}|\om_1$, so $M^\Uu_{\om_1}$ agrees
with $M^\Uu_{\om_1}$ about $V_{\om_1}$.

    Let $\alpha+1=\min(\kappa,\om_1]_\Tt$ and
$\beta+1=\min(\kappa,\om_1]_\Uu$. So $E^\Tt_\alpha$, $G$ and $E^\Uu_\beta$
measure the same sets and are compatible through
$\min(\nu^\Tt_\alpha,\nu^\Uu_\beta)$. Now
\[ \varrho(G)=\varrho(E^\Tt_\alpha)\leq\nu(E^\Tt_\alpha)\leq\nu^\Tt_\alpha.\]
Likewise for $\Uu,\beta$, so
$\varrho(E^\Tt_\alpha)=\varrho(G)=\varrho(E^\Uu_\beta)$.

\setcounter{case}{0}
\begin{case}\label{case:varrho=nu}
$\nu(E^\Tt_\alpha)=\varrho(E^\Tt_\alpha)=\varrho(G)=\varrho(E^\Uu_\beta)=
\nu(E^\Uu_\beta)$.

    If $\alpha=\beta$, then also $\lh(E^\Tt_\alpha)=\lh(E^\Uu_\beta)$, so the
compatibility implies $E^\Tt_\alpha=E^\Uu_\beta=E^\Uu_\alpha$, contradiction. So
assume $\alpha<\beta$. Then
$\nu(E^\Uu_\beta)<\lh(E^\Tt_\alpha)<\lh(E^\Uu_\beta)$, and $\lh(E^\Tt_\alpha)$
is a cardinal in $M^\Uu_\beta|\lh(E^\Uu_\beta)$, so $E^\Uu_\beta$ is indexed too
late for a premouse, so is $\Uu$-special. So $M^\Uu_\beta$
has form $(N,F_0,E^\Uu_\beta)$, and by our rules on
using $\Uu$-special extenders, $E^\Tt_\beta\neq\emptyset$. So
$M^\Tt_\beta|\OR^N$ is active.

Now $\OR^N$ is a cardinal of $\Ult_0(M^\Uu_\beta,E^\Uu_\beta)$. But
$E^\Tt_\alpha\rest\nu(E^\Tt_\alpha)=E^\Uu_\beta\rest\nu(E^\Uu_\beta)$ and
$M^\Tt_\alpha|\kappa=M^\Uu_\beta|\kappa$. So $\Ult_0(M^\Tt_\alpha,E^\Tt_\alpha)$
and $M^\Tt_{\alpha+1}$ agree with $\Ult_0(M^\Uu_\beta,E^\Uu_\beta)$ past
$\OR^N+1$, and agree with $M^\Uu_\beta$ strictly below
$\OR^N$. But then for all $\gamma\in(\alpha,\beta)$,
$E^\Tt_\gamma=\emptyset$. So $M^\Tt_\beta=M^\Tt_{\alpha+1}$, so
$M^\Tt_\beta|\OR^N$ is passive, contradiction.
\end{case}

\begin{case}
\label{case:varrho<nu}
$\varrho(E^\Tt_\alpha)<\nu(E^\Tt_\alpha)$.

By the case hypothesis, $E^\Tt_\alpha$ is $\Tt$-special. We digress
for a moment.

\begin{dfn}Let $E$ be a short extender, weakly amenable over $N$. For
$\varrho\leq\nu(E)$,
we say $E\rest\varrho$ is a \emph{natural segment of $E$} iff either
(i) $\varrho=0$, or (ii)
$\varrho\geq(\kappa_E^+)^N$ and $\nu(E\rest\varrho)=\varrho$ and
\[ \pow(\varrho)\inter\Ult_0(N,E)=\pow(\varrho)\inter\Ult_0(N,
E\rest\varrho).\]
Let the \emph{segmentation} of $E$ be
\[ S=\{\varrho\mid E\rest\varrho\text{ is a natural segment of }E\}.\]
For
$\sigma,\varrho\in S$
with
$\sigma\leq\varrho$ let
\[ j_{\sigma\varrho}:\Ult_0(N,E\rest\sigma)\to\Ult_0(N,E\rest\varrho) \]
be the
factor embedding, and if $\sigma<\varrho$, let
$\kappa_\sigma=\crit(j_{\sigma\varrho})$. If
$\varrho=\min(S\cut(\sigma+1))$, let $G_\sigma$ be the
$(\kappa_\sigma,\varrho)$ extender derived from $j_{\sigma\varrho}$. We call
the extenders $G_\sigma$ the \emph{natural factors} of $E$.
\end{dfn}

    If $N$ is a premouse and $E=F^N$, then clearly $S=\{0,\nu(E)\}$. If
$E=E^\Tt_{\alpha'}$ is $\Tt$-special, then
$\varrho(E)=\min(S\cut\{0\})$. Let
$\beta\in[0,\alpha']_\Tt$
be least as in line (\ref{eqn:bicephalus_tp3_projectum}) (with $\alpha'$
replacing $\alpha$). Then $S\cut\{0\}$
consists of those ordinals $\varrho$ of either the form $\varrho=\varrho(E)$,
or the form $\varrho=\nu(E^\Tt_\delta)$ for some $\delta$ such that
$\delta+1\in(\beta,\alpha']_\Tt$. Moreover,
\[ G_0=E^\Tt_{\alpha'}\rest\varrho(E^\Tt_{\alpha'})\]
and if $\sigma\in S\cut\{0\}$ and there is $\varrho\in S$ such that
$\sigma<\varrho$ then
\[ G_\sigma=E^\Tt_\delta\rest\nu(E^\Tt_\delta),\] where $\delta$ is least
such
that $\delta+1\leq_\Tt\alpha'$ and $\crit(E^\Tt_\delta)\geq\sigma$. Further,
$E^\Tt_\delta$ is not $\Tt$-special. These remarks follow
from the calculations in \cite[4.3 and following Remark]{cmip}
using fact \ref{*2} and \ref{lem:extass}, much as in the
proof of \ref{lem:Ddecomp}.

We now continue with Case \ref{case:varrho<nu}. Let
$S=\left<\varrho_\xi\right>_{\xi<\om_1}$ be the segmentation of $G$ and let
$\left<G_\sigma\right>_{\sigma\in S}$ be the sequence of
natural factors of $G$. This sequence is just
the concatenation of the sequences of natural factors of the extenders
$E^\Tt_\delta$ for $\delta+1\in(\kappa,\om_1]_\Tt$. Likewise for the natural
factors of the extenders $E^\Uu_\delta$ for $\delta+1\in(\kappa,\om_1]_\Uu$.

We have $\varrho_0=0$ and $\varrho_1=\varrho(E^\Tt_\alpha)<\nu(E^\Tt_\alpha)$,
and
$E^\Tt_\alpha$ is $\Tt$-special. So
\[ G_{\varrho_0}=E^\Tt_\alpha\rest\varrho_1=E^\Uu_\beta\rest\varrho_1\]
and there is $\gamma_0$ such that $E^\Tt_{\gamma_0}$ is non-$\Tt$-special and
$\varrho_2=\nu(E^\Tt_{\gamma_0})$ and
\[ G_{\varrho_1}=E^\Tt_{\gamma_0}\rest\varrho_2.\]
Now if $\varrho_1<\nu(E^\Uu_\beta)$ then there is $\gamma'$ such that
$E^\Uu_{\gamma'}$ is non-$\Uu$-special and $\varrho_2=\nu(E^\Uu_{\gamma'})$ and
\[ G_{\varrho_1}=E^\Uu_{\gamma'}\rest\varrho_2, \]
but then $(M^\Tt_{\gamma_0}||\lh(E^\Tt_{\gamma_0}),E^\Tt_{\gamma_0})$ is a
premouse, and
likewise for $E^\Uu_{\gamma'}$, so $E^\Tt_{\gamma_0}=E^\Uu_{\gamma'}$,
contradiction. So $\varrho_1=\nu(E^\Uu_\beta)$.

Now let $\alpha_0=\alpha$ and $\beta_0=\beta$ and
\[ \beta_1+1=\min((\beta_0+1,\om_1]_\Uu).\]
So $\varrho_2=\varrho(E^\Uu_{\beta_1})$ and
\[ G_{\varrho_1}=E^\Uu_{\beta_1}\rest\varrho_2.\]
Now like before, $\varrho_2<\nu(E^\Uu_{\beta_1})$ (here if
$\lh(E^\Tt_{\gamma_0})<\lh(E^\Uu_{\beta_1})$, use the argument
from Case \ref{case:varrho=nu}). So there is a non-$\Uu$-special
$E^\Uu_{\delta_1}$ such that $\varrho_3=\nu(E^\Uu_{\delta_1})$ and
\[ G_{\varrho_2}=E^\Uu_{\delta_1}\rest\varrho_3. \] 
But then as before,
\[ G_{\varrho_2}=E^\Tt_{\alpha_2}\rest\varrho(E^\Tt_{\alpha_2})\] for
some $\Tt$-special $E^\Tt_{\alpha_2}$ with
$\varrho(E^\Tt_{\alpha_2})<\nu(E^\Tt_{\alpha_2})$. So by \ref{*2},
$E^\Tt_{\alpha_0}$
has only two natural factors, $G_{\varrho_0}$ and $G_{\varrho_1}$.

This
pattern continues through $\om$ stages, producing a sequence of overlapping
$\Tt$- and $\Uu$-special extenders $E^\Tt_{\alpha_{2i}}$ and
$E^\Uu_{\beta_{2i+1}}$, for $i<\om$, each having exactly two natural
factors. The natural factors of $E^\Tt_{\alpha_{2i}}$
are $G_{\varrho_{2i}}$ and $G_{\varrho_{2i+1}}$; the natural factors of
$E^\Uu_{\beta_{2i+1}}$ are $G_{\varrho_{2i+1}}$ and $G_{\varrho_{2i+2}}$.
   
For $i<\om$ let $\kappa_i=\crit(G_{\varrho_i})$ and
$F_{2i}=E^\Tt_{\alpha_{2i}}$. So $\kappa_{2i}=\crit(F_{2i})$. The next claim
shows that $M^\Tt_{\om_1}$ is illfounded, a contradiction.

\begin{clm*} For each
$i<\om$ we have
$i_{F_{2i}}(\kappa_{2i})>\kappa_{2i+2}$.
\end{clm*}

\begin{proof} We will take $i=0$ for simplicity, but the general case is
similar. We have
\[
E^\Tt_{\gamma_0}\rest\nu(E^\Tt_{\gamma_0})=G_{\varrho_1}=E^\Uu_{\beta_1}
\rest\varrho(E^\Uu_{\beta_1}), \]
and $\varrho_2=\varrho(E^\Uu_{\beta_1})<\nu(E^\Uu_{\beta_1})$. Let
$\zeta<_\Uu\beta_1$ be least such that
$\varrho_2\leq\crit(i^\Uu_{\zeta\beta_1})$.
Then $\gamma_0\geq\zeta$ because otherwise there is some extender used along
$[0,\zeta]_\Uu$ which sends its critical point $>\varrho_2$, but then
$\varrho(E^\Uu_{\beta_1})>\varrho_2$, a contradiction.
And $\varrho_2=\nu(F_1(M^\Uu_\zeta))$ and
\[ F_1(M^\Uu_\zeta)\rest\varrho_2=G_{\varrho_1}. \]
So $E^\Tt_{\gamma_0}$ and $F_1(M^\Uu_\zeta)$ are essentially the same extender.
Therefore
\[ M^\Uu_\zeta||\OR(M^\Uu_\zeta) =
\Ult(M^\Tt_{\gamma_0},E^\Tt_{\gamma_0})||\OR(M^\Uu_\zeta), \]
which implies that $E^\Uu_\xi=\emptyset$ for all $\xi\in[\zeta,\gamma_0)$,
so $M^\Uu_{\gamma_0}=M^\Uu_\zeta$, and implies that either
\[ \lh(F_1(M^\Uu_\zeta))=\lh(E^\Tt_{\gamma_0})\ \ \&\ \
F_1(M^\Uu_\zeta)=E^\Tt_{\gamma_0}\ \ \&\ \
E^\Uu_{\gamma_0}=F_0(M^\Uu_{\zeta}) \]
or
\[ \lh(F_1(M^\Uu_{\zeta}))>\lh(E^\Tt_{\gamma_0})\ \ \&\
\ E^\Uu_{\gamma_0}=\emptyset=E^\Tt_{\gamma_0+1}\ \ \&\
\ E^\Uu_{\gamma_0+1}=F_0(M^\Uu_\zeta). \]
Note that, in any case, letting
$F=F_1(M^\Uu_\zeta)$, we have
\[ i_{G_{\varrho_1}}(\kappa_1)=i_F(\kappa_1)>\delta(M^\Uu_\zeta)\geq\kappa_2. \]
The fact that $i_F(\kappa_1)>\delta(M^\Uu_\zeta)$ is because $F$ is not of
superstrong type, which is because $F_1^P$ is not of superstrong type (this is
easily preserved under iteration). The last inequality is because $\kappa_2=\crit(E^\Uu_{\delta_1})$ and $M^{*\Uu}_{\delta_1+1}=M^\Uu_\zeta$.

But $i_{F_0}(\kappa_0)>i^{*\Tt}_{\gamma_0+1}(\kappa_1)=i_F(\kappa_1)$. This completes the proof of the claim (in the case that $i=0$).\end{proof}

This completes Case \ref{case:varrho<nu}, and
by symmetry, the
proof.\qedhere\end{case}\end{proof}

Using the lemma, one can establish  ``Uniqueness of the next extender'':

\begin{proof}[Theorem \ref{thm:uniquemaxLEcon}, Proof Sketch]
Suppose not.
By \cite[\S9, \S12]{fsit}, $B=(N_\lambda,E,F)$ is a
mixed bicephalus. But $B$ is $(0,\om_1+1)$-iterable, by the proof in
\cite[\S 12]{fsit}. But the iterability of $B$ contradicts Lemma \ref{lem:bicephalus}.

(At the request of the referee, we provide a sketch of the iterability proof.
It is almost the same as the case for a bicephalus with two active type 2 extenders.
Given a tree $\Tt$ on $B$ and its lift $\Uu$ on $R$, and given $\alpha<\lh(\Tt)$ such that $[0,\alpha]_\Tt$ does not drop,
write \[ M^\Tt_\alpha=M_\alpha=(N_\alpha,F_{0\alpha},F_{1\alpha}),\]
\[ B_\alpha=i^\Uu_{0\alpha}(B)=(A_\alpha,E_{0\alpha},E_{1\alpha}).\]
We will have then a lifting map
\[ \pi_\alpha:M_\alpha\to B_\alpha,\]
with
\[ \pi_\alpha\com i^\Tt_{0\alpha}=i^\Uu_{0\alpha}\rest B,\]
and in particular, $\pi_\alpha(\delta(M_\alpha))=\delta(B_\alpha)$.
Moreover,
\[ \pi_\alpha:(N_\alpha,F_{i\alpha})\to(A_\alpha,E_{i\alpha}) \]
is a weak $0$-embedding for both $i=0$ and $i=1$ (that is, $\Sigma_0$-elementary, and $\Sigma_1$-elementary
on an $\in$-cofinal set). So $\pi_\alpha$ lifts extenders from $\es^{N_\alpha}$
to $\es^{B_\alpha}$, and lifts $F_{i\alpha}$ to $E_{i\alpha}$.
Moreover, in the case that $E^\Tt_\alpha=F_{1\alpha}$,
because we use the exchange ordinal
$\nu^\Tt_\alpha=\delta(M^\Tt_\alpha)$
and since $i^\Uu_{0\alpha}$ is fully elementary, we have
\[ \pi_\alpha(\nu^\Tt_\alpha)=\delta(B_\alpha)=\nu(E_{1\alpha}), \]
so in this case we have the usual sort of correspondence of exchange ordinals $\nu^\Tt_\alpha$ and $\nu^\Uu_\alpha$.)
\end{proof}

\bibliographystyle{plain}
\bibliography{biblio}

\end{document}